\documentclass[11pt,a4paper]{article}
\usepackage[english]{babel}
\usepackage{amsmath,amsthm,amssymb,graphicx,xcolor,bm} %, hyperref ,mathtools}
 \allowdisplaybreaks
  
\newtheorem{theorem}{Theorem}
\newtheorem{lemma}{Lemma}
\newtheorem{corollary}{Corollary}
\newtheorem{proposition}{Proposition}
\theoremstyle{definition}
\newtheorem{remark}{Remark}
\begin{document}

%\begin{frontmatter}
%\title{On coverage and local radial rates \\ of DDM-credible sets}
%\runtitle{Local radial rates of DDM-credible sets}
%\begin{aug}
%\author{\fnms{Eduard} \snm{Belitser} %\thanksref{m1}
%\ead[label=e1]{e.n.belitser@vu.nl}}
%\address{Department of Mathematics \\ %Faculty of Sciences, 
%VU University Amsterdam\\  \printead{e1}}
%\runauthor{E. Belitser}
%\affiliation{VU University Amsterdam} %\thanksmark{m1}}
%\end{aug}

\title{On coverage and local radial rates of DDM-credible sets}
\author{Eduard Belitser \\ {\it VU University Amsterdam}}
\date{}%\today}  
\maketitle

\renewcommand{\abstractname}{\vspace{-\baselineskip}}
\begin{abstract}
For a general statistical model, we introduce the notion of 
\emph{data dependent measure} (DDM)  on the model parameter.
Typical examples of DDM are the posterior distributions. %with respect to priors on the parameter.  
Like for posteriors, the quality of a DDM is characterized by the
contraction rate which we allow to be local, i.e., depending on the parameter.  
We construct confidence sets as \emph{DDM-credible sets} and 
address the issue of optimality of such sets, via a trade-off between 
its ``size'' (the \emph{local radial rate}) and its coverage probability.  
In the mildly ill-posed inverse signal-in-white-noise model, we construct 
a DDM as empirical Bayes posterior with respect to a certain prior, and 
define its (default) credible set.  Then we introduce \emph{excessive bias restriction} 
(EBR), more general than \emph{self-similarity} and \emph{polished tail condition} 
recently studied in the literature. Under EBR, we establish 
the confidence optimality of our credible set with some local 
(\emph{oracle}) radial rate.  %which is the best over certain family of local rates. 
We also derive the oracle estimation inequality and the oracle DDM-contraction 
rate, non-asymptotically and uniformly in $\ell_2$. %When applied appropriately, 
The obtained local results are more  powerful than global:  adaptive 
minimax results %(estimation, DDM-contraction, confidence inference) 
for a number of smoothness scales follow as consequence, in particular, 
the ones considered by Szab\' o et al.\ (2015) \cite{Szabo&etal:2015}. 
\end{abstract}

%\begin{keyword}[class=MSC]
%\kwd[Primary ]{62G15}
%\kwd{62C05}
%\kwd[; secondary ]{62G99}
%\end{keyword}
%\begin{keyword} %\kwd{data dependent measure}
%\kwd{DDM-credible ball}
%\kwd{excessive bias restriction}
%\kwd{local radial rate}
%\end{keyword}
%\end{frontmatter}

\vspace{-20pt} 
\footnote{\vspace{-\baselineskip}\\
\textit{MSC2010 subject classification:}
primary 62G15, 62C05; secondary 62G99. }

\section{Introduction}

Suppose we observe  a random element  $X^{(\varepsilon)}   \sim 
\mathrm{P}_0^{(\varepsilon)} \in \mathcal{P}^{(\varepsilon)}$,
$X^{(\varepsilon)}\in\mathcal{X}^{(\varepsilon)}$ 
for some measurable space
$(\mathcal{X}^{(\varepsilon)},\mathcal{A}^{(\varepsilon)})$, where
$\mathcal{A}^{(\varepsilon)}$ is a $\sigma$-algebra
on $\mathcal{X}^{(\varepsilon)}$. 
In fact, we consider a sequence of observation models parametrized by $\varepsilon>0$.
Parameter $\varepsilon$ is assumed to be known, it reflects in some sense 
the influx of information in the data $X^{(\varepsilon)}$ as
$\varepsilon \to 0$.
For instance, $\varepsilon$ can be the variance of an additive noise, or
$\varepsilon=n^{-1/2}$, where $n$ is the sample size.
To avoid overloaded notations, we will often drop the dependence on
$\varepsilon$;  for example, $X=X^{(\varepsilon)}$ etc.  

Let $\mathcal {P}=\big\{\mathrm{P}_{\theta}, \, \theta \in\Theta\big\}$ 
and  $\mathrm{P}_0 = \mathrm{P}_{\theta_0}$, where
$\theta_0 \in\Theta \subseteq \mathcal{L}$ is an unknown parameter of
interest belonging to some subset $\Theta$ of a linear 
space $\mathcal{L}$ equipped with a (semi-)metric 
$d(\cdot,\cdot): \, \mathcal{L} \times \mathcal{L} \to \mathbb{R}_+=[0,+\infty)$. 
From now on, when we deal with probabilities of events in terms of the data 
$X\sim \mathrm{P}_\theta$, we write $\mathrm{P}_\theta$.
By $\theta_0$ we denote the so called ``true'' value of the 
parameter $\theta$ to distinguish it from the variable $\theta\in \Theta$. 

The aim is to construct an optimal (to be defined later) confidence set for the parameter 
$\theta_0\in \Theta$ on the basis of observation $X \sim \mathrm{P}_0 \in \mathcal{P}$,
with a prescribed coverage probability.
The convention throughout this paper is that we measure the size of a set 
by the smallest possible radius of a ball containing that set.
It is thus sufficient to consider only confidence balls as confidence sets.
Let $B(l_0,r) =\{l\in \mathcal{L}: \, d(l_0,l) \le r\}$ be the ball in space $\mathcal{L}$ 
with center $l_0\in \mathcal{L}$ and radius $r\ge 0$.
Denote by $\mathcal{B}_\mathcal{L}$ the corresponding Borel
$\sigma$-algebra on $\mathcal{L}$ and by $\mathcal{B}_\mathbb{R}$ the usual Borel
$\sigma$-algebra on $\mathbb{R}$.
A general confidence ball for the parameter $\theta$ is of the form
$B(\tilde{\theta},  \tilde{r})= \{\theta \in \mathcal{L}: \, d(\theta,
\tilde{\theta}) \le \tilde{r}\}$, 
with some \emph{data dependent center} (DD-center)
$\tilde{\theta}=\tilde{\theta}(X)=\tilde{\theta}(X,\varepsilon)$, 
$\tilde{\theta}: \, \mathcal{X} \to \Theta$, 
and some \emph{data dependent radius} (DD-radius)
$\tilde{r}=\tilde{r}(X)=\tilde{r}(X,\varepsilon)$, 
$\tilde{r}: \, \mathcal{X} \to \mathbb{R}_+ = \{a\in \mathbb{R}: \,a\ge 0\}$. 
The quantities $\tilde{\theta}$ and $\tilde{r}$ are 
$(\mathcal{A},\mathcal{B}_\mathbb{R})$-measurable functions of the data.
%such functions of the data $X$ that $d(\theta, \tilde{\theta})$ is
%$(\mathcal{B}_\mathcal{L}\times \mathcal{B}_\mathcal{L},\mathcal{B}_\mathbb{R})$-measurable 
%and $\tilde{r}$ is $(\mathcal{A},\mathcal{B}_\mathbb{R})$-measurable.

Suppose we are given a \emph{data dependent measure} (DDM)
$\mathrm{P}(\cdot|X)$ on $\Theta$ (we will say: a DDM on the parameter $\theta$). 
In order to settle the measurability issue for the rest of the paper, 
by DDM we will always mean a measurable probability measure in the sense that 
for  all $x \in \mathcal{X}$ the quantity $\mathrm{P}(\cdot|X=x)$ is a probability measure on 
$(\mathcal{B}_{\mathcal{L}},\Theta)$ (can be relaxed to $\mathrm{P}_{\theta_0}$-almost all 
$x \in \mathcal{X}$, for all $\theta_0\in\Theta$) and 
$\mathrm{P}(B|X)$ is  $\mathcal{A}$-measurable for each $B \in \mathcal{B}_{\mathcal{L}}$. 
Typically, a DDM is obtained  by using a Bayesian approach, 
as the resulting posterior (or empirical Bayes posterior) distribution  
with respect to some prior on $\Theta$, see Supplement for more details 
on how Bayes approach yields DDM's. We slightly abuse the traditional 
notation $\mathrm{P}(\cdot|X)$ because in general a DDM does not have 
to be a conditional distribution. Notice that empirical Bayes posteriors are, 
strictly speaking, not conditional distributions either.
Other examples considered in the literature that fall under the category 
of DDMs are (generalized) fiducial distributions and bootstrap.   

%The DDM approach has a conceptual advantage
%that, no matter how complex the model and the setting of the problem is, 
A ``good'' DDM can be used for all kind of inference: e.g., 
estimation, construction of confidence sets (nowadays termed as \emph{uncertainty quantification}), testing.
As to confidence sets, given a DDM $\mathrm{P}(\cdot|X)$ on $\Theta$,
we can take a \emph{DDM-credible set} %(also exploited in the Bayesian literature) 
$C_\alpha(X)$ of level $\alpha \in [0,1]$, i.e., $\mathrm{P}(\theta\in C_\alpha(X)|X) \ge \alpha$, 
as a candidate confidence set. 
In this paper we focus on the following, for now loosely formulated, question:
\begin{center}\emph{When does DDM-credibility lead to confidence?}\end{center}
%The properties of a confidence set $C_\alpha(X)$ are studied from the 
%$\mathrm{P}_{\theta_0}$-perspective, i.e., $X \sim \mathrm{P}$ for some 
%``true'' $\mathrm{P} \in \mathcal{P}$.

Let us specify the optimality framework for confidence sets. 
We would like to construct  such a confidence ball $B(\hat{\theta}, C\hat{r})$
that for any $\alpha_1,\alpha_2\in (0,1]$ and some functional 
%$\mathnormal{r}(\mathrm{P})$ 
%$R_\varepsilon(\mathrm{P})$ 
%$\r_\varepsilon(\mathrm{P})$ 
%$r_\varepsilon(\mathrm{P})$ 
%$\mathnormal{r}_\varepsilon(\mathrm{P})$ 
%$\mathit{r}_\varepsilon(\mathrm{P})$ 
%$\mathrm{r}_\varepsilon(\mathrm{P})$
%$\mathsf{r}_\varepsilon(\mathrm{P})$ 
%$\mathfrak{r}_\varepsilon(\mathrm{P})$ 
%$\rho_\varepsilon(\mathrm{P})$
$\mathnormal{r}(\theta)=r_\varepsilon(\theta)$, 
$\mathnormal{r}_\varepsilon:\, \Theta\to \mathbb{R}_+$,
there exists  $C, c>0$ such that 
for all $\varepsilon \in (0,\varepsilon_0]$ with some $\varepsilon_0>0$,
\begin{align}
\label{conf_ball_problem}
\sup_{\theta\in\Theta_{cov}} \mathrm{P}_\theta \big(\theta\not\in B(\hat{\theta},C\hat{r}) \big) 
\le \alpha_1, \quad 
\sup_{\theta\in\Theta_{size}} \mathrm{P}_\theta 
\big( \hat{r} \ge c \mathnormal{r}_\varepsilon(\theta) \big) 
\le \alpha_2,
\end{align}
where $\Theta_{cov}, \Theta_{size} \subseteq \Theta$.
In some papers, a confidence set satisfying the first 
relation in (\ref{conf_ball_problem}) is called
\emph{honest} over $\Theta_{cov}$. 
%We will avoid this unfortunate terminology and say instead
%that $B(\hat{\theta},C\hat{r})$ has coverage $1-\alpha_1$ 
%uniformly over $\Theta_0$.
The quantity $\mathnormal{r}_\varepsilon(\theta)$
has the meaning of the effective radius of the confidence 
ball $B(\hat{\theta},C\hat{r})$.  
We call the quantity $\mathnormal{r}_\varepsilon(\theta)$ 
\emph{radial rate}. 
Clearly, there are many possible radial rates, but it is desirable to find 
the ``fastest'' (i.e., smallest) radial rate  $\mathnormal{r}_\varepsilon(\theta)$, 
for which  the relations (\ref{conf_ball_problem}) hold for ``massive''  
 $\Theta_{cov},\Theta_{size} \subseteq \Theta$, ideally for
$\Theta_{cov} =\Theta_{size}= \Theta$. 
The two relations in (\ref{conf_ball_problem}) are called \emph{coverage} and \emph{size properties}.
Asymptotic formulation is also possible: $\limsup_{\varepsilon \to 0}$ 
should be taken, constants $\alpha_1$, $\alpha_2$, $C$, $c$ (possibly sets $\Theta_{cov},
\Theta_{size}$) can be allowed to  depend on $\varepsilon$. 

Thus the following optimality aspects are involved in the framework 
(\ref{conf_ball_problem}): the coverage, the radial rate, 
and the uniformity subsets $\Theta_{cov}, \Theta_{size}$. The optimality is basically 
a trade-off between these complementary aspects pushed to the utmost limits,
when further improving upon one aspect leads to a deterioration in another aspect.
For example, the smaller the local radial rate $\mathnormal{r}_\varepsilon(\theta)$ 
in (\ref{conf_ball_problem}), the better. But if  it is too small, the size requirement 
in (\ref{conf_ball_problem}) may hold uniformly only over some 
``thin''  set $\Theta_{size}\subset\Theta$. %instead of the whole $\Theta$. % in (\ref{conf_ball_problem}).
On the other hand, if one insists on $\Theta_{cov}=\Theta_{size}=\Theta$,
then it may be impossible to establish (\ref{conf_ball_problem}) for interesting 
(relatively small) radial rates $\mathnormal{r}_\varepsilon(\theta)$.

One approach to optimality is via minimax estimation framework.
It is assumed that 
$\theta \in \Theta_\beta\subseteq\Theta$ for some ``smoothness'' parameter 
$\beta \in \mathcal{B}$, which  may be known or unknown (non-adaptive 
or adaptive formulation)�. The key notion here is the so called 
\emph{minimax rate} $R_\varepsilon(\Theta_\beta)$, see Supplement. 
%(associated with some loss function).
%e.g., $R_\varepsilon^2(\Theta_\beta)=\inf_{\hat{\theta}}\sup_{\theta\in \Theta_\beta} 
%\mathrm{E}_\theta \|\hat{\theta}-\theta\|^2$)
%in the problem of estimating $\theta\in\Theta_\beta$. 
%In many concrete situations a minimax estimator $\hat{\theta}$ and the minimax 
%rate $R_\varepsilon(\Theta_\beta)$ are derived. %$R_\varepsilon(\Theta_\beta)$ is then known. 
%According to the minimax framework,  
The radial rate is taken to be
$\mathnormal{r}_\varepsilon(\theta)=R_\varepsilon(\Theta_\beta)$, 
which is a global quantity as it is constant for all $\theta\in\Theta_\beta$.  
%By using lower bounds from the minimax estimation theory, 
In the nonadaptive case, it can be shown that the minimax rate $R_\varepsilon(\Theta_\beta)$ 
is the \emph{best global radial rate} (i.e., among all radial rates that are constant 
on $\Theta_\beta$); see Supplement for more details. %cf.\  Robins and van der Vaart (2006). 

An adaptation problem arises when, for a given family of models 
$\{\Theta_\beta, \, \beta\in\mathcal{B}\}$ (called \emph{scale}), 
we only know that $\theta\in  \Theta_\beta$ for some unknown %parameter 
$\beta\in\mathcal{B}$. In fact, $\theta\in\cup_{\beta\in\mathcal{B}} \Theta_\beta \subseteq\Theta$ 
and the problem becomes in general more difficult. 
For a $\Theta'_{cov}\subseteq \Theta$, we want to construct such 
a confidence ball $B(\hat{\theta},C\hat{r})$ that 
\begin{equation}
\label{adapt_conf_ball_problem}
\sup_{\theta\in\Theta'_{cov}} \mathrm{P}_\theta 
\big(\theta\not\in B(\hat{\theta},C\hat{r}) \big) 
\le \alpha_1, \;\;
\sup_{\theta\in\Theta_\beta} \mathrm{P}_\theta \big( \hat{r} 
\ge c R_\varepsilon(\Theta_\beta) \big) 
\le \alpha_2 \;\, \forall\beta\in\mathcal{B},
\end{equation}
possibly in asymptotic setting: put $\limsup_{\varepsilon \to 0}$ 
in front of both $\sup$ in (\ref{adapt_conf_ball_problem}).
Ideally, $\mathcal{B}$ is ``massive'' and $\Theta'_{cov} \supseteq\Theta_\beta$.
However, in general it is impossible to construct  optimal (fully) adaptive confidence set 
in the minimax sense: the coverage requirement in (\ref{adapt_conf_ball_problem}) does not hold 
even for $\Theta'_{cov}= \Theta_\beta$. For the classical many normal 
means model, there are negative results in \cite{Li:1989}, 
\cite{Baraud:2004}, \cite{Cai&Low:2006}; this is also discussed in \cite{Robins&vanderVaart:2006}.
A way to achieve adaptivity is  to remove the so called \emph{deceptive parameters}
(in \cite{Szabo&etal:2015} they are called \emph{inconvenient truths}) from $\Theta$, i.e.,
consider a strictly smaller set  $\Theta'_{cov}\subset \Theta$.
%Some limited amount of adaptivity can be achieved
%either by imposing some more structure on the set $\mathcal{B}$ 
%(e.g., $\mathcal{B}=[\beta_0,2 \beta_0]$ for some known $\beta_0>0$)   or 
%by taking a smaller set $\Theta'_{cov}\subset \Theta$, the one with certain 
%``troublemakers'' (also called \emph{deceptive parameters}, 
%Szab\'o et al.\  (2015) call them \emph{inconvenient truths}) removed from $\Theta$.
Examples are: $\Theta'_{cov}=\Theta_{ss}$, the so called  
\emph{self-similar} parameters (related to Sobolev/Besov scales)
introduced in \cite{Picard&Tribouley:2000} and later studied in \cite{Bull:2012}, 
\cite{Bull&Nickl:2013},  \cite{Szabo&etal:2015},  %van der Vaart and van Zanten (2015), 
\cite{Nickl&Szabo:2014}, \cite{Serra&Krivobokova:2014};
and $\Theta'_{cov}=\Theta_{pt}$, a more general class 
of \emph{polished tail} parameters introduced in
\cite{Szabo&etal:2015}.  %, van der Vaart and van Zanten (2015). 
More literature on adaptive minimax confidence sets: \cite{Low:1997}, 
\cite{Beran&Dumbgen:1998}, \cite{Picard&Tribouley:2000},  
\cite{Juditsky&Lambert-Lacroix:2003}, \cite{Genovese&Wasserman:2008},  
\cite{Gine&Nickl:2010}, \cite{Hoffmann&Nickl:2011}, 
\cite{Bull:2012}, \cite{Bull&Nickl:2013}, \cite{Nickl&Szabo:2014},
\cite{Szabo&etal:2015, Szabo&etal:2015b}.

In all the above mentioned papers  global minimax radial rates 
$R_\varepsilon(\Theta_\beta)$ (as in (\ref{adapt_conf_ball_problem})) were studied. 
In this paper we allow local radial rates as in the framework (\ref{conf_ball_problem}). 
When applied appropriately, the local approach is actually more powerful and flexible. 
Namely, suppose that a local radial rate $\mathnormal{r}_\varepsilon(\theta)$ 
is such that, for some uniform $c>0$,
\begin{align}
\label{local_impl_global}
\mathnormal{r}_\varepsilon(\theta) 
\le c R_\varepsilon(\Theta_\beta), \quad \text{for all} \;\; \theta \in \Theta_\beta, \;
\beta\in\mathcal{B}.
\end{align}
%In this case we say  $\mathnormal{r}_\varepsilon(\theta)$ 
%\emph{covers} the  scale $\{\Theta_\beta, \, \beta\in\mathcal{B}\}$.
If in addition $\Theta'_{cov}\subseteq \Theta_{cov}$ and 
$\Theta_\beta \subseteq \Theta_{size}$ for all $\beta\in\mathcal{B}$, 
then the results of type (\ref{conf_ball_problem}) imply the results 
of type (\ref{adapt_conf_ball_problem}), %this holds 
\emph{simultaneously for all scales} $\{\Theta_\beta, \, \beta\in\mathcal{B}\}$
for which  (\ref{local_impl_global}) is satisfied. %(in fact, for all sets for which (\ref{local_impl_global}) holds).
We say that the local radial rate $\mathnormal{r}_\varepsilon(\theta)$ \emph{covers} these scales;
more details are in Supplement.
%This makes local results (\ref{conf_ball_problem}) more attractive 
%than the global framework (\ref{adapt_conf_ball_problem}). However, 
%local (nontrivial) results %(with interesting local rates that cover some typical smoothness scales) 
%seem to be more difficult to derive, we are not aware of any paper dealing with this. %in the present context.  
%%in the context of uncertainty quantification. This paper is the first attempt in this direction.

%%If one local radial rate is uniformly smaller than the other, preserving 
%%$\Theta'_0=\Theta$ in (\ref{conf_ball_problem}) at the same time, 
%%we say it is \emph{stronger}. 
%It is desirable to obtain results with  local radial rates that as small as possible, 
%%(thus covering as many functional scales as possible), 
%but still preserving $\Theta_{size}=\Theta$ in (\ref{conf_ball_problem}). %at the same time.
%%In this paper, we always require $\Theta'_0=\Theta$. 
%
%%(e.g., uniformly smaller than the global minimax rates over some scales), 
%%The meaning of `oracle' local radial rate is discussed later, 
%%for now: roughly, the one that is the best (i.e., minimal) over 
%%a certain family of local rates, so that the resulting oracle rate
%%is uniformly smaller than global minimax rates over some scales.  
%%implying optimality in the minimax sense over some scales.

In Section \ref{sec_general_appr} we consider a general  setting and 
present two types of conditions on a DDM $\mathrm{P}(\cdot|X)$:  
the upper and lower bounds on the DDM-contraction rate
in terms of a given local radial rate $\mathnormal{r}_\varepsilon(\theta_0)$.
%with which the DDM $\mathrm{P}(\cdot|X)$ concentrates around the 
%true $\theta_0$ and a DD-center $\hat{\theta}$ from the $\mathrm{P}_{\theta_0}$-perspective. 
Roughly speaking, the upper bound condition means that the DDM $\mathrm{P}(\cdot|X)$ contracts
at $\theta_0$ with the local rate at least $\mathnormal{r}_\varepsilon(\theta_0)$,
from the  $\mathrm{P}_{\theta_0}$-perspective;
then one can also construct a DD-center $\tilde{\theta}$ which is 
an estimator of $\theta_0$ with the rate $\mathnormal{r}_\varepsilon(\theta_0)$.  
The lower bound condition means that the DDM concentrates 
around the DD-center $\tilde{\theta}$ at a rate that is not faster than 
$\mathnormal{r}_\varepsilon(\theta_0)$.
We show that the upper bound condition allows to control 
the size of the $\mathrm{P}(\cdot|X)$-credible ball, whereas
the lower bound  is in some sense the minimal condition for providing its 
sufficient $\mathrm{P}_{\theta_0}$-coverage.
%thus establishing the relations (\ref{conf_ball_problem}) for 
%appropriate choices of the involved quantities.
%For example, the Bernstein-von Mises property ensures 
%(being stronger than needed) the asymptotic versions of such conditions, 
%leading to the well know fact that under BvM-property credible sets are 
%asymptotically confidence sets. 

%Actually, for the lower bound, the local approach is the right one 
%as the lower bound cannot be established for the minimax rate 
%as it is too big for the lower bound to hold with this rate. 

In Section \ref{section_gsf} we consider the %canonical 
mildly ill-posed inverse signal-in-white-noise model %Gaussian sequence model
and implement the general approach of Section \ref{sec_general_appr}.
We construct a DDM $\mathrm{P}(\cdot|X)$, 
which is %on the one hand the DDM-mixture over family of DDMs, 
%and on the other hand it can be seen as 
in fact the empirical Bayes posterior resulting from a certain two-level hierarchical prior.  
For the proposed DDM, we first prove the upper bound type result. 
Namely, we establish that the DDM  $\mathrm{P}(\cdot|X)$  contracts, 
from the $\mathrm{P}_{\theta_0}$-perspective,  to $\theta_0$  with 
the local rate $\mathnormal{r}_\varepsilon(\theta_0)$, 
which is the best (fastest) contraction rate  over some family of DDMs
(therefore also called \emph{oracle rate}).
The DDM contraction result is non-asymptotic and uniform in $\theta_0\in \ell_2$.
%a DD-center $\tilde{\theta}$ and
%a $\mathrm{P}(\cdot|X)$-credible ball around $\tilde{\theta}$
%(or alternatively, a \emph{default credible ball}).
%For the proposed DDM, we prove the upper bound type conditions 
%with the local radial rate $\mathnormal{r}_\varepsilon(\theta_0)$
%uniformly over the whole set $\Theta=\ell_2$. 
%These are of interest on its own and roughly mean the following:
%uniformly in $\theta_0\in \ell_2$, the DDM  $\mathrm{P}(\cdot|X)$  contracts, 
%from the $\mathrm{P}_{\theta_0}$-perspective, at $\theta_0$  with 
%the local rate $\mathnormal{r}_\varepsilon(\theta_0)$;  
%the DD-center $\tilde{\theta}$ constructed by using the DDM $\mathrm{P}(\cdot|X)$ 
%converges to $\theta_0$ with the rate $\mathnormal{r}_\varepsilon(\theta_0)$. 
The local radial rate $\mathnormal{r}_\varepsilon(\theta_0)$  
satisfies (\ref{local_impl_global}) for typical smoothness scales 
such as  Sobolev and analytic ellipsoids, Sobolev hyperrectangles, 
tail classes, certain scales of Besov classes and %(very rich) families
$\ell_p$-bodies. 
This means that we obtained, as consequence of our local result, the
adaptive minimax contraction rate results over all these scales for the DDM $\mathrm{P}(\cdot|X)$.  
An accompanying result is that, by using the DDM $\mathrm{P}(\cdot|X)$, 
a DD-center $\tilde{\theta}$ can be constructed that converges to $\theta_0$ 
also with the local rate $\mathnormal{r}_\varepsilon(\theta_0)$, thus also yielding 
the panorama of the minimax adaptive estimation results over all these scales simultaneously. 
%This oracle estimation result is in the same spirit as the oracle results for the risk hull minimization 
%method developed by Cavalier and Golubev (2006).
%(the RHM oracle rate contains an additional extra penalty term, 
%but the multiplicative factor is rather tight as compared to ours.

%Besides, some interesting corollaries are obtained about the ``one-sided'' inference on 
%the ``smoothness'' of $\theta_0$, which can be used for testing problems.

Although the upper bound results are of interest on its own, our main purpose is 
to construct an optimal  (according to the framework (\ref{conf_ball_problem})) confidence set. 
To this end, the established upper bound results imply the 
size relation for a $\mathrm{P}(\cdot|X)$-credible ball in (\ref{conf_ball_problem}) with the 
local radial rate $\mathnormal{r}_\varepsilon(\theta_0)$, uniformly over $\Theta_{size}=\ell_2$. 
%In a way, the DDM $\mathrm{P}(\cdot|X)$ summarizes the data very efficiently,
%it ``has it all'':  contracts to the true $\theta_0$ with 
%the oracle local rate, can be used for estimation 
%and constructing confidence sets as $\mathrm{P}(\cdot|X)$-credible set, 
%with strong optimality properties such as simultaneous adaptive minimaxity 
%over a number of scales of nonparametric classes. It can also be used for testing, 
%but we do not study this problem. 
For the coverage relation in (\ref{conf_ball_problem}) to hold, we also need the lower bound results.
It turns out that the lower bound result can be established uniformly only over some 
$\Theta_{cov} \subset\ell_2$, which forms an actual  restriction.  
This is in accordance with the above mentioned fact that it is impossible 
to construct optimal (fully) adaptive confidence set in the minimax sense.
We propose a set  $\Theta_{cov}=\Theta_{eb}$ of (non-deceptive) 
parameters satisfying the so called \emph{excessive bias restriction} 
and derive the lower bound uniformly over this set. 
Combining the obtained upper and lower bounds, we establish the optimality 
(\ref{conf_ball_problem}) of  a (default) DDM-credible 
ball with $\Theta_{cov}=\Theta_{eb}$, $\Theta_{size}=\ell_2$ 
and the local radial rate $\mathnormal{r}_\varepsilon(\theta_0)$. 
The class $\Theta_{eb}$ is more general than the earlier mentioned  self-similar  and 
polished tail parameters, 
namely, $\Theta_{ss}\subseteq \Theta_{pt}\subseteq\Theta_{eb}$.
Moreover, the established (local) optimality (\ref{conf_ball_problem}) 
%with the local radial rate $\mathnormal{r}_\varepsilon(\theta_0)$
implies the global optimality (\ref{adapt_conf_ball_problem}) in the sense of adaptive 
minimaxity over all scales for which (\ref{local_impl_global}) is fulfilled, in particular 
for the ones considered by Szab\' o et al.\ (2015).
In this paper, we primarily  interested in non-asymptotic assertions, 
asymptotic versions can be readily obtained. 
Section \ref{sec_proofs} contains the proofs of the main results. 
The elaboration on some points  
and some background  information related to the paper 
are provided in Supplement.

\section{General DDM-based construction of confidence ball}
\label{sec_general_appr}

\subsection{DDM-credible ball}

Suppose we are given a DDM $\mathrm{P}(\cdot|X)$ on $\theta$.
The goal of this section is to construct a confidence set by using this DDM
and to elaborate on its coverage and size.
Recall that our optimality framework is (\ref{conf_ball_problem}), with 
a local radial rate $\mathnormal{r}(\theta_0)=\mathnormal{r}_\varepsilon(\theta_0)$. 
%We allow local radial rates, the advantageous features 
%of local radial rates (as compared to global ones) 
%are already discussed in Introduction and will also be discussed below 
%in Section \ref{section_gsf} at length for the  projection oracle rate in the Gaussian sequence framework. 
In this section we are not concerned with specific choices for radial rates 
and simply  suppose that we are given some local radial rate 
$\mathnormal{r}(\theta_0)$. 
As $X=X^{(\varepsilon)}$, %DD-center $\hat{\theta}= \hat{\theta}(X)$ and 
the DDM $\mathrm{P}(\cdot|X)$ depends on $\varepsilon$. 
Hence, so do all the DDM-based quantities. In this section we omit this dependence 
completely to ease the notations.
The convention for the rest of this section is that all assumptions and claims hold 
for all $\varepsilon \in [0,\varepsilon_0]$ with some $\varepsilon_0>0$.

First we present the general construction of a confidence ball by using 
the DDM $\mathrm{P}(\theta|X)$ and a DD-center $\hat{\theta}= \hat{\theta}(X)$. 
For a $\kappa \in(0,1)$, define the DD-radius  
\begin{equation}
\label{radius}
\hat{r}_\kappa= \hat{r}(\kappa,X,\hat{\theta}) = 
\inf\big\{r:\, \mathrm{P}\big(d(\theta,\hat{\theta}) \le r |X\big) \ge 1-\kappa
\big\}
\end{equation}
and then, for an $M>0$, construct the confidence ball 
\begin{equation}
\label{conf_ball}
B(\hat{\theta},  M\hat{r}_\kappa)= \big\{\theta \in \Theta: \,
d(\theta,
\hat{\theta}) \le M \hat{r}_\kappa\big\}.
\end{equation}
For $M=1$, (\ref{conf_ball}) is the smallest DDM-credible ball around $\hat{\theta}$ of  
level $1-\kappa$. For a good DDM that concentrates around $\theta_0$ from 
the $\mathrm{P}_{\theta_0}$-perspective (i.e., under 
$X \sim \mathrm{P}_{\theta_0}$), a DDM-credible set should 
also be a good confidence set, but its $\mathrm{P}_{\theta_0}$-coverage
is in general lower than $1-\kappa$ because of uncertainty in the data.
The multiplicative factor $M$, not dependent on $\varepsilon$,
is intended to inflate the DDM-credible ball of level $1-\kappa$ to account for this uncertainty.
%intended to trade-off the size of the ball against its  coverage probability.

Now we construct a confidence ball by using only the given DDM $\mathrm{P}(\cdot|X)$, 
without a predetermined DD-center. For a $p\in (1/2,1)$, define first
\begin{equation}
\label{radius*}
\hat{r}^*=\hat{r}^*(p)=\inf\big\{r:\, \mathrm{P}(d(\theta,\theta')
\le r |  X)\ge p \; \mbox{ for some } \theta' \in \Theta \big\}.
\end{equation}
This is the smallest possible radius of DDM-credible ball of level $p$.
Next, for some $\varsigma >0$, take any (measurable function of data $X$)
$\check{\theta}\in \Theta$ that satisfies
\begin{equation}
\label{hat_theta}
\mathrm{P}(\theta: \, d(\theta, \check{\theta})\le  (1+\varsigma) \hat{r}^* |X)\ge p.
\end{equation}
We call the constructed $\check{\theta}=\check{\theta}(p,\varsigma)$
\emph{default DD-center}, with respect to the DDM $\mathrm{P}(\cdot|X)$.
In words, $\check{\theta}=\check{\theta}(p,\varsigma)$ is the center of the
ball of nearly the smallest radius subject to the constraint that
its DDM $\mathrm{P}(\cdot|X)$-mass is at least $p$.

Finally, define the \emph{default DDM-credible ball}:
for a $\kappa \in (0,1)$,  %with respect to the DDM $\mathrm{P}(\cdot|X)$.
\begin{equation}
\label{def_ball}
\tilde{B}=\tilde{B}_M=\tilde{B}_{M,\kappa}=B(\check{\theta},  M\hat{r}_\kappa),
\end{equation}
where $B(\check{\theta},  M\hat{r}_\kappa)$ is
defined by (\ref{radius}) and (\ref{conf_ball}), and  $\check{\theta}$ is defined
by (\ref{radius*}) and (\ref{hat_theta}).

\subsection{Conditions}
Here we present some conditions used later for establishing general statements about 
the coverage and the size of the confidence ball (\ref{conf_ball}) (and  (\ref{def_ball})).
For $\theta_0\in\Theta$,  $M, \delta \ge 0$, some local radial rate 
$\mathnormal{r}(\theta_0)$,  some DDM $\mathrm{P}(\cdot|X)$ 
and DD-center $\hat{\theta}=\hat{\theta}(X)$, introduce the following conditions.
\begin{itemize}

\item[(A1)]  For  some   
$\phi_1(M)=\phi_1(M,\varepsilon,\theta_0, \hat{\theta}) \ge 0$,  
such that $\phi_1(M)\downarrow 0$ as $M \uparrow\infty$, 
\[
\mathrm{E}_{\theta_0} \big[\mathrm{P}(d(\theta,\hat{\theta}) \ge M
\mathnormal{r}(\theta_0)|X) \big] \le 
\phi_1(M).
\]
 
\item[(A2)] For some $\psi(\delta)=\psi(\delta,\varepsilon,\theta_0,\hat{\theta})\ge 0$
such that  $\psi(\delta)\downarrow 0$ as $\delta \downarrow 0$,
\[
\mathrm{E}_{\theta_0} \big[\mathrm{P}(d(\theta,\hat{\theta}) \le \delta
\mathnormal{r}(\theta_0)|X) \big] \le 
\psi(\delta).
\]

\item[(A3)]  For  some   
$\phi_2(M)=\phi_2(M,\varepsilon,\theta_0,\hat{\theta}) \ge 0$ such that 
$\phi_2(M)\downarrow 0$ as $M \uparrow \infty$, 
\[
\mathrm{P}_{\theta_0} \big(d(\theta_0,\hat{\theta}) \ge M
\mathnormal{r}(\theta_0) \big) 
\le \phi_2(M).
\]
\end{itemize}

%Of course, 
Conditions (A1)--(A3) trivially hold for the functions
%\begin{align*}
$\phi_1(M,\varepsilon,\theta_0, \hat{\theta}) = 
\mathrm{E}_{\theta_0} \big[\mathrm{P}(d(\theta,\hat{\theta}) 
\ge M\mathnormal{r}(\theta_0)|X) \big]$,  
$\psi(\delta,\varepsilon,\theta_0, \hat{\theta})=
\mathrm{E}_{\theta_0} \big[\mathrm{P}(d(\theta,\hat{\theta}) 
\le \delta \mathnormal{r}(\theta_0)|X) \big]$, 
$\phi_2(M,\varepsilon,\theta_0, \hat{\theta})=
\mathrm{P}_{\theta_0} \big(d(\theta_0,\hat{\theta}) \ge M
\mathnormal{r}(\theta_0) \big).$
%\end{align*}
Conditions (A1)--(A3) become really useful when the functions $\phi_1, \psi,\phi_2$ 
do not depend on $\varepsilon \in (0,\varepsilon_0]$ 
and  $\theta_0\in \Theta_0$, for some $\varepsilon_0>0$ and 
$\Theta_0 \subseteq \Theta$ (preferably $\Theta_0=\Theta$).
Then  (A1) means that $\mathrm{P}(\cdot|X)$ concentrates, from the 
$\mathrm{P}_{\theta_0}$-perspective, around
$\hat{\theta}$ with the radial rate at least $\mathnormal{r}(\theta_0)$,  
(A2) means that $\mathrm{P}(\cdot|X)$ concentrates around
$\hat{\theta}$ with the radial rate at most 
$\mathnormal{r}(\theta_0)$.
%(not too much concentration of $\mathrm{P}(\cdot|X)$-mass around $\hat{\theta}$)
%(in a way, no too fast ``leakage'' of $\mathrm{P}(\cdot|X)$-mass through $\hat{\theta}$). 
Condition (A3) means that the DD-center $\hat{\theta}$ is 
an estimator of $\theta_0$ with the rate $\mathnormal{r}(\theta_0)$.
Together (A1) and (A2)  imply that $\mathrm{P}(\cdot|X)$ concentrates, from the 
$\mathrm{P}_{\theta_0}$-perspective, on the spherical shell 
$\{\theta: \, \delta \mathnormal{r}(\theta_0) \le d(\theta, \hat{\theta}) \le M\mathnormal{r}(\theta_0)\}$
for sufficiently small $\delta$ and large $M$.

Condition (A1) is reminiscent of the definition of the so called (global) 
posterior contraction rate $R_\varepsilon(\Theta)$ from 
the nonparametric Bayes literature:  
$\Pi( d(\theta_0,\theta) \ge M R_\varepsilon(\Theta)|X)$ should 
be small for sufficiently large $M$ from the $\mathrm{P}_{\theta_0}$-probability perspective. 
The following introduces a counterpart of a local contraction rate for a general 
DDM $\mathrm{P}(\cdot|X)$. 
\begin{itemize}
\item[(\~A1)]  For  some   
$\varphi(M)=\varphi(M,\varepsilon,\theta_0) \ge 0$ such that 
$\varphi(M)\downarrow 0$ as $M \uparrow \infty$,
\[
\mathrm{E}_{\theta_0} \big[\mathrm{P}(d(\theta_0,\theta) \ge M
\mathnormal{r}(\theta_0)|X) \big] \le 
\varphi(M).
\]
\end{itemize}
Clearly, condition (A1) is implied by conditions  (\~A1) and (A3) for the function
$\phi_1(M)= \phi_2(a M) + \varphi((1-a)M)$ with any $a\in (0,1)$. 
%We could therefore impose (\~A1), (A2) and (A3) instead of (A1)--(A3).
 
Introduce a strengthened version of  condition (A2).
\begin{itemize}
\item[(\~A2)]  
For  some   
$\psi(\delta)=\psi(\delta,\varepsilon,\theta_0) \ge 0$ such that 
$\psi(\delta)\downarrow 0$ as $\delta \downarrow 0$
and any DD-center $\tilde{\theta}=\tilde{\theta}(X)$,
$
\mathrm{E}_{\theta_0} \big[\mathrm{P}(d(\theta,\tilde{\theta}) \le \delta
\mathnormal{r}(\theta_0)|X) \big] \le 
\psi(\delta).
$
\end{itemize}
The difference between $\psi$ from (A2) and $\psi$ from (\~A2) is that 
the latter does not depend on the DD-center. 
We keep however the same notation for the function $\psi$ in
(\~A2) as in (A2) without confusion as we are never going to 
use both conditions simultaneously.

Instead of non-asymptotic conditions, even in (regular) parametric 
models one typically verifies asymptotic versions. %instead. 
In Supplement we introduce asymptotic (as $\varepsilon \to 0$) versions 
of conditions (A1)--(A3), (\~A1)--(\~A2) denoted as (AA1)--(AA3) 
and (A\~A1)--(A\~A2). % respectively.
%Then one can %An interested reader should be able to reproduce 
The asymptotic versions of all the assertions below can be reproduced by using 
(AA1)--(AA3) instead of (A1)--(A3). 
%\end{remark}
More remarks about the conditions are in Supplement.

\subsection{Conditions for default confidence ball}
\label{susec_default_ball}

The following proposition claims that condition (\~A1)  implies
conditions (A1) and (A3) for the default DD-center $\check{\theta}$  
defined by (\ref{radius*})--(\ref{hat_theta}), with appropriate choices of %functions 
$\phi_1$ and $\phi_2$.
Hence,  (\~A1)--(\~A2) imply (\~A1) and (A2) which in  turn 
imply (A1)--(A3) for %the default DD-center 
$\check{\theta}$.

\begin{proposition}
\label{prop1}
Let condition (\~A1) be fulfilled with 
function $\varphi(M)$ and let the default DD-center $\check{\theta}$ be 
defined by (\ref{radius*}) and (\ref{hat_theta}).
Then condition (A1) holds with function 
$\phi_1(M)=\varphi(aM/(2+\varsigma))/p + \varphi((1-a)M)$ for any $a\in
(0,1)$,
and condition (A3) holds with function $\phi_2(M) =
\varphi(M/(2+\varsigma))/p$.
\end{proposition}
\begin{proof}
If (A3) holds true with $\phi_2(M) = \varphi(M/(2+\varsigma))/p$, then,
by using this and (\~A1), we obtain that, for any $a \in (0,1)$, 
\begin{align*}
\mathrm{E}_{\theta_0} \big[\mathrm{P}(d(\theta,\check{\theta}) \ge M
\mathnormal{r}(\theta_0)|X) \big] \le &
\mathrm{E}_{\theta_0} \big[\mathrm{P}(d(\theta,\theta_0) \ge a M
\mathnormal{r}(\theta_0)|X) \big]  \\
& +
\mathrm{E}_{\theta_0} \big[\mathrm{P}(d(\theta_0,\check{\theta}) \ge(1- a) M
\mathnormal{r}(\theta_0)|X) \big],
\end{align*}
which implies  (A1) with $\phi_1(M)=  
\varphi(a M/(2+\varsigma))/p + \varphi((1-a)M)$.

Therefore, it remains to show (A3) with the function 
$\phi_2(M) = \varphi(M/(2+\varsigma))/p$.
From  (\~A1) it follows by the Markov inequality  that 
\[
\mathrm{P}_{\theta_0}\big(\mathrm{P}(\theta \in B(\theta_0,M
\mathnormal{r}(\theta_0)|X)
\ge p\big)\ge 1-\tfrac{\varphi(M)}{p}.
\] 
By (\ref{hat_theta}), the ball 
$B(\check{\theta},(1+\varsigma)\hat{r}^*)$ 
has $\mathrm{P}(\cdot|X)$-probability at  least $p$.
If the ball
$B(\theta_0,M \mathnormal{r}(\theta_0))$ also has 
$\mathrm{P}(\cdot|X)$-probability at least $p$ (which happens with
$\mathrm{P}_{\theta_0}$-probability 
at least $1-\frac{\varphi(M)}{p}$), then, firstly,  
$\hat{r}^*\leq M \mathnormal{r}(\theta_0) $ by virtue of  
the definition (\ref{radius*}) of $\hat{r}^*$, and, secondly,  
the balls $B(\check{\theta},(1+\varsigma)\hat{r}^*)$ and
$B(\theta_0,M \mathnormal{r}(\theta_0))$
must intersect, otherwise the total
$\mathrm{P}(\cdot|X)$-mass would exceed $2p>1$. 
Hence, by the triangle inequality,
$
d(\theta_0,\check{\theta}) \le
(1+\varsigma)\hat{r}^*+M\mathnormal{r}(\theta_0) 
\le (2+\varsigma)M\mathnormal{r}(\theta_0),
$
with $\mathrm{P}_{\theta_0}$-probability at least $1-\varphi(M)/p$.
Hence, condition (A3) holds with $\phi_2(M) = \varphi(M/(2+\varsigma))/p$ 
for the default DD-center $\check{\theta}$.
%defined by (\ref{radius*}) and (\ref{hat_theta}). 
\end{proof}

\begin{remark}
\label{rem_def_ball}
Of course,  $\hat{r}^*$  depends on $p$ and $\check{\theta}$ depends on
both $p$ and $\varsigma$. We however skip this dependences from the
notations by assuming from now on that $p=2/3$ and $\varsigma=1/2$. We %also 
take $a=1/2$ in Proposition \ref{prop1}. According to Proposition \ref{prop1}, 
if condition (\~A1) is fulfilled with function $\varphi(M)$, then conditions (A1) and (A3) hold for the 
default DD-center $\check{\theta}$, with the functions $\phi_1(M)= 3\varphi(M/5)/2 + \varphi(M/2)$ 
and $\phi_2(M) = 3\varphi(2M/5)/2$ respectively.
\end{remark}

\subsection{Coverage and size of the DDM-credible set}
\label{subsec_cov_size}

Recall that our main goal is to construct a confidence ball satisfying the  
optimality framework (\ref{conf_ball_problem}).
In this subsection we present some simple general (coverage and size) properties  
of the DDM-credible ball $B(\hat{\theta},M\hat{r}_\kappa)$ 
defined by (\ref{conf_ball}) with a DDM $\mathrm{P}(\cdot|X)$ and 
a DD-center $\hat{\theta}$ satisfying (A1)--(A3). 
Next we briefly outline how these properties can be used
to establish the optimality framework (\ref{conf_ball_problem}) in concrete 
settings. 
%Finally, we provide some insight about
%the minimality of conditions (A1)--(A3) for the framework 
%(\ref{conf_ball_problem}).
 
The following proposition gives an upper bound for the 
coverage probability of the confidence ball (\ref{conf_ball}).
\begin{proposition} 
\label{prop2a}
For a $\theta_0\in\Theta$ and 
some radial rate $\mathnormal{r}(\theta_0)$, let 
$\kappa \in(0,1)$ and the ball $B(\hat{\theta},M\hat{r}_\kappa)$  
be defined by (\ref{conf_ball}) with
a DDM  $\mathrm{P}(\cdot|X)$ and a DD-center $\hat{\theta}$ 
satisfying conditions (A2) and (A3).
Then for any $M, \delta>0$,
\[
\mathrm{P}_{\theta_0} \big(\theta_0 \not \in
B(\hat{\theta},M\hat{r}_\kappa)\big)
=\mathrm{P}_{\theta_0} \big(d(\theta_0,\hat{\theta}) >M\hat{r}_\kappa\big)\le
\phi_2(M \delta)+ \frac{\psi(\delta)}{1-\kappa}.
\]
\end{proposition}
\begin{proof}
By the Markov inequality,  (\ref{radius})  and conditions (A2) and (A3), we derive
\begin{align*}
&\mathrm{P}_{\theta_0}\big(d(\theta_0,\hat{\theta})>M\hat{r}_\kappa\big)
%\\&
\le
\mathrm{P}_{\theta_0} \big(d(\theta_0,\hat{\theta})>M\hat{r}_\kappa,  
\hat{r}_\kappa\ge \delta \mathnormal{r}(\theta_0)\big) + 
\mathrm{P}_{\theta_0} \big(\hat{r}_\kappa < \delta \mathnormal{r}(\theta_0)\big) 
\\&
\le
\mathrm{P}_{\theta_0} \big(d(\theta_0,\hat{\theta})>M \delta
\mathnormal{r}(\theta_0)\big) +
\mathrm{P}_{\theta_0} \big(\mathrm{P}(d(\theta,\hat{\theta}) \le 
\delta \mathnormal{r}(\theta_0)|X)\ge 1-\kappa\big)\\
&\le  \phi_2(M\delta)+
\frac{\mathrm{E}_{\theta_0}(P(d(\theta,\hat{\theta}) \le 
\delta \mathnormal{r}(\theta_0)|X)}{1-\kappa}\le
\phi_2(M\delta)+ \frac{\psi(\delta)}{1-\kappa}. \qedhere
\end{align*}
\end{proof}
It  is not difficult to see that (A2) guarantees that the rate
$\mathnormal{r}(\theta_0)$ is actually sharp. Indeed,
as is already derived in the proof of Proposition \ref{prop2a},
\begin{align*}
\mathrm{P}_{\theta_0} \big(\hat{r}_\kappa \le \delta
\mathnormal{r}(\theta_0)\big)
%&\le\mathrm{P}_{\theta_0}\big(\mathrm{P}(d(\theta,\hat{\theta})\le
%\delta \mathnormal{r}(\theta_0)|X) \ge 1-\kappa\big)\\
%&\le\frac{\mathrm{E}_{\theta_0}\big(\mathrm{P}(d(\theta,\hat{\theta})\le
%\delta\mathnormal{r}(\theta_0)|X) \big)}{1-\kappa}
\le
\frac{\psi(\delta)}{1-\kappa}.
\end{align*}
%\end{remark}

The following assertion gives some bound on the effective size  
of $B(\hat{\theta},M\hat{r}_\kappa)$  in terms of the local radial rate
$\mathnormal{r}(\theta_0)$ from the
$\mathrm{P}_{\theta_0}$-perspective.

\begin{proposition} 
\label{prop2b}
For a $\theta_0\in\Theta$, 
let a DDM  $\mathrm{P}(\cdot|X)$ and a DD-center $\hat{\theta}$ satisfy %condition 
(A1) for some radial rate
$\mathnormal{r}(\theta_0)$.
Let $\hat{r}_\kappa$  be defined by (\ref{radius}).
Then  for any $\kappa\in (0,1)$, $M>0$, 
\[
\mathrm{P}_{\theta_0}\big(\hat{r}_\kappa \ge M
\mathnormal{r}(\theta_0)\big)
\le \frac{\phi_1(M)}{\kappa}.
\]
\end{proposition}

\begin{proof} By the conditional Markov inequality, (\ref{radius})  
and condition (A1),  
\begin{align*}
\mathrm{P}_{\theta_0} \big(\hat{r}_\kappa \ge M
\mathnormal{r}(\theta_0)\big)
&\le
\mathrm{P}_{\theta_0}\big(\mathrm{P}(d(\theta,\hat{\theta})\le
M\mathnormal{r}(\theta_0)|X) \le 1-\kappa\big)\\
&=
\mathrm{P}_{\theta_0}\big(\mathrm{P}(d(\theta,\hat{\theta})>
M\mathnormal{r}(\theta_0)|X)>\kappa\big) \\&
\le
\frac{\mathrm{E}_{\theta_0}\big(\mathrm{P}(d(\theta,\hat{\theta})>
M\mathnormal{r}(\theta_0)|X) \big)}{\kappa}\le
\frac{\phi_1(M)}{\kappa}. \qedhere
\end{align*}
\end{proof}
%\begin{remark}

Suppose conditions (A1)--(A3) are fulfilled for some DDM $\mathrm{P}(\cdot|X)$ and DD-center $\hat{\theta}$,
with some local radial rate $r(\theta_0)$ and functions $\phi_1,\psi, \phi_2$.
Let us elucidate what else is needed in concrete situations to derive 
the optimality framework (\ref{conf_ball_problem}).
%In general the functions $\phi_1$, $\psi$ and $\phi_2$  depend 
%on the unknown $\theta_0$ (and $\varepsilon$), so that
%we cannot apply the above assertions directly for establishing (\ref{conf_ball_problem}).
Suppose the following uniform bounds hold: 
\[
\phi_1(M,\varepsilon,\theta_0) \le %\tilde{\phi}_1(M,\varepsilon)\le 
\bar{\phi}_1(M) \quad \forall \theta_0\in\Theta_{size}\subseteq \Theta,
\]
%for all $\theta_0\in\Theta_{size}\subseteq \Theta$,  
\[
\phi_2(M,\varepsilon,\theta_0) \le %\tilde{\phi}_2(M,\varepsilon)\le
\bar{\phi}_2(M), \;
\psi(M,\varepsilon,\theta_0) \le %\tilde{\psi}(M,\varepsilon)\le
\bar{\psi}(M), \quad \forall  \theta_0\in\Theta_{cov}\subseteq \Theta,
\]  
for all $\varepsilon\in(0,\varepsilon_0]$,
where $\bar{\phi}_1(M)\downarrow 0$, $\bar{\phi}_2(M) \downarrow 0$ as 
$M \uparrow \infty$ and $\bar{\psi}(\delta) \downarrow 0$ as $\delta \downarrow 0$.
Clearly, then Propositions \ref{prop2a} and \ref{prop2b} ensure (\ref{conf_ball_problem}) 
for the ball $B(\hat{\theta},M\hat{r}_\kappa)$ and the radial rate 
$\mathnormal{r}(\theta_0)$ by taking sufficiently large $M$. 
In fact, we can optimize the choice of $M$  as follows: first determine
\[
\min_{\delta>0}\Big\{\bar{\phi}_2(M\delta)+ \frac{\bar{\psi}(\delta)}{1-\kappa}
\Big\} =
\bar{\phi}(M,\kappa), 
\]
where $\bar{\phi}(M,\kappa)\downarrow 0$ as $M\uparrow\infty$.
Next, take constants $M_1$ and $M_2$ sufficiently large so that
$\bar{\phi}(M_1,\kappa) \le \alpha_1$ and $\bar{\phi}_1(M_2)/\kappa\le \alpha_2$.
Then the optimality framework (\ref{conf_ball_problem}) holds with $C=M_1$ and $c=M_2$.

Finally, let us mention some additional material provided in Supplement.
\begin{itemize}
\item[-] Two examples, the normal model and the so called 
\emph{Bernstein-von Mises} case, demonstrating the application of Propositions 
\ref{prop2a} and  \ref{prop2b}.
\item[-] A corollary from Propositions \ref{prop1}--\ref{prop2b} 
for the default confidence ball $\tilde{B}_{M,\kappa}$ defined by  
(\ref{def_ball}), which can also be used for establishing 
the optimality framework (\ref{conf_ball_problem}).
\item[-] A proposition, %\ref{prop4} 
demonstrating that (A2) is in some sense the minimal condition for providing a
sufficient $\mathrm{P}_{\theta_0}$-coverage of the $\mathrm{P}(\cdot|X)$-credible ball
with the sharpest rate.
\end{itemize}

%In Supplement we consider two examples, the normal model and the so called 
%Bernstein-von Mises case, demonstrating the application of Propositions 
%\ref{prop2a} and  \ref{prop2b}.
%
%In Supplement we establish a corollary from Propositions \ref{prop1},  
%\ref{prop2a} and \ref{prop2b} for the default confidence ball $\tilde{B}_{M,\kappa}$ (\ref{def_ball}),
%under (\~A1)--(\~A2) which can also be used for establishing 
%the optimality framework (\ref{conf_ball_problem}).
%%This corollary can be used for establishing the optimality framework (\ref{conf_ball_problem})
%%in the same way as above, provided the functions $\varphi$ and  $\psi$ from conditions 
%%(\~A1) and (\~A2) are bounded uniformly 
%%over appropriate sets $\Theta_{cov}$ and $\Theta_{size}$.
%
%Finally, we give a proposition %\ref{prop4} 
%in Supplement, demonstrating that (A2) is in some sense the minimal condition for providing a
%sufficient $\mathrm{P}_{\theta_0}$-coverage of the $\mathrm{P}(\cdot|X)$-credible ball
%with the sharpest rate.

\section{Inverse signal-in-white-noise model}
\label{section_gsf}

\subsection{The model}

Let $\mathbb{N}=\{1,2,\ldots\}$ and 
$\sigma=(\sigma_i, \, i \in \mathbb{N})$ be a positive 
nondecreasing sequence. We observe 
\begin{align}
\label{model}
X=X^{(\varepsilon)}=
(X_i, \, i \in \mathbb{N})\sim 
\mathrm{P}_\theta=\mathrm{P}_\theta^{(\varepsilon)}=
\bigotimes_{i\in\mathbb{N}} N(\theta_i,\sigma_i^2), \quad 
\sigma_i^2 = \varepsilon^2 \kappa_i^2,
\end{align}
i.e., $X_i\stackrel{\rm \tiny ind}{\sim}  N(\theta_i,\sigma_i^2)$, 
$i\in\mathbb{N}$. Here $\theta=(\theta_i,\, i\in\mathbb{N}) \in\Theta=\ell_2$ 
is an unknown parameter of interest.
Without loss of generality, we set 
\[
\varepsilon^2= \min_i\sigma_i^2 
= \sigma_1^2 \quad  \text{and} \quad \kappa_i= \sigma_i/\varepsilon \ge 1, 
\quad \text{so that} \quad \sigma_i^2 = \varepsilon^2 \kappa_i^2.
\]
Thus, the nondecreasing sequence $\{\kappa_i^2, \, i \in \mathbb{N}\}$ reflects 
the ill-posedness of the model and $\varepsilon^2$ is the noise intensity 
describing the information increase 
in the data $X^{(\varepsilon)}$ as $\varepsilon\to 0$.
The model (\ref{model}) is known to be the sequence version of 
the \emph{inverse signal-in-white-noise model}. 
There is now a vast literature about this model,  
especially for the direct case: $\kappa_i^2 =1$, $i\in\mathbb{N}$. 
This model is of a canonical type and serves, by virtue of the so called 
\emph{equivalence principle}, as a purified approximation to some 
other statistical models.
The direct case of the model (\ref{model}) can be related, in exact terms, to 
the generalized linear Gaussian model as introduced by \cite{Birge&Massart:2001},  
the continuous white noise model, certain discrete regression model; 
and as an approximating model, to the density estimation problem, 
spectral function estimation, various regression models.
Examples of inverse problems fitting the framework (\ref{model})
can be found in \cite{Cavalier:2008}; see further references therein. 
The statistical inference results for the generic model (\ref{model}) can %in principle 
be conveyed to other models, according to the equivalence principle.
However, in general the problem of establishing the equivalence 
in a precise sense is a delicate task. We will not go into this, 
but focus on the model (\ref{model}). Some more information can be found  in Supplement.

By default, all summations and products are over $\mathbb{N}$, 
unless otherwise specified, e.g., $\bigotimes_i =\bigotimes_{i\in\mathbb{N}}$.
Introduce some notations: $\|\theta\| = (\sum_i \theta_i^2)^{1/2}$ 
is the $\ell_2$-norm; for $a,b \in\mathbb{R}$, $\lfloor a \rfloor = \max\{z \in \mathbb{Z}:\, z \le a\}$, 
$\Sigma(a)= \sum_{i\le a} \sigma^2_i$, $a\vee b=\max\{a,b\}$, $a\wedge b = \min\{a,b\}$;
%$\mathbb{N}_k= \{1,\ldots, k\}$ for $k\in\mathbb{N}$,
%$S_1 \setminus S_2 = \{s \in S_1:\, s \not\in S_2\}$ for sets $S_1,S_2$,  
%$\mathbb{N}_k^c = \mathbb{N} \setminus \mathbb{N}_k$, 
$\varphi(x,\mu,\sigma^2)$ is the $N(\mu,\sigma^2)$-density at $x$,
$N(\mu,0)$ means a Dirac measure at $\mu$; %, i.e., $\mathrm{P}(Z=\mu)=1$.
the indicator function $\mathrm{1}\{E\}=1$ if the event $E$ occurs and is zero otherwise.
Let $\sum_{i=k}^{n} a_i =0 $ if $n<k$.
If random quantities appear in a relation, then this relation should be
understood in $\mathrm{P}_{\theta_0}$-almost sure sense,
for the ``true'' $\theta_0 \in \Theta$.

We complete this subsection with conditions on $\sigma_i^2$'s (or, equivalently, on $\kappa_i^2$'s):
for any $\rho, \tau_0 \ge1$, $\gamma>0$,  there exist some positive $K_1$, $K_2=K_2(\rho)$, 
$K_3=K_3(\gamma)$, $K_4 \in (0,1)$, $\tau>2$ (this can be relaxed to $\tau \ge 1$) 
and $K_5=K_5(\tau_0)$ such that  the relations
 \begin{align}
&(i)\; n\sigma_n^2 \le  K_1 \Sigma(n), \;\;
(ii)\; \Sigma(\rho n)\le  K_2(\rho) \Sigma(n), \; \notag\\ 
&(iii)\; \sum\nolimits_n \!\! e^{-\gamma n} \Sigma(n) \le K_3(\gamma) \sigma_1^2,
\label{condition_sigma}\\
&(iv)\;
%\Sigma(m) - \Sigma(\lfloor m/\tau \rfloor ) \ge K_4\Sigma(m), \;
\Sigma(\lfloor m/\tau \rfloor) \le (1-K_4)\Sigma(m), \;\;
%\sum_{i=\lfloor \frac{m}{\tau} \rfloor+1}^m \!\!\! \sigma_i^2\ge K_4\Sigma(m), \quad
(v)\;
l \sigma^2_{\lfloor l/\tau_0\rfloor} \ge 
%K_5(\tau_0) (\Sigma(l) -\Sigma(\lfloor l/\tau_0 \rfloor)),
K_5(\tau_0) \sum_{i=\lfloor l/\tau_0 \rfloor+1}^l \sigma_i^2, & \notag
\end{align}
hold for all $n\in\mathbb{N}$, all $m \ge \tau$ and all $l \ge \tau_0$, 
Although there is in principle some freedom in choosing sequence 
$\kappa_i$ describing the ill-posedness of the problem, 
to avoid unnecessary technical complications, 
from now on we assume the so called \emph{mildly ill-posed} case:
$\kappa_i^2 = i^{2p}$, $i \in \mathbb{N}$, for some $p\ge 0$.
\begin{remark}
\label{rem_cond_sigma}
The  mildly ill-posed case $\kappa_i^2 = i^{2p}$
satisfies (\ref{condition_sigma}) with $K_1=2p+1$, 
$K_2 = (\rho+1)^{2p+1}$, $K_3 =\frac{4(8p+4)^{2p}}{(e\gamma)^{2p+1}(e^{\gamma/2}-1)}$
(a rough bound), $K_4=\frac{1}{2}$, $\tau$ can be any number satisfying 
 $\tau \ge 2^{1+1/(2p+1)}$ and $K_5 = (2\tau_0)^{-2p}$;
see Supplement for the calculations.
\end{remark}

\subsection{Constructing DDM $\mathrm{P}(\theta|X)$ as empirical Bayes posterior}
\label{construction_DDM}
 
Here we construct a DDM $\mathrm{P}(\cdot|X)$ on $\theta$ 
which we later use for constructing a confidence set as DDM-credible ball, 
according to the general approach
described in Section \ref{sec_general_appr}. 
The optimality (\ref{conf_ball_problem})
will then be established for appropriate choices of involved quantities. 

For some fixed $K,\alpha>0$, introduce  the following 
(mixture) DDM on $\theta$:
\begin{equation}
\label{ddm1}
\mathrm{P}(\cdot|X)=\mathrm{P}_{K,\alpha}(\cdot|X)
=\sum\nolimits_I\mathrm{P}_I(\cdot|X)\mathrm{P}(\mathcal{I}=I|X),
\end{equation}
where the family of DDMs 
%$\prescript{DD}{}M$ %$\prescript{I}{I}\! X_i ~_IX$ $\mathcal{P}(\mathbb{N},X)=
$\{\mathrm{P}_I(\cdot|X), \, I\in \mathbb{N}\}$ 
on $\theta$  and the DDM $\mathrm{P}(\mathcal{I}=I|X)$ on $I$  are 
%defined as follows: %by (\ref{measure_P_I}) and (\ref{P(I|x)}) respectively.
%For some $L>0$ define a family of  DDM's on $\theta$ as follows:
\begin{align}
\label{measure_P_I}
\mathrm{P}_I(\cdot|X) &=\bigotimes\nolimits_i N\big(X_i(I), 
L\sigma^2_i\mathrm{1}\{i\le I\}\big),   %\quad I\in \mathbb{N}, 
\\
\label{p(I|X)}
\mathrm{P}(\mathcal{I}=I|X)&=
\frac{\lambda_I \bigotimes_i  \varphi(X_i,X_i(I),\tau_i^2(I)+\sigma^2_i)}
{\sum_J  \lambda_J \bigotimes_i  \varphi(X_i,X_i(J),\tau_i^2(J)+\sigma^2_i)},
%\quad I\in\mathbb{N}.
\end{align}
with $L=\frac{K}{K+1}$ (can be any positive value), $C_\alpha=e^\alpha-1$, and
\begin{align}
\label{tau_i} 
X_i(I) = X_i\mathrm{1}\{i\le I\}, \;
\tau_i^2(I) =%\tau_i^2(I,\varepsilon,K_1,K_2)=    
K\varepsilon^2 \mathrm{1}\{i\le I\},\; 
\lambda_I=C_\alpha e^{-\alpha I}, \; 
i,I\in \mathbb{N},
\end{align}
so that $\sum_I\lambda_I =1$. 
The quantity (\ref{p(I|X)}) exists as $\mathrm{P}_{\theta_0}$-almost sure limit of 
%(as $n\to \infty$) 
\begin{align*}
\mathrm{P}_n(\mathcal{I}&=I|X) =
\frac{\lambda_I \bigotimes_{i=1}^n  \varphi(X_i,X_i(I),\tau_i^2(I)+\sigma^2_i)}
{\sum_J  \lambda_J \bigotimes_{i=1}^n  \varphi(X_i,X_i(J),\tau_i^2(J)+\sigma^2_i)}.
%\\&= \frac{e^{-\alpha I} (K+1)^{-I/2} 
%\exp\big\{ -\sum_{i=I+1}^n \frac{X_i^2}{2\varepsilon^2}\big\}} 
%{\sum_J e^{-\alpha J}  (K+1)^{-(J\wedge n)/2} 
%\exp\big\{-\sum_{i=J+1}^n \frac{X_i^2}{2\varepsilon^2}\big\}}\\
%&=\bigg[f_I(X_1,\ldots, X_I)+\sum_{J=I+1}^\infty 
%%%e^{-[(J\wedge n)-I](\alpha+\frac{1}{2}\log(K+1))}
%\exp\Big\{\sum_{i=I+1}^{J\wedge n} \Big(\frac{X_i^2}{2\varepsilon^2} 
%-\alpha-\frac{1}{2}\log(K+1)\Big)\Big\}\bigg]^{-1},
\end{align*}
%with some function $f_I$.
%since $\mathrm{P}_n(\mathcal{I}=I|X)\in[0,1]$ and is decreasing in $n> I$ almost surely.  

The DDM (\ref{ddm1}) can be associated with the
empirical Bayes posterior originating from the following two-level 
hierarchical prior $\Pi$: %$\theta \sim \Pi$ means that 
\begin{equation}
\label{prior_pi}
\theta|(\mathcal{I}=I)\sim \Pi_{I,\mu(I)}=\bigotimes_i N(\mu_i(I),\tau^2_i(I)), 
\;\;\mathrm{P}(\mathcal{I}=I)=\lambda_I, %\; I\in\mathbb{N}, 
\end{equation}
where $\tau^2_i(I)$ and $\lambda_I$ are defined by (\ref{tau_i}),  
$\mu(I)=(\mu_i(I), \, i \in \mathbb{N})$ 
with $\mu_i(I) = \mu_{I,i} \mathrm{1}\{i\le I\}$.
%\begin{equation}
%\label{tau_i_} 
%\mu_i(I) = \mu_i \mathrm{1}\{i\le I\},\quad
%\tau_i^2(I) =%\tau_i^2(I,\varepsilon,K_1,K_2)=    
%K\varepsilon^2 \mathrm{1}\{i\le I\},\quad 
%\lambda_I=C_\alpha e^{-\alpha I}, \quad 
%i,I\in \mathbb{N}. %\quad \text{for some} \quad K>0.
%\end{equation}
Indeed, the model (\ref{model}) and the prior (\ref{prior_pi}) 
lead to the corresponding marginal $\mathrm{P}_{X,\mu}(X)$
%\[
%\mathrm{P}_{X,\mu}(X)=\sum_I \lambda_I \mathrm{P}_{X,I,\mu(I)}(X)=
%\sum_I \lambda_I \bigotimes_i 
%\varphi(X_i,\mu_i(I), \tau^2(I) +\varepsilon^2)
%\] 
and the posterior 
$\Pi_{\mu}(\cdot|X)=\sum_I\Pi_\mu(\cdot|X,\mathcal{I}=I)\Pi_\mu(\mathcal{I}=I|X)$, 
%where
%\[
%\Pi_\mu(\cdot|X,\mathcal{I}=I)=\bigotimes_i 
%N\Big(\frac{\tau^2_i(I)X_i^2+\varepsilon^2\mu_i(I)}{\varepsilon^2+\tau^2_i(I)},
%\frac{\varepsilon^2\tau^2_i(I)}{\varepsilon^2+\tau^2_i(I)}\Big),
%\]
%\[
%\Pi_\mu(\mathcal{I}=I|X)=
%\frac{\lambda_I \bigotimes_i  \varphi(X_i,\mu_i(I),\tau_i^2(I)+\varepsilon^2)}
%{\sum_J  \lambda_J \bigotimes_i  \varphi(X_i,\mu_i(J),\tau_i^2(J)+\varepsilon^2)},
%\] 
where $\mu=(\mu(I), \, I \in \mathbb{N})$ (mind that $\mu$ is a sequence 
of sequences).  Then 
$\mathrm{P}(\cdot|X) = \Pi_{\hat{\mu}}(\cdot|X)$, where
$L=\frac{K}{K+1}$ %(which we can set without loss of generality) 
and $\hat{\mu}= (\hat{\mu}(I), \, I\in \mathbb{N})$
(with $\hat{\mu}(I)= (\hat{\mu}_i(I), \, i \in \mathbb{N})$ and
$\hat{\mu}_i(I) = X_i \mathrm{1}\{i\le I\}$)
is the empirical Bayes estimator obtained by maximizing the marginal 
$\mathrm{P}_{X,\mu}(X)$ %likelihood 
with respect to $\mu$.
Indeed, as is easy to see, 
$\mathrm{P}_I(\cdot|X) = \Pi_{\hat{\mu}}(\cdot|X,\mathcal{I}=I)$ and
$\mathrm{P}(\mathcal{I}=I|X) = \Pi_{\hat{\mu}}(\mathcal{I}=I|X)$. 
%\begin{remark}

Notice that we actually follow the Bayesian tradition since the obtained DDM $\mathrm{P}(\cdot|X)$, 
defined by (\ref{ddm1}), results from certain empirical Bayes posterior. 
However, in principle we can manipulate with different ingredient in constructing 
DDMs. For example, different choices for $\mathrm{P}_I(\cdot|X)$ and $\mathrm{P}(\mathcal{I}=I|X)$ 
in (\ref{ddm1}) are possible,  not necessarily coming from the (same) Bayesian approach.
For example, some other $\mathrm{P}(\mathcal{I}=I|X)$ in (\ref{ddm1}) will do the job as well, 
another constant $L>0$ is possible, etc. More on this %can be found 
is in Supplement.
%\end{remark

\begin{remark}
\label{rem_ddm2}
One more choice for DDM within Bayesian tradition is the
empirical Bayes posterior with respect to $I$:
\begin{align}
\label{ddm2}
\hat{\mathrm{P}}(\cdot|X)=\mathrm{P}_{\hat{I}}(\cdot|X),\quad \text{with} \quad
\hat{I}=\min\Big\{\arg\!\max_{I\in\mathbb{N}} \mathrm{P}(\mathcal{I}=I|X) \Big\},
%\min\big\{ I \in \mathbb{N}:\,\mathrm{P}(\mathcal{I}=I|X)
%=\max_J \mathrm{P}(\mathcal{I}=J|X)\big\},
\end{align}
where $\mathrm{P}_I(\cdot|X)$ and $\mathrm{P}(\mathcal{I}=I|X)$ 
are defined by respectively (\ref{measure_P_I}) and (\ref{p(I|X)}). 
%The $\arg\!\max$ gives a subset of $\mathbb{N}$ in general, 
%$\hat{I}$ is the smallest element in this set.
All the below claims %(and their proofs) 
about the DDM $\mathrm{P}(\cdot|X)$ defined by (\ref{ddm1}) 
%and the DDM $\mathrm{P}(\mathcal{I}=I|X)$ defined by (\ref{p(I|X)}) 
hold also for the DDM $\hat{\mathrm{P}}(\cdot|X)$
%and the quantity $\hat{I}$ given by (\ref{ddm2}), 
exactly in the same way; see Supplement. % for the explanation. 
A connection of the DDM $\hat{\mathrm{P}}(\cdot|X)$ to penalized estimators 
is also discussed in Supplement. 
\end{remark}

%\subsection{The local DDM-contraction rate as oracle over a family of rates} 

\subsection{Local DDM-contraction rate: upper bound}
\label{ddm_contraction}

First we introduce the local contraction rate for the DDM $\mathrm{P}(\cdot|X)$.
%The idea of construction (\ref{ddm1}) is as follows. 
Notice that  the DDM $\mathrm{P}(\cdot|X)$  
is a random mixture over DDMs $\mathrm{P}_I(\cdot|X)$, $I \in \mathbb{N}$.
From the $\mathrm{P}_{\theta_0}$-perspective,   
each $\mathrm{P}_I(\cdot|X)$ contracts to the true $\theta_0$ with 
the local rate  $\mathnormal{r}(I,\theta_0)$:
%the quadratic local rate
\begin{equation}
\label{local_risks}
\mathnormal{r}^2(I,\theta_0)=\mathnormal{r}_\sigma(I,\theta_0)=
 \sum_{i\le I} \sigma_i^2+
\sum_{i> I} \theta_{0,i}^2,\quad I \in\mathbb{N}.
\end{equation}
Indeed, denoting $X(I)=(X_i\mathrm{1}\{i\le I\}, i\in\mathbb{N})$, 
we evaluate %by the %conditional Markov inequality
\begin{align}
\mathrm{E}_{\theta_0} \mathrm{P}_I(\|\theta-\theta_0\|
&\ge M \mathnormal{r}(I,\theta_0) |X) 
\le \frac{\mathrm{E}_{\theta_0} \big[\|X(I)-\theta_0\|^2+L \sum_{i\le I} \sigma_i^2 \big]}
{M^2\mathnormal{r}^2(I,\theta_0)} \notag\\
\label{cond_Markov}
&=\frac{2\sum_{i\le I} \sigma_i^2+\sum_{i>I}^\infty 
\theta_{0,i}^2 }{M^2\mathnormal{r}^2(I,\theta_0)}
\le \frac{2}{M^2}.
\end{align}
Thus, we have the family of local rates $\mathcal{P}=\mathcal{P}(\mathbb{N}) = 
\{r(I,\theta), \, I \in\mathbb{N}\}$.
For each $\theta\in \ell_2$, there is the best choice 
$I_o=I_o(\theta)=I_o(\theta,\sigma)$  of parameter $I$, 
called \emph{oracle}, corresponding to the smallest possible 
rate $\mathnormal{r}(I_o,\theta)$ called 
the \emph{oracle rate} (over the family $\mathcal{P}$) given by
\begin{eqnarray}
\label{oracle_risk}
\mathnormal{r}^2(\theta)=%\mathnormal{r}^2_\sigma(\theta)=
\mathnormal{r}^2(I_o,\theta)=%\mathnormal{r}^2_\sigma(I_o,\theta)=
\min_{I\in \mathbb{N}} \mathnormal{r}^2(I,\theta)=
\sum_{i\le I_o} \sigma_i^2+
\sum_{i>I_o} \theta_i^2.
\end{eqnarray}  
Notice $\mathnormal{r}^2(\theta) \ge \sigma_1^2= \varepsilon^2$ 
and $I_o(\theta) \ge 1$ for any $\theta\in\ell_2$, because we minimize 
over $\mathbb{N}$. This is not restrictive since if  the minimum is taken over 
$I \in \mathbb{N} \cup \{0\}$, all the results  
will hold only for the oracle rate with an additive penalty term, a multiple of %the parametric rate 
$\varepsilon^2$. This will boil down to the same resulting local rate.
%\end{remark}

The following theorem establishes the local upper bound (\ref{oracle_risk}) for 
the contraction rate of the DDM  $\mathrm{P}(\cdot|X)$ 
defined by (\ref{ddm1}). %uniformly over $\ell_2$.
\begin{theorem}[%Local DDM-contraction rate: 
Upper bound]
\label{th1}
Let the DDM $\mathrm{P}(\cdot|X)$ and the local %(oracle) 
rate $\mathnormal{r}(\theta)$ be defined by (\ref{ddm1})
and (\ref{oracle_risk}) respectively, with $K\ge 1.87$, $\alpha>0$.
Then there exists a constant $C_{or}=C_{or}(K,\alpha)$ such that,
for any $\theta_0\in\ell_2$ and $M>0$,
\[
\mathrm{E}_{\theta_0} \mathrm{P}\big(\|\theta-\theta_0\|
\ge M \mathnormal{r}(\theta_0) \big| X\big) \le \frac{C_{or}}{M^2}.
\]  
\end{theorem}
We provide the proof of this theorem in Section \ref{sec_proofs}.
%\begin{remark}
Although the condition $K\ge 1.87$  %has no special meaning, 
emerges as an artifact of the proof technique, %(imprecise constants in the inequalities), 
in a way it has the same meaning as bounds for the penalty constants 
for the penalized estimators. More on this is in Supplement.
%in Subsection \ref{szabo}. %see Remark \ref{alt_conf_ball}.
%Also, it is possible to improve the  constant $C_{or}$ 
%by optimizing the choices of $K, L, \alpha$ and by using more accurate bounds. 
%%As this is not our prime goal, 
%We however present a succinct proof rather than pursue the 
%most accurate constant.
%\end{remark} 

%Theorem \ref{th1} claims that, from the $\mathrm{P}_{\theta_0}$-perspective, 
%the DDM $\mathrm{P}(\cdot|X)$ 
%contracts to $\theta_0$ with the local rate $\mathnormal{r}(\theta_0)$.

Theorem \ref{th1} establishes a non-asymptotic local upper bound for 
the contraction rate of the DDM (\ref{ddm1}) for the %mildly inverse signal-in-white-noise 
model (\ref{model}), uniformly over $\ell_2$-space.
This ensures
%the first ingredient for establishing the confidence optimality (\ref{conf_ball_problem})
%of the default ball $B(\hat{\theta},M\hat{r}_\kappa)$. Namely, 
the size property in (\ref{conf_ball_problem}) for the default  confidence ball 
 (\ref{def_ball}) by using the DDM (\ref{ddm1}), %defined by (\ref{def_ball}), 
with the radial rate $\mathnormal{r}(\theta_0)$ defined by (\ref{oracle_risk})
and $\Theta_{size} =\ell_2$. 
%This implies the size relation in (\ref{conf_ball_problem}) for the radius of 
%the default ball $B(\hat{\theta},M\hat{r}_\kappa)$
We will come back to this when proving the main result, Theorem \ref{th6}.

Besides being an ingredient for establishing the confidence 
optimality (\ref{conf_ball_problem}), the above theorem is of its own interest.
The results with local contraction rates are intrinsically adaptive
%because if its local nature.
%it claims that, from the $\mathrm{P}_{\theta_0}$-perspective, the DDM $\mathrm{P}(\cdot|X)$ 
%contracts to $\theta_0$ with the local rate $\mathnormal{r}(\theta_0)$.
%This implies local adaptiveness 
in the sense that the contraction rate $\mathnormal{r}(I_o,\theta_0)$ is fast 
for ``smooth'' $\theta_0$'s and  slow for ``rough'' ones. 
This is a stronger and more refined property 
than being globally adaptive. % cf.\  \cite{Babenko&Belitser:2010}. 
Let us elucidate the potential strength of local results. 

To characterize the quality of Bayesian procedures, the notion of posterior 
contraction rate was first introduced and studied in \cite{Ghosal&etal:2000}. 
Clearly, it extends directly to DDMs 
%(recall however that our DDM (\ref{ddm1}) is actually an empirical Bayes posterior), 
and a non-asymptotic version of this notion for DDMs is in fact given by condition (\~A2). 
Typically in the literature, contraction rate is related to the (global) minimax rate 
$R(\Theta_\beta)$ over a certain  set $ \Theta_\beta\ni \theta_0$.
%(For example, for the quadratic loss function 
%$R^2_\varepsilon(\Theta_\beta) = \inf_{\tilde{\theta}}\sup_{\theta\in\Theta_\beta}
%\mathrm{E}\|\tilde{\theta}-\theta\|^2$, where the infimum is taken over all possible estimators 
%$\hat{\theta}=\hat{\theta}(X)$, measurable functions of the data $X$.) 
The optimality of Bayesian procedures is then understood in the sense of adaptive minimax
posterior convergence rate: given a prior (knowledge of $\beta$ is 
not used in the prior), the resulting posterior contracts, 
from the $\mathrm{P}_{\theta_0}$-perspective, to the ``true'' $\theta_0\in\Theta_\beta$ 
with the minimax rate $R(\Theta_\beta)$. 

For a scale $\Theta(\mathcal{B})=\{\Theta_\beta, \, \beta\in\mathcal{B}\}$, let
$\{R(\Theta_\beta), \, \beta \in \mathcal{B}\}$ be  the family of the pertaining 
minimax rates. Suppose (\ref{local_impl_global}) is fulfilled for the local 
rate $\mathnormal{r}(\theta_0)$ defined by (\ref{oracle_risk}) and  
$\{R(\Theta_\beta), \, \beta \in \mathcal{B}\}$. Then, in view of (\ref{local_impl_global}), 
Theorem \ref{th1} entails that the DDM $\mathrm{P}(\cdot|X)$ (\ref{ddm1})
must also contract to $\theta_0$ with (at least) the minimax rate  $R_\varepsilon(\Theta_\beta)$ 
uniformly in $\theta_0\in\Theta_\beta$ for each $\beta\in\mathcal{B}$. 
Thus, the adaptive %and uniform 
(over the scale $\Theta(\mathcal{B})$) 
minimax contraction rate result for $\mathrm{P}(\cdot|X)$ follows immediately. 
Foremost, Theorem \ref{th1} implies adaptive minimax results 
\emph{simultaneously for all scales} %(!), %(in fact, for all sets),
for which (\ref{local_impl_global}) is fulfilled.
%that are covered by this oracle.
%Thus, the usefulness and strength of the oracle approach stands or falls by, 
%firstly, the availability of an oracle result for a family of local rates
%(like our Theorem \ref{th1}), secondly, the range of scales that are covered 
%by the oracle for which such a result is established. Finally, 
In particular, (\ref{local_impl_global}) is satisfied for 
%the local rate (\ref{oracle_risk})  and %, among others, 
the following scales: Sobolev and analytic ellipsoids, Sobolev hyperrectangles 
(in fact, rather general $\ell_2$-ellipsoids and hyperrectangles, considered below), 
certain scales of Besov classes  and $\ell_p$-bodies, tail classes. 
See Supplement for details where we also consider the situation when the 
local oracle results over one family of rates imply the local oracle results over 
another family of rates. 

%The notion of covering can be extended to different families of local rates.
%If one family of local rates covers another  (we say the first 
%family is \emph{stronger} than the second), then the results for the oracle local 
%rate over the first family will imply the results for the oracle local rate for the second family.

For example, consider general ellipsoids  and hyperrectangles 
\begin{align}
\label{ellipsoids}
\mathcal{E}(a) = \big\{ \theta \in \ell_2: \, \sum\nolimits_i (\tfrac{\theta_i}{a_i})^2 \le 1\big\}, \;
\mathcal{H}(a)= \{ \theta \in \ell_2: \, |\theta_i|\le a_i, \, i \in \mathbb{N}\},
\end{align}
where $a=(a_i, i \in \mathbb{N})$ is nonincreasing  sequence of numbers
in $[0,+\infty]$ which converge to 0 as $i\to \infty$, $a_1 \ge c_1\varepsilon$ for some $c_1>0$. 
Here we adopt the conventions $0/0=0$ and $x/(+\infty)=0$ for $x \in \mathbb{R}$.
Let $R^2(\Theta) = \inf_{\hat{\theta}}\sup_{\theta \in \Theta} 
\mathrm{E}_\theta \|\hat{\theta}-\theta\|^2$ denote the (quadratic) minimax risk 
over a set $\Theta$, where the infimum is taken over all possible estimators 
$\hat{\theta}=\hat{\theta}(X)$, measurable functions of the data $X$. 
One can show (see Supplement) that 
\begin{align}
\label{loc<global}
\sup_{\theta_0 \in \mathcal{E}(a)} \mathnormal{r}^2(\theta_0)\le (2\pi)^2 R^2(\mathcal{E}(a)), 
\;\; 
\sup_{\theta_0 \in \mathcal{H}(a)}\mathnormal{r}^2(\theta_0)\le \tfrac{5}{2} R^2(\mathcal{H}(a)).
\end{align}
Instead of $(2\pi)^2$, one can put a tighter constant $4.44$ in the direct case, 
which possibly holds for the ill-posed case as well; see Supplement. 
Then Theorem \ref{th1} implies  that
%\begin{align}
%\label{impl_th1}
\[
\sup_{\theta_0\in \Theta(a)} 
\mathrm{E}_{\theta_0} \mathrm{P}\big(\|\theta-\theta_0\|
\ge M R(\Theta(a)) \big| X\big) \le \frac{C}{M^2}, \quad \text{for some} \;\;
C=C(K,\alpha),
\]
%\end{align}
where $\Theta(a)$ is either $\mathcal{E}(a)$ or $\mathcal{H}(a)$ for some \emph{unknown} $a$.
In particular, we obtain the minimax contraction rates for the four 
scales considered in \cite{Szabo&etal:2015}:
two families of ellipsoids (Sobolev and analytic) and two families of
hyperrectangles (Sobolev and parametric); see Supplement.
%Minimax contraction rates for the DDM $\mathrm{P}(\cdot|X)$ 
%over %more general  certain families of $\ell_p$-bodies 
%other smoothness scales (e.g., families of tail classes and certain  $\ell_p$-bodies) 
%can be derived in a similar way. 

It turns out that the DDM $\mathrm{P}(\cdot|X)$ defined by (\ref{ddm1}) can also be used for estimating 
the parameter $\theta_0$. Namely, define the estimator
\begin{equation}
\label{tilde_estimator}
\tilde{\theta}= \mathrm{E}(\theta|X) = 
\sum\nolimits_I X(I) \mathrm{P}(\mathcal{I}=I|X), \;\;  %\text{with} \quad 
X(I)=(X_i\mathrm{1}\{i\le I\}, i\in\mathbb{N}),
\end{equation}
which is just the DDM $\mathrm{P}(\cdot|X)$-expectation. 
This estimator satisfies the following oracle estimation inequality.
\begin{theorem}[Oracle inequality]
\label{th2}
Let the conditions of Theorem \ref{th1} be fulfilled,
$\theta_0\in\ell_2$, and $\tilde{\theta}$  be defined by (\ref{tilde_estimator}).
Then there exist  a constant $C_{est}=C_{est}(K,\alpha)\ge 1$
such that 
$\mathrm{E}_{\theta_0} \|\tilde{\theta}-\theta_0\|^2
\leq C_{est}\mathnormal{r}^2(\theta_0)$, where
the oracle rate $\mathnormal{r}(\theta)$ is  defined by (\ref{oracle_risk}).
\end{theorem}
This theorem yields the whole panorama of the minimax adaptive 
estimation results in the mildly ill-posed inverse setting, simultaneously over all scales 
for which (\ref{local_impl_global}) is fulfilled. The proof of this theorem 
is essentially contained in the proof of Theorem \ref{th1}, 
but it is still provided in Supplement. Similar result has been obtained in
\cite{Cavalier&Golubev:2006} for the estimator based on the risk hull minimization 
method. In that paper, the oracle rate has an extra penalty term
but the multiplicative constant is very tight.

\subsection{Local DDM-contraction rate: lower bound under EBR}  
\label{lebr}
 
As we mentioned in the introduction, %it has been realized by many researchers that 
in general it is impossible to construct optimal (fully) adaptive confidence set 
in the minimax sense with a prescribed high coverage probability.
Actually, this is a genuine problem, not connected with an optimality framework used: 
global minimax %(over certain scale) 
(\ref{adapt_conf_ball_problem}),  or local (\ref{conf_ball_problem}). 
%Namely, the same problem must emerge (and it does) for the local framework 
%(\ref{conf_ball_problem}) as well, %(at least for local rates that cover typical scales), 
%otherwise we would have constructed an adaptive confidence set in the minimax sense. 
Clearly, the same problem should occur for the  local approach (\ref{conf_ball_problem}), 
because otherwise we would have solved the minimax version of the problem as well. 
The intuition %behind this phenomenon 
is  that there are so called ``deceptive''  parameters $\theta_0$ 
that ``trick'' the DDM $\mathrm{P}(\cdot|X)$ in the sense that the %resulting
random radius $\hat{r}$ defined by (\ref{radius}) is overoptimistic, i.e.,
of a smaller order than the actual radial rate  $\mathnormal{r}(\theta_0)$. 
The coverage probability is then too small. 
%Namely, one can prove that $\hat{r}^2$ is always of the order
%$\varepsilon^2 I_o(\theta_0)$ which is just the variance term  
%of the oracle rate (\ref{oracle_risk}).
%The DDM $\mathrm{P}(\cdot|X)$ can recover the variance 
%term of the oracle rate from the data, but not the bias term 
%and this is where the mismatch can occur. If the bias term
%$\sum_{i\in\mathbb{N}^c_{I_o}} \theta_{0,i}^2$ in (\ref{oracle_risk})
%is of a bigger order than the variance term, then $\hat{r}$ is of a smaller order 
%than it should be, which leads to a small coverage probability of the resulting confidence ball. 
%A deceptive $\theta_0$  ``pretends'' to be smooth (small variance term, detected by the DDM), 
%but it is not  really (big bias term, not detected by the DDM).

A way to fix this problem is to remove a set (preferably, minimal) of deceptive 
parameters  from the set $\Theta$ (in our case $\ell_2$) and derive the coverage 
relation in (\ref{conf_ball_problem}) for the remaining set of non-deceptive parameters. %$\Theta_{cov}$. 
In a different framework, \cite{Picard&Tribouley:2000} 
introduced such a set, the so called \emph{self-similar} (SS) parameters, 
studied later by many authors in various settings and models.  
%%Gin\'e and Nickl (2010), Hoffmann and Nickl (2011), Bull (2012), Bull and Nickl (2013), 
%%Szab\' o, van der Vaart and van Zanten (2014a, 2014b), Nickl  and Szab\' o (2014)
%%studied self-similar parameters in different settings  and different models. 
%Szab\'o et al.\ %van der Vaart and van Zanten 
%(2015) defined $\ell_2$ self-similar parameters adopted to the Sobolev 
%hyperrectangle scale as follows. Introduce the Sobolev 
%hyperrectangle $\mathcal{H}_S(\beta,Q)= \mathcal{H}(a)$, with 
%$a_i^2=Qi^{-(2\beta+1)}$ and $\mathcal{H}(a)$ given by (\ref{ellipsoids}).
%%$ \{\theta \in \ell_2: \,\sup_i  i^{2\beta+1} \theta_i^2 \le Q \}$
%For some $0<\beta_{\min}<\beta_{\max}<\infty$, the class of self-similar 
%parameter is $\Theta_{ss}=\cup_{\beta\in[\beta_{\min},\beta_{\max}]} \Theta_{ss}(\beta)$,
%%\begin{align*}
%%\label{def_ss}
%where $\Theta_{ss}(\beta)=\Theta_{ss}(\beta,Q,\gamma, N_0, \rho_0)=
%\big\{ \theta  \in \mathcal{H}_S(\beta,Q): \, \sum_{i=N}^{\rho_0 N}
%\theta_i^2 \ge \gamma Q N^{-2\beta},\; \forall N\ge N_0\big\}$
%for some $\gamma, N_0>0$ and $\rho_0 \ge 2$.
%%\end{align*}
%%In a way,  the set $\Theta_{ss}$  consists of  
%%the ``edges'' of the Sobolev ellipsoids  $\Theta_\beta(Q)$, 
%%$\beta \in[\beta_{\min},\beta_{\max}]$.
%%\begin{remark}
%%Instead of Sobolev hyperrectangles $H_\beta(Q)$, one can use Sobolev ellipsoids 
%%$E_\beta(Q) = \{\theta \in \ell_2: \, \sum_i i^{2\beta} \theta_i^2 \le Q \}$ 
%%and an adjusted self-similarity condition, with the same conclusions, but different constants.
%%\end{remark}
A somewhat restrictive feature of  the self-similarity property is that
it is linked to the Sobolev (Besov) smoothness scale.  
In \cite{Szabo&etal:2015} a more general condition is introduced
that is not linked to a particular smoothness scale, 
the \emph{polished tail} (PT) condition: for some $L_0>0$ ($L_0\ge 1$ for $\Theta_{pt}$ 
to be not empty), $N_0\in\mathbb{N}$ and $\rho_0\ge 2$,
\[
\Theta_{pt}=\Theta_{pt}(L_0,N_0,\rho_0)=\Big\{\theta \in \ell_2:\, 
\sum_{i=N}^\infty \theta_i^2 \le L_0 \sum_{i=N}^{\rho_0 N} \theta_i^2, 
\; \forall N\ge N_0 \Big\}.
\]
%For any $\beta,Q,\gamma, N_0>0$ and $\rho_0 \ge 2$ there exist 
%$L_0, N'_0>0$ and $\rho'_0\ge 2$ such that 
%$\Theta_{ss}(\beta,Q,\gamma, N_0, \rho_0) \subseteq\Theta_{pt}(L_0,N'_0,\rho'_0)$.
%%e.g., one can take $N'_0 =N_0$, $\rho'=\rho$ and $L_0=\gamma^{-1}$.
In \cite{Szabo&etal:2015} it is shown that  %We express this property as 
$\Theta_{ss} \subseteq \Theta_{pt}$, i.e., PT is more general than SS.

Introduce the \emph{surrogate oracle rate} 
$\mathnormal{r}(\bar{I}_o,\theta_0)$, with the   
\emph{surrogate oracle} $\bar{I}_o$ defined as follows:
\begin{align}
\label{surr_oracle}
\bar{I}_o =\arg\!\min_I \mathnormal{R}^2(I,\theta_0), \quad
\mathnormal{R}^2(I,\theta_0)=\mathnormal{R}_\sigma^2(I,\theta_0) = I\varepsilon^2+ 
\sum_{i>I} \frac{\theta_{0,i}^2}{\kappa_i^2}. %\quad I \in \mathbb{N}.
\end{align}
The quantity $R(I,\theta_0)$ is nothing else but the oracle rate for the parameter 
$\bar{\theta}=(\theta_i/\kappa_i,\, i\in\mathbb{N})$ in the ``direct'' model 
$\tilde{X}= (X_i/\kappa_i, \, i \in \mathbb{N})\sim\bigotimes_i N(\bar{\theta}_i,\varepsilon^2)$.
Note that $\bar{I}_o=I_o$ in the direct case $\kappa^2_i=1$.

Introduce the \emph{excessive bias restriction} (EBR): 
$\theta_0\in \Theta_{eb}(\tau)$  for $\tau>0$,
\[
\Theta_{eb}=\Theta_{eb}(\tau)=\Theta_{eb}(\tau,\varepsilon)
=\Big\{\theta \in \ell_2:\,\sum_{i >\bar{I}_o}
\theta_i^2 \le \tau  \sum_{i\le \bar{I}_o} \sigma_i^2\Big\},
\]
where $\bar{I}_o=\bar{I}_o(\theta)$ is defined by (\ref{surr_oracle}).
Note that in principle $\Theta_{eb}$ also depends on $\varepsilon$ as 
we consider the non-asymptotic setting. For asymptotic 
considerations (as $\varepsilon \to 0$), we can introduce a uniform (in $\varepsilon$) 
version of EBR:
\[
\bar{\Theta}_{eb}(\tau, \varepsilon_0) = \big\{\theta \in \Theta_{eb}(\tau,\varepsilon)
 \; \text{ for all } \varepsilon \in (0, \varepsilon_0]\big\} = 
 \cap_{\varepsilon \in (0, \varepsilon_0]}\Theta_{eb}(\tau,\varepsilon).
\]
We  will not consider $\bar{\Theta}_{eb}(\tau, \varepsilon_0)$ 
and $\Theta_{eb}(\tau,\varepsilon)$ separately  and will always 
use the latter notation  $\Theta_{eb}(\tau)$ for both in what follows, 
with the understanding that whenever one needs the uniform version, 
one can think of  $\Theta_{eb}(\tau)$ as $\bar{\Theta}_{eb}(\tau, \varepsilon_0)$,
as all assertions below hold also for the uniform version of EBR.

%Basically, EBR means that the bias term  may not dominate 
%(up to a constant factor) the variance term in the surrogate oracle rate $\mathnormal{r}(\bar{I}_o,\theta_0)$:
%in a way, the parameters from $\Theta_{eb}$ are ``typical''.
%%the oracle bias is of the same order as the oracle variance. 
%Atypical ones have atypically excessive biases and these are removed from 
%$\ell_2$ once we assume $\theta \in \Theta_{eb}$. In this case tricking the 
%DDM is not possible anymore.
%%Another way of looking at this condition is that the local rate $\mathnormal{r}(\theta_0)$ 
%%is estimable from the data for $\theta \in \Theta_{eb}$.

Let us show that EBR is less restrictive than PT, i.e., 
for any $L_0\ge 1$, $N_0\in\mathbb{N}$ and $\rho_0\ge 2$, 
there exists a $\tau>0$ such that 
$\Theta_{pt}(L_0,N_0,\rho_0) \subseteq \Theta_{eb}(\tau)$.  
From (\ref{surr_oracle}), it follows that for  any $I>\bar{I}_o$, 
$\sum_{i=\bar{I}_o+1}^{I} \frac{\theta_i^2}{\sigma_i^2} \le I-\bar{I}_o$.  
Besides, by condition (i) in (\ref{condition_sigma}),
$(n-l) \sigma_n^2 \le K_1 \sum_{i=l+1}^n \sigma_i^2$ 
(indeed, $\sigma$ is non-decreasing and $K_1\le 1$) 
for all $n,l \in \mathbb{N}$ such that $n> l$.
Using the last two relations and the property (ii) 
from (\ref{condition_sigma}), we obtain for any
$\theta\in \Theta_{pt}(L_0,N_0,\rho_0)$ that
\begin{align*}
\sum_{i=\bar{I}_o+1}^\infty  \theta_i^2 &=
\sum_{i=\bar{I}_o+1}^{N_0\bar{I}_o -1}  \theta_i^2 
+\sum_{i=N_0\bar{I}_o}^\infty \theta_i^2 \le  
\sum_{i=\bar{I}_o+1}^{N_0\bar{I}_o -1}  \theta_i^2+
L_0 \sum_{i=N_0\bar{I}_o}^{\rho_0 N_0 \bar{I}_o} \theta_i^2 \\
&\le L_0  \sigma_{\rho_0 N_0\bar{I}_o}^2 
\sum_{i=\bar{I}_o+1}^{\rho_0 N_0 \bar{I}_o} \frac{\theta_i^2}{\sigma_i^2}
\le  L_0  \sigma_{\rho_0 N_0\bar{I}_o}^2 (\rho_0 N_0 \bar{I}_o-\bar{I}_o) \\
&\le L_0 K_1 \sum_{i=\bar{I}_o+1}^{\rho_0 N_0 \bar{I}_o}  \sigma_i^2 
\le L_0 K_1 \sum_{i=1}^{\rho_0 N_0 \bar{I}_o}  \sigma_i^2 
\le L_0 K_1 K_2(\rho_0 N_0) \sum_{i=1}^{\bar{I}_o}  \sigma_i^2 ,
\end{align*}
so that $\Theta_{pt}(L_0,N_0,\rho_0)
\subseteq \Theta_{eb} (L_0 K_1 K_2(\rho_0 N_0))$ for any $N_0\ge 1$.

Summarizing the relations between three types of conditions describing 
non-deceptive parameters introduced above,
$
\Theta_{ss}\subseteq \Theta_{pt} \subseteq \Theta_{eb}.
$
Thus EBR is the most general condition among these three. %set of non-deceptive parameters.
As to the question how big (or ``typical'') that set $\Theta_{eb}$ 
%of non-deceptive parameters 
is, \cite{Szabo&etal:2015} gives three types of arguments for the 
PT-parameters: topological, minimax and Bayesian. Since
$\Theta_{eb} \supseteq \Theta_{pt}$, the same arguments  certainly 
apply to $\Theta_{eb}$; 
%We will not discuss this any further, but refer to the paper of 
see \cite{Szabo&etal:2015} for more details on this.

%\subsection{Local DDM-contraction rate: lower bound under EBR}
%\label{lower_bound}

Now we are ready to formulate the lower bound result for the %local 
DDM-contraction rate.
\begin{theorem}[Small ball DDM-probability]
\label{th2a_posterior}
Let the DDM $\mathrm{P}(\cdot|X)=\mathrm{P}_{K,\alpha}(\cdot|X)$ 
be given by (\ref{ddm1}),
with parameters $K,\alpha>0$ such that %chosen in such a way that 
\begin{equation}
\label{rho}
\alpha < a(K)\triangleq\frac{1}{4} -\frac{1}{2} \log\Big(\frac{K+1}{2}\Big). 
\end{equation}
Then there exists $C_{sb}=C_{sb}(K,\alpha)>0$  such that,
for any $\theta_0\in \ell_2$, any DD-center $\hat{\theta}=\hat{\theta}(X)$ and 
any $\delta\in(0,\delta_{sb}]$ with $\delta_{sb}=1 \wedge 
\Big(\sqrt{\frac{K(2p+1)}{K+1}}\big(\frac{a(K)-\alpha}{4ea(K)}\big)^{p+\frac{1}{2}}\Big)$,
\[
\mathrm{E}_{\theta_0} \mathrm{P} \big(\|\theta- \hat{\theta}\|
\le \delta \Sigma^{1/2}(\bar{I}_o)| X\big) \le 
C_{sb}\delta \big[\log(\delta^{-1})\big]^{p+1/2},
\]
where $\Sigma(\bar{I}_o) =\sum_{i\le \bar{I}_o} \sigma_i^2$ and  
$\bar{I}_o=\bar{I}_o(\theta_0)$ is defined by (\ref{surr_oracle}).
\end{theorem}
Notice that the effective rate in the above lower bound is 
determined by the variance term  
of the oracle surrogate rate $\mathnormal{r}^2(\bar{I}_o,\theta_0)$.
Recall that EBR says basically that the variance term
is the main term in the surrogate oracle rate $\mathnormal{r}^2(\bar{I}_o,\theta_0)$.
Under $\theta_0\in \Theta_{eb}(\tau)$, we thus have  
$\mathnormal{r}^2(\theta_0) =\mathnormal{r}^2(I_o,\theta_0) \le \mathnormal{r}^2(\bar{I}_o,\theta_0) \le (1+\tau) \Sigma(\bar{I}_o)$.
This yields the following corollary.
\begin{corollary}[%local DDM-contraction rate: 
Lower bound under EBR] 
\label{cor_lower_bound}
Let the conditions of Theorem \ref{th2a_posterior} be satisfied.
Then, %there exists $C_{eb}=C_{eb}(K,\alpha,\tau)>0$  such that,
for any $\theta_0\in \ell_2$, any DD-center $\hat{\theta}=\hat{\theta}(X)$, any $\tau>0$ 
and any $\delta\in(0,\delta_{eb}]$ with $\delta_{eb} = (1+\tau)^{-1/2} \delta_{sb}$,
\[
\sup_{\theta_0\in \Theta_{eb}(\tau)} 
\mathrm{E}_{\theta_0} \mathrm{P} \big(\|\theta- \hat{\theta}\|\le 
\delta \mathnormal{r}(\theta_0)| X\big) \le C_{eb}\delta \big[\log(\delta^{-1})\big]^{p+1/2},
\]
where $C_{eb}=%C_{eb}(K,\alpha,\tau)=
C_{sb} \sqrt{1+\tau}$, $\delta_{sb}$ and $C_{sb}$ are from Theorem \ref{th2a_posterior}.
\end{corollary}
A bound (not the sharpest) for the constant $C_{sb}$ can be found in the proof of 
Theorem \ref{th2a_posterior}. The above assertion implies condition (\~A2)  
for the DDM  $\mathrm{P}(\cdot|X)$ satisfying the conditions of 
Theorem \ref{th2a_posterior}, with $\psi(\delta) =C_{eb}\delta\big[\log(\delta^{-1})\big]^{p+\frac{1}{2}}$,
uniformly in $\theta_0\in\Theta_{eb}(\tau)$.
This ensures the coverage relation in (\ref{conf_ball_problem}), see the next subsection.

\subsection{The main result: confidence ball under EBR}

In this subsection we establish the main result of the paper.
Let the DDM $\mathrm{P}(\cdot|X)
=\mathrm{P}_{K,\alpha}(\cdot|X)$ be given by (\ref{ddm1}),
with constants $K,\alpha>0$ such that $K > 1.87$ and condition (\ref{rho}) is fulfilled. 
%The constants $K,\alpha,L$ are fixed throughout this subsection.
By using the DDM  $\mathrm{P}(\cdot|X)$, we construct the default DD-center
$\check{\theta}$ (\ref{hat_theta})
and the default confidence ball  $\tilde{B}=B(\check{\theta}, M\hat{r}_\kappa)$ 
given by (\ref{def_ball}), with some fixed $\kappa \in (0,1)$, say $\kappa=\frac{1}{2}$. 
Theorem \ref{th1} implies (\~A1) with 
$\varphi(M)=\frac{C_{or}}{M^2}$. % (see also Remark \ref{rem_tildeA}). 
Then by Proposition \ref{prop1},  
(A1) and (A3) are also fulfilled for the DDM $\mathrm{P}(\cdot|X)$ and 
the default DD-center $\check{\theta}$, 
with (see Remark \ref{rem_def_ball}) $\phi_1(M)= \frac{42C_{or}}{M^2}$, 
$\phi_2(M)=\frac{10C_{or}}{M^2}$, uniformly in $\theta_0\in\ell_2$.

Let us bound the coverage probability of the default confidence ball 
$B(\check{\theta}, M\hat{r}_\kappa)$. In view of Corollary \ref{cor_lower_bound}, 
we conclude that condition (\~A2) is met with 
$\psi(\delta) = C_{eb}\delta \big[\log(\delta^{-1})\big]^{p+1/2}$, 
uniformly in $\theta_0\in \Theta_{eb}(\tau)$. 
%According to (\ref{A3_conf_ball}),  (A3) is also fulfilled  with 
%$\phi_2(M)= C_{est}/M$, uniformly in $\theta_0\in \ell_2$.
As also (A3) is fulfilled  with $\phi_2(M)=\frac{10C_{or}}{M^2}$ uniformly in $\theta_0\in \ell_2\supseteq \Theta_{eb}$,
by applying Proposition \ref{prop2a} we derive that,
for each $\theta_0\in \Theta_{eb}(\tau)$,
\[
%\sup_{\theta_0\in \Theta_{eb}(\tau)}\!\!\!
\mathrm{P}_{\theta_0} \big(\theta_0 \not \in
B(\check{\theta},M\hat{r}_\kappa)\big)\le
\phi_2(M \delta)+ \frac{\psi(\delta)}{1-\kappa} = 
\frac{10C_{or}}{M^2\delta^2} +
\frac{C_{eb}\delta \big[\log(\delta^{-1})\big]^{p+\frac{1}{2}}}{1-\kappa}
\]
for any $M,\delta>0$.
For $\alpha_1\in(0,1)$ and $\delta_{eb}$ defined in Corollary \ref{cor_lower_bound}, %Theorem \ref{th2a_posterior}, 
we take
\[
\delta_1=\max\big\{\delta\in (0,\delta_{eb}]:\,
 C_{eb}(1-\kappa)^{-1} \delta \big[\log(\delta^{-1})\big]^{p+1/2} \le \alpha_1/2\big\}
\]
and $M_1=\min\{M\in\mathbb{N}:\, 10C_{or}/(M\delta_1)^2 \le \alpha_1/2 \}$.
Then, for all $M\ge M_1$,  
\begin{align}
\label{rel_alpha_1}
\sup_{\theta_0\in \Theta_{eb}(\tau)}
\mathrm{P}_{\theta_0} \big(\theta_0 \not \in
B(\check{\theta},M\hat{r}_\kappa)\big)\le \alpha_1.
\end{align}

Now, since condition (A1) is satisfied with $\phi_1(M)= \frac{42C_{or}}{M^2}$, applying 
Proposition  \ref{prop2b}  yields that  
the  size $\hat{r}_\kappa$ of the confidence ball  
$B(\check{\theta},  M\hat{r}_\kappa)$ is of the local radial rate order: 
\[
\mathrm{P}_{\theta_0}\big(\hat{r}_\kappa \ge M
\mathnormal{r}(\theta_0)\big)
\le \frac{\phi_1(M)}{\kappa}= \frac{42C_{or}}{\kappa M^2},
\]
for any $M>0$ and all $\theta_0\in\ell_2$.
For $\alpha_1\in(0,1)$, take 
$M_2=\min\{M\in\mathbb{N}:\, 42C_{or}/(\kappa M^2) \le \alpha_2 \}$.
Then for any $M \ge M_2$ 
\begin{align}
\label{rel_alpha_2}
\sup_{\theta_0\in\ell_2} \mathrm{P}_{\theta_0}\big(\hat{r}_\kappa \ge M
\mathnormal{r}(\theta_0)\big)
\le \alpha_2.
\end{align}
By combining (\ref{rel_alpha_1}) and  (\ref{rel_alpha_2}), we obtain 
the main result of the paper.
\begin{theorem}[Confidence optimality under EBR]
\label{th6}
Let  the DDM $\mathrm{P}(\cdot|X)
=\mathrm{P}_{K,\alpha}(\cdot|X)$ be given by (\ref{ddm1}),
with constants $K,\alpha>0$ such that $K \ge 1.87$ and (\ref{rho}) is fulfilled. 
Further, let $B(\check{\theta},  M\hat{r}_\kappa)$ 
be the default  confidence ball defined by (\ref{def_ball}).
Then for any $\tau>0$ and any $\alpha_1,\alpha_2 \in (0,1)$ there exist 
$C_0=C_0(\alpha_1, \tau)$ and $c_0=c_0(\alpha_2)$ such that, for any
$C \ge C_0$  and $c \ge c_0$, the following relations hold
\[
\sup_{\theta_0\in\Theta_{eb}(\tau)} \mathrm{P}_{\theta_0} 
\big(\theta_0\not\in B(\check{\theta},C\hat{r}_\kappa) \big) 
\le \alpha_1, \quad 
\sup_{\theta_0\in\ell_2} \mathrm{P}_{\theta_0} 
\big( \hat{r}_\kappa \ge c \mathnormal{r}(\theta_0) \big) 
\le \alpha_2,
\]
where the local radial rate $\mathnormal{r}(\theta_0)$ is defined by (\ref{oracle_risk}).
\end{theorem}

In the proofs of Theorems \ref{th1} and \ref{th2a_posterior}, 
tighter exponential bounds are possible 
(based on the exponential %large deviation 
bounds for the $\chi^2$-distribution),
which would presumably lead to exponential functions $\varphi$ and  $\psi$ in 
conditions (\~A1) and  (\~A2). We however use simpler bounds obtained 
by the Markov inequality for the sake of a succinct presentation.
Another useful feature of this approach is that it can be extended to 
non-normal DDMs (\ref{measure_P_I}); see Supplement.
 
\subsection{Concluding remarks} % and discussion}
\label{szabo}

%In this subsection we collect some remarks, more remarks can be found in Supplement. 

%\begin{remark}
%\label{rem_undersmooth}
\paragraph{Range for constant $K$}
%Note that the conditions $K\ge 1.87$ and $\alpha<a(K)$ (with $a(K)$ defined by (\ref{rho})) 
%in Theorem \ref{th6} come from Theorems \ref{th1} and \ref{th2a_posterior} respectively.
The condition $\alpha<a(K)$ ((\ref{rho}) in Theorem \ref{th2a_posterior}) 
limits room for choosing constants $K,\alpha >0$, because $a(K)>0$ only for 
$K \in (0,2e^{1/2}-1)$. One can choose, for example, $K=2$ and %$K_2=0$, 
$\alpha=0.04$. 
%we may have lost room because of (possibly crude) bounds in the proof.  
Even less room for $K$ remains if we also want Theorem \ref{th1} to hold 
(and this is needed for the main result, Theorem \ref{th6}). Indeed,  
then $K$ has also to satisfy $K\ge 1.87$, so that the final range of allowable $K$'s  
becomes $K \in [1.87, 2.29] \subset [1.87,2e^{1/2}-1)$.
The conditions $K\ge 1.87$ and $\alpha<a(K)$  
%is intuitively clear as the parameter $\alpha$ has to be relatively small
%for the theorem to hold,  
are apparently more strict than needed for the corresponding theorems to hold, 
since of course not the most accurate bounds are used in the proof. 
%because of artifacts of our proof technique.

%The larger the parameter $\tau$, the larger the constant $C_{eb}$. 
%On the other hand, $C_{eb}$ can be made smaller by choosing a  
%larger $L$. This in turn would increase the constant $C_{or}$
%in Theorem \ref{th1}, (cf.\ Remark \ref{rem_choice} where we 
%mentioned the best choice  $L=0$ to make the constant $C_{or}$
%as small as possible, but this choice would make $C_{eb}$ infinite).
%\end{remark} 

\paragraph{Alternative DD-center and confidence ball}
In Theorem \ref{th6}, instead of the default DD-center $\check{\theta}$ we can 
use the estimator $\tilde{\theta}$ defined by  (\ref{tilde_estimator}). Indeed,
by Theorems \ref{th1} and \ref{th2}, we have that, uniformly in $\theta_0\in\ell_2$,
\begin{align*}
&\mathrm{E}_{\theta_0} \big[\mathrm{P}(\|\theta-\tilde{\theta}\|\ge M
\mathnormal{r}(\theta_0)|X) \big] \le 
\mathrm{E}_{\theta_0} \big[\mathrm{P}(\|\theta-\theta_0\| \ge 
\tfrac{1}{2}M\mathnormal{r}(\theta_0)|X) \big] \notag\\
%\label{A1_conf_ball}
& \qquad + 
\mathrm{E}_{\theta_0} \big[\mathrm{P}(\|\theta_0-\tilde{\theta}\| \ge 
\tfrac{1}{2}M\mathnormal{r}(\theta_0)|X) \big]
\le  \frac{4(C_{or}+C_{est})}{M^2}=\phi_1(M),\\
%\label{A3_conf_ball}
&\mathrm{P}_{\theta_0}\big(\|\theta_0-\tilde{\theta}\| \ge 
 M\mathnormal{r}(\theta_0)\big)
\le \frac{\mathrm{E}_{\theta_0} \|\theta_0-\tilde{\theta}\| ^2}
{M^2 \mathnormal{r}^2(\theta_0)} \le \frac{C_{est}}{M^2}=\phi_2(M).
\end{align*}
This means that  conditions (A1) and (A3) are also fulfilled for 
the estimator $\tilde{\theta}$ defined by (\ref{tilde_estimator}), 
%and the DDM $\mathrm{P}(\cdot|X)$,
with $\phi_1(M)=\frac{4(C_{or}+C_{est})}{M^2}$ and  
$\phi_2(M)=\frac{C_{est}}{M^2}$, uniformly in $\theta_0\in\ell_2$.
Arguing as above, we obtain that Theorem \ref{th6} also holds for the DD-center $\tilde{\theta}$ 
and the confidence ball  $B(\tilde{\theta},  M\hat{r}_\kappa)$. 
%defined by  (\ref{tilde_estimator}) and (\ref{radius})--(\ref{conf_ball}) respectively.
%In fact, any other estimator that satisfy the oracle inequality in  Theorem \ref{th2}
%will do the job, for example, an estimator which is based on risk hull minimization method, due to 
%Cavalier and Golubev (2006).
%%for the direct case: a blockwise Stein's estimator from Cavalier and Tsybakov (2001) or a 
%%penalized estimator from Birg\' e and Massart (2001). 

\paragraph{Connection with the minimax results of \cite{Szabo&etal:2015}} 
%Connection to the results of Szab\' o et al.\ (2015).}
For the mildly ill-posed inverse signal-in-white-noise model (\ref{model}), 
an intriguing paper \cite{Szabo&etal:2015} deals with a certain 
Sobolev type family of priors, indexed by a smoothness parameter. 
The proposed DDM is the empirical Bayes posterior with respect to the smoothness 
parameter. This DDM is then used to construct a DDM-credible ball 
whose coverage and size properties are studied. 
The main results of  the paper %this intriguing paper 
are the asymptotic (in our notation: as $\varepsilon\to 0$) versions of the minimax 
framework (\ref{adapt_conf_ball_problem}) with $\Theta'_{cov} =\Theta_{pt}$
(the \emph{polished tail} class $\Theta_{pt}$ defined in Subsection \ref{lebr}),
and four choices of scales: Sobolev type scales of hyperrectangles and ellipsoids 
and the two so called \emph{supersmooth} scales 
(analytic ellipsoid and parametric hyperrectangle). %$C^{00}$ and $S^{\infty,c,d}$. 
%In the paper, these are denoted as $\Theta^\beta(M)$,  $S^\beta(M)$, $C^{00}$ and 
%$S^{\infty,c,d}$, and given by the formulas (1.1), (1.2), (3.7) and (3.8), respectively. 
The proposed DDM %, the (Sobolev) empirical Bayes posterior, 
is well suited to model Sobolev-type scales: the optimal (minimax) radial rates are 
obtained in the size relation of  (\ref{adapt_conf_ball_problem}) for Sobolev hyperrectangles 
and ellipsoids; but only suboptimal rates are obtained 
for the two supersmooth scales.

Since $\Theta_{pt}\subseteq \Theta_{eb}$ and the considered
four scales are particular examples of scales for which (\ref{local_impl_global}) holds,
the non-asymptotic versions of the minimax results for all 
these four scales (including the two supersmooth scales) immediately 
follow from Theorem \ref{th6} for the DDM $\mathrm{P}(\cdot|X)$ (\ref{ddm1}). 
Asymptotic versions can readily be derived from 
the non-asymptotic ones. 
We emphasize that the scope of the DDM $\mathrm{P}(\cdot|X)$ 
in delivering the minimax rates extends further than just these four scales. Theorem \ref{th6} 
implies the minimax results  of type (\ref{adapt_conf_ball_problem}) for all 
scales for which (\ref{local_impl_global}) holds; for example, in view of (\ref{loc<global}), for all ellipsoids  
$\mathcal{E}(a)$ and hyperrectangles $\mathcal{H}(a)$ defined by (\ref{ellipsoids}).
%with non-increasing $a$.  
Other smoothness scales %(e.g., consisting of tail classes, $\ell_p$-bodies) 
can also be treated. Some details are provided in Supplement. 
 
%Actually,  our approach allows to get the minimax results 
%(\ref{adapt_conf_ball_problem}) %for our DDM-credible ball
%for the ill-posed case as well. Minor modifications are needed in the construction 
%of the DDM to accommodate the ill-posedness. The size property in (\ref{adapt_conf_ball_problem})
%can be derived for all four smoothness scales  with the corresponding  minimax rates,
%by using the local results for the direct case. 
%Also in the ill-posed case, the resulting DDM  handles more than the 
%considered four scales. In the same way we can derive the minimax results  
%for other smoothness scales as well. 
%However, the coverage property in (\ref{adapt_conf_ball_problem}) for the ill-posed case 
%is not a consequence of this property for the direct case, 
%the proof for the ill-posed case  has to be done essentially from scratch (although somewhat similar to the direct 
%case). It turns out that the coverage property holds for some $\Theta'_{cov} = \Theta'_{eb}$, 
%different from $\Theta_{eb}$ (defined by (\ref{cond_ebr})), but still such that 
%$\Theta_{pt} \subseteq \Theta'_{eb}$. Details are provided in Supplement. 
 
\section{Proofs of Theorems \ref{th1} and \ref{th2a_posterior}}
\label{sec_proofs}

\subsection{Proof of Theorem \ref{th1}}
%We prove Theorem \ref{th1} in several steps.
\paragraph{Step 1: bounds for $\mathrm{E}_{\theta_0}\mathrm{P}(\mathcal{I}=I|X)$}
 
%By the dominated convergence theorem,
%\begin{eqnarray}
%\label{mart_convergence}
%\mathrm{E}_{\theta_0}{\mathrm{P}(\mathcal{I}=I|X)}=
%\lim_{m\to\infty}\mathrm{E}_{\theta_0} \mathrm{P}_n(\mathcal{I}=I|X),
%\end{eqnarray}
%where $\mathrm{P}_n(\mathcal{I}=I|X)$ is defined by (\ref{P_n}).
For any $I, I_0 \in \mathbb{N}$ and any $h \in [0,1]$,  we have 
\begin{align}
\label{th1_lem3}
\mathrm{E}_{\theta_0}\mathrm{P}(\mathcal{I}=I|X)
\le \mathrm{E}_{\theta_0}\Big[\frac{\lambda_I 
\bigotimes_i  \varphi(X_i,X_i(I),\tau_i^2(I)+\sigma_i^2)}
{ \lambda_{I_0} \bigotimes_i \varphi(X_i,X_i(I_0),\tau_i^2(I_0)
+\sigma_i^2)} \Big]^h.
%= \frac{e^{-\alpha I} (K+1)^{-I/2} 
%\exp\big\{ -\sum_{i=I+1}^n \frac{X_i^2}{2\sigma_i^2}\big\}} 
%{\sum_J e^{-\alpha J}  (K+1)^{-(J\wedge n)/2} 
%\exp\big\{-\sum_{i=J+1}^n \frac{X_i^2}{2\sigma_i^2}\big\}}.
\end{align}
Recall the elementary identity:
for $Y\sim N(\mu,\sigma^2)$ and $b>-\sigma^{-2}$,
\begin{eqnarray}
\label{element_ineq}
\mathrm{E}\big(\exp\{- bY^2/2\big\}\big) 
%= (1+b\sigma^2)^{-1/2} \exp\Big\{-\dfrac{\mu^2 b}{2(1+b\sigma^2)}\Big\}
=\exp\Big\{-\dfrac{\mu^2 b}{2(1+b\sigma^2)} 
-\frac{1}{2} \log(1+b\sigma^2)\Big\}.
\end{eqnarray}
Using (\ref{th1_lem3}) and (\ref{element_ineq}) with $h=1$,
we derive that, for any $I,I_0 \in \mathbb{N}$ such that $I<I_0$,  
\begin{align}
\mathrm{E}_{\theta_0}\mathrm{P}(\mathcal{I}=I|X) 
&\le\mathrm{E}_{\theta_0}
\frac{\lambda_I (K+1)^{-I/2} 
\exp\big\{ -\sum_{i=I+1}^\infty \frac{X_i^2}{2\sigma_i^2}\big\}} 
{\lambda_{I_0}  (K+1)^{-I_0/2} 
\exp\big\{-\sum_{i=I_0+1}^\infty \frac{X_i^2}{2\sigma_i^2}\big\}}\notag\\
&= e^{\alpha(I_0-I)}(K+1)^{(I_0-I)/2} 
\mathrm{E}_{\theta_0}
\exp\Big\{-\frac{1}{2}\sum_{I+1}^{I_0} \frac{X_i^2}{\sigma_i^2}\Big\}\notag\\
\label{rel_I<I_o}
&=e^{-(\alpha+a_k)I}\exp\Big\{(\alpha+a_K)I_0-
\frac{1}{4}\sum_{I+1}^{I_0}\frac{\theta_{0,i}^2}{\sigma_i^2}\Big\},
\end{align}
where $a_K= \frac{1}{2} \log(\frac{K+1}{2})$.
%Now apply (\ref{th1_lem3}) with $h=1/2$ and (\ref{element_ineq}) to the case $I>I_0$:
%\begin{align}
%\mathrm{E}_{\theta_0}\mathrm{P}(\mathcal{I}=I|X) 
%&\le e^{\alpha (I_0-I)/2} (K+1)^{(I_0-I)/4} 
%\mathrm{E}_{\theta_0}
%\exp\Big\{\frac{1}{4}\sum_{I+1}^{I_0} \frac{X_i^2}{\sigma_i^2}\Big\}\notag\\
%\label{rel_I>I_o}
%&= \exp\Big\{-\Big[\frac{\alpha}{2}+\frac{1}{4}\log\Big(\frac{K+1}{4}\Big) \Big] (I-I_0)
%+\frac{1}{2}\sum_{I_0+1}^{I} \frac{\theta_{0,i}^2}{\sigma_i^2}
%\Big\}.
%\end{align}
Now apply (\ref{th1_lem3}) and (\ref{element_ineq}) to the case $I>I_0$:
for any $h\in[0,1)$,
\begin{align}
&\mathrm{E}_{\theta_0}\mathrm{P}(\mathcal{I}=I|X) 
\le e^{\alpha h(I_0-I)} (K+1)^{(I_0-I)h/2} 
\mathrm{E}_{\theta_0}
\exp\Big\{\frac{h}{2}\sum_{i=I_0+1}^{I} \frac{X_i^2}{\sigma_i^2}\Big\}= \notag\\
\label{rel0_I>I_o}
& e^{-\alpha h I/2}\exp\Big\{-\frac{\alpha h I}{2}+\alpha hI_0- b_{K,h}(I-I_0)
+\frac{h}{2(1-h)}\sum_{i=I_0+1}^{I}\frac{\theta_{0,i}^2}{\sigma_i^2}\Big\},
\end{align}
where $b_{K,h}=\frac{h}{2}\log(K+1) +\frac{1}{2} \log(1-h)$. 
Clearly, $b_{K,h} >0$ if $K>(1-h)^{-1/h}-1$. 
Now take $h=0.1$ in (\ref{rel0_I>I_o}), 
then $b_{K,0.1}=\frac{1}{20}\log(K+1)+\frac{1}{2}\log(0.9)>0$ 
since $K\ge 1.87>(10/9)^{10}-1$ by the condition of the theorem.
Thus, for any $I,I_0 \in \mathbb{N}$ such that $I>I_0$, we derive  
\begin{align}
\label{rel_I>I_o}
\mathrm{E}_{\theta_0}\mathrm{P}(\mathcal{I}=I|X) 
&\le 
%e^{-\alpha I/20}\exp\Big\{-\frac{\alpha I}{20}+\frac{\alpha I_0}{10} 
%+\frac{1}{18}\sum_{i=I_0+1}^{I}\frac{\theta_{0,i}^2}{\sigma_i^2}\Big\}= 
e^{-\alpha I/20}\exp\Big\{-\frac{\alpha}{20} \Big(I - 2I_0 
-\frac{10}{9\alpha}\sum_{i=I_0+1}^{I}\frac{\theta_{0,i}^2}{\sigma_i^2}  \Big)
\Big\}.
\end{align}

\paragraph{Step 2: a bound by the sum of three terms}
Recall $\mathnormal{r}^2(I,\theta_0)=\sum_{i\le I}\sigma_i^2
+\sum_{i> I} \theta_{0,i}^2$ and $\mathnormal{r}^2(\theta_0)=\mathnormal{r}^2(I_o,\theta_0)=
\min_I \mathnormal{r}^2(I,\theta_0)$. Notice that 
\begin{align}
\label{risk_expansion}
\mathnormal{r}^2(I, \theta_0) \le \mathnormal{r}^2(\theta_0) + 
\mathrm{1} \{I \le I_o\} \sum_{i=I+1}^{I_o} \theta_{0,i}^2+ 
\mathrm{1} \{I>I_o\}  \sum_{i=I_o+1}^I \sigma_i^2.
\end{align}
Next, as $\mathrm{P}_I (\cdot|X)=\bigotimes_i N(X_i\mathrm{1}\{i\in \mathbb{N}_I\}, 
L\sigma_i^2\mathrm{1}\{i\le I\})$ with $L=\frac{K}{K+1} \le 1$, we obtain 
by applying the Markov inequality that
\begin{align}
\mathrm{P}_I (\|\theta - \theta_0\| &\ge M \mathnormal{r}(\theta_0)|X)
\le \frac{ \mathrm{E}_I (\|\theta - \theta_0\|^2 |X)}{M^2 \mathnormal{r}^2(\theta_0)} 
\notag \\ 
&=
\frac{L \sum_{i\le I} \sigma_i^2+\sum_{i>I}  \theta_{0,i}^2
+\sum_{i \le I} (X_i -\theta_{0,i})^2}{M^2 \mathnormal{r}^2(\theta_0)} 
\notag\\ &
\label{upsilon_bound1}
\le 
\frac{\mathnormal{r}^2(I,\theta_0)+\sum_{i \le I} \sigma_i^2\xi_i^2}
{M^2 \mathnormal{r}^2(\theta_0)} \triangleq \upsilon_I,
\end{align}
where $\xi_i=\sigma_i^{-1}(X_i -\theta_{0,i}) 
\stackrel{\tiny \rm ind}{\sim}N(0,1)$  
from the $\mathrm{P}_{\theta_0}$-perspective.
Denote for brevity $p_I=\mathrm{P}(\mathcal{I}=I|X)$,  so that 
$p_I \in [0,1]$ and $\sum_I p_I =1$.
In view of  (\ref{ddm1}) and (\ref{upsilon_bound1}),
\begin{align}
\mathrm{P}(\|\theta - \theta_0\| \ge M \mathnormal{r}(\theta_0)|X)
&\le  \sum_I  \upsilon_I p_I 
\label{sum_of_Ts}
=T_1 +T_2 +T_3,
\end{align}
where $T_1=\sum_{I \le I_o}  \upsilon_I p_I$, $T_2=
\sum_{I_o <I \le \tau I_o}  \upsilon_I p_I$, $T_3=\sum_{I> \tau I_o}  \upsilon_I p_I $,
and $\tau>2$ to be chosen later.

\paragraph{Step 3: handling the term $T_1$}
For $\tau_1> 0$ to be chosen later, introduce the sets 
\begin{align*}
\mathcal{O}^-&=\mathcal{O}^-(\tau_1,\theta_0)=
\Big\{ I: \, I \le I_o, \, \sum_{i=I+1}^{I_o} \theta_{0,i}^2 \le \tau_1\sum_{i=1}^{I_o}  \sigma_i^2\Big\}, \\
\mathcal{N}^-&=\mathcal{N}^-(\tau_1,\theta_0)=
\Big\{ I: \, I \le I_o, \, \sum_{i=I+1}^{I_o} \theta_{0,i}^2 > \tau_1\sum_{i=1}^{I_o}  \sigma_i^2\Big\}.
%\{I: \, I \le I_o\} \setminus \mathcal{O}^-(\tau_1).
\end{align*}
By (\ref{risk_expansion}), $\max_{I \in\mathcal{O}^-} \mathnormal{r}^2(I,\theta_0) 
\le (1+\tau_1) \mathnormal{r}^2(\theta_0)$. This and (\ref{upsilon_bound1}) imply
\begin{align}
\label{rel_T1a}
\mathrm{E}_{\theta_0} \sum_{I\in \mathcal{O}^-} \upsilon_I p_I   \le 
\mathrm{E}_{\theta_0} \max_{I \in\mathcal{O}^-} \upsilon_I  \le
\frac{1+\tau_1}{M^2} +
\frac{\mathrm{E}_{\theta_0}\sum_{i\le I_o}\sigma_i^2 \xi_i^2 }{M^2 \mathnormal{r}^2(\theta_0)} 
\le \frac{2+\tau_1}{M^2}.
\end{align}

The property (i) from (\ref{condition_sigma}) yields that
$I_o \le \frac{K_1}{\sigma_{I_o}^2}\sum_{i=1}^{I_o}\sigma_i^2$. 
Besides, for each $I \in \mathcal{N}^-$,
$\sum_{i=1}^{I_o}  \sigma_i^2< \tau_1^{-1} \sum_{i=I+1}^{I_o}\theta_{0,i}^2$.
Now set $\tau_1 = \frac{4(\alpha+a_K)K_1}{5}$. %($a_K= \frac{1}{2} \log\big(\frac{K+1}{2}\big)>0$). 
The last two relations and (\ref{rel_I<I_o}) imply that, for each $I \in \mathcal{N}^-$,
\begin{align}
\mathrm{E}_{\theta_0} p_I &= \mathrm{E}_{\theta_{0}}{\mathrm{P}(\mathcal{I}=I|X)} \notag 
\le 
e^{-(\alpha+a_K) I} \exp\Big\{(\alpha+a_K)I_o-
\frac{1}{4}\sum_{I+1}^{I_0}\frac{\theta_{0,i}^2}{\sigma_i^2}\Big\} \\
&\le 
e^{-(\alpha+a_K) I} \exp\Big\{\frac{(\alpha+a_K)K_1}{\sigma_{I_o}^2}
\sum_{i=1}^{I_o}\sigma_i^2-\frac{1}{4\sigma_{I_o}^2}\sum_{i=I+1}^{I_o} 
\theta^2_{i,0} \Big\} \notag\\
\label{rel_Ep_I}
& 
\le e^{-(\alpha+a_K)I}  
\exp\Big\{- \frac{1}{\sigma_{I_o}^{2}}\sum_{i=I+1}^{I_o} 
\theta_{0,i}^2 \Big\}.
\end{align}
Using the fact that $\max_{x\ge 0}\{xe^{-cx} \}\le (ce)^{-1}$ (for any $c>0$), 
 (\ref{risk_expansion}), (\ref{upsilon_bound1}) and (\ref{rel_Ep_I}), we obtain
\begin{align*}
&\mathrm{E}_{\theta_0} \sum_{I\in \mathcal{N}^-} \upsilon_I p_I  
\le 
\mathrm{E}_{\theta_0}  \sum_{I\in \mathcal{N}^-}
\frac{\mathnormal{r}^2(\theta_0) +\sum_{i=I+1}^{I_o} 
\theta_{0,i}^2+\sum_{i \le I} \sigma_i^2\xi_i^2}
{M^2 \mathnormal{r}^2(\theta_0)} p_I\\
&\le
\frac{1}{M^2}+
\frac{\mathrm{E}_{\theta_0} \sum_{i \le I_o} \sigma_i^2\xi_i^2}
{M^2 \mathnormal{r}^2(\theta_0)} +
 \sum_{I\in \mathcal{N}^-}
\frac{\big(\sum_{i=I+1}^{I_o}\theta_{0,i}^2\big)\mathrm{E}_{\theta_0}  p_I}
{M^2 \mathnormal{r}^2(\theta_0)}\\
&\le
\frac{2}{M^2} 
%\frac{ \sum_{i \le I_o}\sigma_i^2}{M^2 \mathnormal{r}^2(I_o,\theta_0)}\\& \qquad 
+\sum_{ I \in\mathcal{N}^-}  
\frac{\big(\sum_{i=I+1}^{I_o} 
\theta_{0,i}^2\big)\exp\big\{-\sigma_{I_o}^{-2}\sum_{i=I+1}^{I_o} 
\theta_{0,i}^2\big\} e^{-(\alpha+a_K)I}}{M^2 \mathnormal{r}^2(\theta_0)}  
\\
&\le \frac{2}{M^2}+\sum_{ I \in \mathcal{N}^-}
\frac{e^{-(\alpha+a_K)I} e^{-1}\sigma_{I_o}^2}{M^2 \mathnormal{r}^2(\theta_0)} 
\le
\frac{2}{M^2} + \frac{e^{-1}}{M^2} \sum_I e^{-(\alpha+a_K)I}=
\frac{C_1} %2+e^{-1} (1-e^{-(\alpha+a_k)})^{-1}}
{M^2}. 
\end{align*}
where $C_1 = 2+\frac{e^{-(1+\alpha+a_K)}}{1-e^{-(\alpha+a_K)}}$, $\tau_1 = \frac{4(\alpha+a_K)K_1}{5}$.
The last relation and (\ref{rel_T1a}) give 
\begin{align}
\label{term_T1}
\mathrm{E}_{\theta_0} T_1 
=\mathrm{E}_{\theta_0} \sum_{I\in \mathcal{O}^-(\tau_1,\theta_0)} \upsilon_I p_I +
\mathrm{E}_{\theta_0} \sum_{I\in \mathcal{N}^-(\tau_1,\theta_0)} \upsilon_I p_I
\le \frac{C_2}{M^2},
%\mathrm{E}_{\theta_0}\frac{\sum_{i\le I_o} \sigma_i^2 \xi_i^2 }{M^2 \mathnormal{r}^2(I_o, \theta_0)} + 
%\frac{\sum_{I\in\mathcal{O}^-(\tau_1)} \mathnormal{r}^2(I, \theta_0)}{M^2 \mathnormal{r}^2(I_o, \theta_0)}
%+\frac{\sum_{I\in\mathcal{N}^-(\tau_1)} \mathnormal{r}^2(I, \theta_0)}{M^2 \mathnormal{r}^2(I_o, \theta_0)}.
\end{align}
where $C_2 = 2+\tau_1+C_1= 4+\frac{4(\alpha+a_K)K_1}{5} + 
\frac{e^{-(1+\alpha+a_K)}}{1-e^{-(\alpha+a_K)}}$.

\paragraph{Step 4: handling the term $T_2$}
Since $p_I \in [0,1]$ and $\sum_I p_I =1$, 
$\mathrm{E}_{\theta_0} T_2 =\mathrm{E}_{\theta_0} \sum_{I_o <I \le \tau I_o}  \upsilon_I p_I
\le \mathrm{E}_{\theta_0} [\max_{I_o <I \le \tau I_o} \upsilon_I ]$.
Using this, (\ref{risk_expansion}), (\ref{upsilon_bound1}), (\ref{sum_of_Ts}) and
the property (ii) from (\ref{condition_sigma}), we get
\begin{align}
&\mathrm{E}_{\theta_0} T_2 
%= \mathrm{E}_{\theta_0} \!\!\! \sum_{I_o <I \le \tau I_o} \!\!\! \upsilon_I p_I
\le \mathrm{E}_{\theta_0} \max_{I_o <I \le \tau I_o} \upsilon_I  %\notag\\&
\le
\frac{\max_{I_o <I \le \tau I_o}\mathnormal{r}^2(I,\theta_0)+
\mathrm{E}_{\theta_0} \sum_{i \le \tau I_o} \sigma_i^2\xi_i^2}
{M^2 \mathnormal{r}^2(\theta_0)}  \notag \\
%&\le 
%\frac{\mathnormal{r}^2(I_o,\theta_0) +\sum_{I_o <I \le \tau I_o}
%\sigma_i^2+\sum_{i \le \tau I_o} \sigma_i^2}
%{M^2 \mathnormal{r}^2(I_o,\theta_0)} \notag\\
&\le 
\label{term_T2}
\frac{1}{M^2}+
\frac{2\sum_{i\le \tau I_o}\sigma_i^2}
{M^2 \mathnormal{r}^2(\theta_0)} %\notag\\ &
\le\frac{1}{M^2}+\frac{2 K_2(\tau) \sum_{i=1}^{I_o}\sigma_i^2}
{M^2 \mathnormal{r}^2(\theta_0)} 
\le \frac{1+2K_2(\tau)}{M^2}.
\end{align}

\paragraph{Step 5: handling the term $T_3$}
For some $\tau_2>0$ to be chosen later, introduce the sets
\begin{align*}
\mathcal{O}^+
&=\mathcal{O}^+(\tau,\tau_2,\theta_0)=
\Big\{ I\in \mathbb{N}: \, I >\tau I_o, \, \sum_{i=I_o+1}^I \sigma_i^2 
\le \tau_2\sum_{i > I_o}  \theta_{0,i}^2\Big\}, \\
\mathcal{N}^+
&= %\{I: \, I> \tau I_o\} \setminus\mathcal{O}^+(\tau,\tau_2)=
\mathcal{N}^+(\tau,\tau_2,\theta_0)= 
\Big\{I\in\mathbb{N}: \, I>\tau I_o, \, \sum_{i=I_o+1}^I \sigma_i^2 > \tau_2\sum_{i > I_o}  \theta_{0,i}^2\Big\}.
\end{align*}
%Things are OK for  $I \in\mathcal{O}^+(\tau,\tau_2)$. %By the oracle definition, we have 
%$\sum_{i=I_o+1}^I \sigma_i^2 \ge \sum_{i=I_o+1}^I \theta_i^2.$
%By using this, we obtain 
By (\ref{risk_expansion}), 
$\max_{I \in\mathcal{O}^+} \mathnormal{r}^2(I,\theta_0) 
\le (1+\tau_2) \mathnormal{r}^2(\theta_0)$.  
Let $I^+ = \max\{\mathcal{O}^+\}$, then 
$\sum_{i\le I^+}\sigma_i^2 \le \sum_{i=1}^{I_o} \sigma_i^2 +\tau_2\sum_{i > I_o}  \theta_{0,i}^2\le 
(1\vee \tau_2) \mathnormal{r}^2(\theta_0)$. 
In view of (\ref{upsilon_bound1}), the last two relations entail that
\begin{align}
\mathrm{E}_{\theta_0} \sum_{I\in \mathcal{O}^+} \upsilon_I p_I   
&\le \mathrm{E}_{\theta_0} \max_{I \in\mathcal{O}^+} \upsilon_I  
%\notag\\&
\le 
\label{rel_O^+}
\frac{1+\tau_2}{M^2} +
\frac{\mathrm{E}_{\theta_0}\sum_{i\le I^+}\sigma_i^2 \xi_i^2 }
{M^2 \mathnormal{r}^2(\theta_0)} 
\le \frac{2(1+\tau_2)}{M^2}.
\end{align}

Let $K_4$ and $\tau>2$ be from property (iv) of (\ref{condition_sigma}),
then $\Sigma(m) - \Sigma(\lfloor m/\tau \rfloor ) \ge K_4\Sigma(m)$ for any $m\ge \tau$.
This entails that, for each $I \in\mathcal{N}^+$,
\begin{align*}
\sum_{i=\lfloor I/\tau \rfloor +1}^I \!\!\! \sigma_i^2 &\ge 
K_4 \sum_{i=1}^I \sigma_i^2  % \Sigma(I) 
\ge K_4\sum_{i=I_o+1}^I \sigma_i^2  %\notag\\& 
\ge
%\label{rel_sigmas}
K_4\tau_2\sum_{i=I_o+1} ^I\theta_i^2 \ge 
K_4\tau_2 \sum_{i=\lfloor I/\tau \rfloor +1}^I \theta_i^2.
\end{align*}
For each $I \in\mathcal{N}^+$, take $I_0=I_0(I)=
\lfloor I/\tau \rfloor$, then apply the property (v) of (\ref{condition_sigma})
and the last inequality with $\tau_2 = \frac{10\tau}{9\alpha(\tau-2)K_4K_5(\tau)}$
to derive
\begin{align*}
I &- 2 I_0- \frac{10}{9\alpha}\sum_{i=I_0+1} ^I \frac{\theta_{0,i}^2}{\sigma_i^2} \ge 
(1-\tfrac{2}{\tau}) I - \frac{10}{9\alpha}\sum_{i=I_0+1} ^I \frac{\theta_{0,i}^2}{\sigma_i^2}\\
&\ge (1-\tfrac{2}{\tau})  K_5  
\sum_{i=I_0+1}^I \frac{\sigma_i^2}{\sigma_{I_0}^2} - 
\frac{10}{9\alpha}\sum_{i=I_0+1}^I \frac{\theta_{0,i}^2}{\sigma_{I_0}^2}  \\
& \ge 
\frac{(\tau-2)K_5(\tau) K_4\tau_2}{\tau} \sum_{i=I_0+1}^I \frac{\theta_{0,i}^2}{\sigma_{I_0}^2} 
- \frac{10}{9\alpha}\sum_{i=I_0+1}^I \frac{\theta_{0,i}^2}{\sigma_{I_0}^2} =0.
\end{align*}
The last relation and the bound (\ref{rel_I>I_o}) with $I_0=\lfloor I/\tau \rfloor$ imply  
that 
\begin{align}
\label{rel_Ep_I2}
\mathrm{E}_{\theta_0} p_I=\mathrm{E}_{\theta_0}\mathrm{P}(\mathcal{I}=I|X) 
\le e^{- \gamma I},
\;\; I \in\mathcal{N}^+, \;\; %\text{where} \;\; 
\gamma= \frac{\alpha}{20}.
\end{align}

%Consider now the case $I \in \mathbb{N}^c _{I_o}=\{I:\, I > I_o\}$.  
%Recall 
%\[
%\mathcal{N}_o=\mathcal{N}_o(a_1,a_2) = \big\{ I\in \mathbb{N}: \, 
%I_o\le I \le\big[ I_o \vee \big(a_1\bar{I}_o+a_2\sum_{i=\bar{I}_o+1}^{I} 
%\frac{\theta_{0,i}^2}{\sigma_i^2}\big)\big] \big\}.
%\]
%Let $\mathcal{N}^+(a,b) =\mathbb{N}^c _{I_o} \setminus \mathcal{N}_o$.
Since $p_I \in[0,1]$ and
$\mathrm{E} \Big[\sum_{i=1}^m \sigma_i^2 \xi_i^2\Big]^2 \le 
3 \Big[\sum_{i=1}^m \sigma_i^2\Big]^2$  for any $m\in\mathbb{N}$,  
we obtain by the Cauchy-Schwartz inequality that 
\begin{align}
\label{rel_sigma_xi}
\mathrm{E}_{\theta_0} \big[p_I \sum_{i\le I} \sigma_i^2 \xi_i^2 \big] 
\le 
\big(\mathrm{E} _{\theta_0} p^2_I\big)^{1/2} \sqrt{3}\sum_{i \le I} \sigma_i^2 
\le \sqrt{3} \big(\mathrm{E} _{\theta_0} p_I\big)^{1/2}\sum_{i\le I} \sigma_i^2 .
\end{align}
%Since $\mathrm{E}_{\theta_0} p_I \le e^{- c I}$, with some $c>0$,
%for each $I \in  \mathcal{N}^+(\tau,\tau_2)$ and  
%$\sum_I \big(\sum_{i=1}^I \kappa_i^2\big) e^{-cI/2} \le C$,
%\begin{align*}
%\mathrm{E}_{\theta_0}\sum_{I\in \mathcal{N}^+(\tau,\tau_2)} 
%p_I \sum_{i\le I} \sigma_i^2 \xi_i^2 
%&\le  
%\sum_{I\in \mathcal{N}^+(\tau,\tau_2)}\!\!\!\!\big(\mathrm{E}_{\theta_0} p_I^2 \big)^{1/2}
%\Big[\mathrm{E}_{\theta_0}\Big(\sum_{i \le I} \sigma_i^2\xi_i^2 \Big)^2\Big]^{1/2} 
%\\
%&\le 
%\sqrt{3} \sum_{I\in \mathcal{N}^+(\tau,\tau_2)} \big(\sum_{i=1}^I \sigma_i^2 \big)
%\big(\mathrm{E}_{\theta_0} p_I\big)^{1/2} \le C \varepsilon^2.
%\end{align*}
Combining (\ref{risk_expansion}), (\ref{upsilon_bound1}), (\ref{rel_Ep_I2}),
(\ref{rel_sigma_xi}), the property (iii) of (\ref{condition_sigma}) and 
the fact that $\varepsilon^2=\sigma_1^2 \le \mathnormal{r}^2(\theta_0)$, we derive
%we obtain
\begin{align*}
&\mathrm{E}_{\theta_0}\sum_{I\in \mathcal{N}^+} p_I \upsilon_I
=
\sum_{I\in \mathcal{N}^+(\tau,\tau_2)}  
\frac{\mathnormal{r}^2(I,\theta_0) \mathrm{E}_{\theta_0} p_I +  
\mathrm{E}_{\theta_0} \big[p_I \sum_{i\le I} \sigma_i^2 \xi_i^2\big] }
{M^2 \mathnormal{r}^2(\theta_0)}\\
&\le
\frac{1}{M^2} + \sum_{I\in \mathcal{N}^+} 
\frac{\big(\sum_{i=I_o+1}^I \sigma_i^2\big) \mathrm{E}_{\theta_0} p_I +%C\varepsilon^2
\sqrt{3}\big(\sum_{i\le I} \sigma_i^2 \big) 
\big(\mathrm{E}_{\theta_0} p_I\big)^{1/2}}{M^2 \mathnormal{r}^2(\theta_0)} \\
&\le \frac{1}{M^2}+ 
\sum_{I\in \mathcal{N}^+} 
\frac{ \varepsilon^2 \big[\big(\sum_{i=I_o+1}^I \kappa_i^2\big) e^{- \gamma  I} +%C\varepsilon^2
\sqrt{3}\big(\sum_{i\le I} \kappa_i^2 \big) e^{-\gamma I/2} \big] }
{M^2 \mathnormal{r}^2(\theta_0)} \\
&\le \frac{1+K_3(\gamma) +\sqrt{3} K_3(\gamma/2)}{M^2}.  
\end{align*}
Finally, the last relation and (\ref{rel_O^+}) entail  the bound
\begin{align}
\label{term_T3}
\mathrm{E}_{\theta_0} T_3 
=\mathrm{E}_{\theta_0} \sum_{I\in \mathcal{O}^+(\tau,\tau_2,\theta_0)} \upsilon_I p_I +
\mathrm{E}_{\theta_0} \sum_{I\in \mathcal{N}^+(\tau,\tau_2,\theta_0)} \upsilon_I p_I
\le \frac{C_3}{M^2},
\end{align}
where $C_3 = 2(1+\tau_2) +1+K_3(\gamma) +\sqrt{3} K_3(\gamma/2)$.

\paragraph{Step 6: finalizing the proof}
Piecing together the relations  (\ref{sum_of_Ts}), (\ref{term_T1}), 
(\ref{term_T2}) and (\ref{term_T3}), we finally obtain
\[
\mathrm{E}_{\theta_0} \mathrm{P}(\|\theta - \theta_0\| 
\ge M \mathnormal{r}(\theta_0)|X) 
\le 
\mathrm{E}_{\theta_0} (T_1 + T_2 + T_3) \le  \frac{C_{or}}{M^2}.
\]
The constant $C_{or}=C_{or}(K,\alpha)$ is as follows:
\begin{align*}
C_{or} &=  
C_2 + 1+2K_2(\tau) +2(1+\tau_2)+1+K_3(\gamma) +\sqrt{3} K_3(\gamma/2),
\end{align*}
where $C_2 = 4+ \frac{4(\alpha+a_K)K_1}{5}+\frac{e^{-(1+\alpha+a_K)}}{1-e^{-(\alpha+a_K)}}$, 
$a_K= \frac{1}{2} \log\big(\frac{K+1}{2}\big)$,
$\tau_2 = \frac{10\tau}{9\alpha(\tau-2)K_4K_5(\tau)}$, $\gamma = \frac{\alpha}{20}$,
the constants  $\tau, K_1, K_2, K_3, K_4, K_5$ are from (\ref{condition_sigma}).

\subsection{Proof of Theorem \ref{th2a_posterior}} 
%We prove Theorem \ref{th2a_posterior} in several steps.

\paragraph{Step 1: first technical lemma}
%Here we give the proof of Lemma \ref{lem_oversmooth} used later.
\begin{lemma} 
\label{lem_oversmooth}
Let DDM $\mathrm{P}_{K,\alpha}(\mathcal{I}=I|X)$ be given by (\ref{p(I|X)}) 
with parameters $K,\alpha>0$ chosen in such a way that 
$a(K)>\alpha$, with  $a(K)$ defined by (\ref{rho}). 
Let $\varkappa_0=\varkappa_0(K,\alpha)= \frac{a(K)-\alpha}{a(K)}$. 
Then for any $\theta_0\in\ell_2$ and any $\varkappa \in [0,\varkappa_0)$
\begin{align}
\label{I<}
\mathrm{E}_{\theta_0} 
 \mathrm{P}(\mathcal{I}\le \varkappa  \bar{I}_o|X)
 \le C \exp\big\{- c \bar{I}_o\big\},
 \end{align}
where $c=a(K)(1-\varkappa)-\alpha>0$, 
$C=C_\alpha^{-1}=(e^\alpha-1)^{-1}$, 
and $\bar{I}_o=\bar{I}_o(\theta_0)$ is defined by (\ref{surr_oracle}).
%There exists a constant $\alpha_0 = \alpha_0(K)$ such that for 
%any $\alpha\in (0,\alpha_0]$ there exist constants 
%$\varkappa \in (0,1)$, $C,c>0$ (all depending only on $K$ and $\alpha$) such that,
%for all $\theta_0\in\ell_2$,
%\begin{align}
%\label{I<}
%\mathrm{E}_{\theta_0} \mathrm{P} \big(\mathcal{I} \le \varkappa I_o(\theta_0)|X \big)
%\le C \exp\{-c I_o(\theta_0)\}. %, \quad \text{uniformly in} \quad \theta_0\in\ell_2.
%\end{align}
\end{lemma}
\begin{proof}[Proof of Lemma \ref{lem_oversmooth}]
By the definition (\ref{surr_oracle}) of the surrogate oracle,  
$\mathnormal{R}^2(I,\theta_0) \ge \mathnormal{R}^2(\bar{I}_o,\theta_0)$
for any $\theta_0\in \ell_2$. For $I<\bar{I}_o$, this implies that
$\sum_{i=I+1}^{\bar{I}_o} \frac{\theta^2_{0,i}}{\sigma_i^2}\ge \bar{I}_o-I$. 
Using this, we obtain that for $I\le \varkappa \bar{I}_o$  
\begin{align*}
\frac{1}{4}\sum_{i=I+1}^{\bar{I}_o} \frac{\theta^2_{0,i}}{\sigma^2_i}
&- \frac{1}{2}\log\Big[\frac{K+1}{2}\Big](\bar{I}_o-I) 
\ge \Big(\frac{1}{4} - \frac{1}{2}\log\Big[\frac{K+1}{2}\Big] \Big) (\bar{I}_o-I) \\
&=a(K)(\bar{I}_o-I)
\ge a(K) (1-\varkappa)\bar{I}_o.
\end{align*}
%with $a(K)=\frac{1}{4} - \frac{1}{2}\log[\frac{K+1}{2}]$. 
The lemma follows from the last relation, (\ref{rel_I<I_o}) and the fact that 
$\sum_I \lambda_I=1$:
\begin{align*}
&\mathrm{E}_{\theta_0}\mathrm{P}(\mathcal{I}\le \varkappa \bar{I}_o |X)
\le\sum_{I\le \varkappa I_o} \frac{\lambda_I}{\lambda_{\bar{I}_o}}
\exp\Big\{- \frac{1}{4}\sum_{i=I+1}^{\bar{I}_o} \frac{\theta^2_{0,i}}{\sigma^2_i}
- \frac{1}{2}\log\Big[\frac{K+1}{2}\Big](\bar{I}_o-I) \Big\} \\
&\le 
\sum_{I\le \varkappa \bar{I}_o} \frac{ \lambda_I}{\lambda_{\bar{I}_o}} 
\exp\{-(a(K)(1-\varkappa))\bar{I}_o\}  
%\le C_\alpha^{-1} \sum_{I\le \varkappa I_o} \lambda_I \exp\{-(A_K-\alpha)h I_o\}
\le \tfrac{1}{C_\alpha} \exp\big\{-(a(K)(1-\varkappa) -\alpha) \bar{I}_o\big\}.  \qedhere
\end{align*} 
\end{proof}

\paragraph{Step 2:  second technical lemma}
%Next we provide one basic technical result.
\begin{lemma}
\label{lem_bound_volume}
Let $\Lambda(S)$ be the Lebesgue measure (or volume) of a bounded 
set $S\subset \mathbb{R}^k$, $k\in\mathbb{N}$, and $B_k(r)=
\{x\in \mathbb{R}^k:\, \|x\|\le r\}$ (here $\|\cdot\|$ is the 
usual Euclidean norm in $\mathbb{R}^k$) be the Euclidean ball of radius 
$r$ in space $\mathbb{R}^k$. Then
\[
\Lambda(B_k(r))  \le e \pi^{-1/2} r^k k^{-(k+1)/2} (2\pi e)^{k/2}.
\]
\end{lemma}
\begin{proof}[Proof of Lemma \ref{lem_bound_volume}]
By using Stirling's approximation for the Gamma function 
$\Gamma(x)=\sqrt{2\pi}\,  x^{x-1/2} e^{-x+\varsigma/(12 x)}$ 
for all $x\ge 1$  and some $0\le \varsigma\le C$, we derive  
\begin{align*}
\Gamma\big(1+\tfrac{k}{2}\big)&=
\sqrt{2\pi} \big(1+\tfrac{k}{2}\big)^{\frac{k+1}{2}} 
e^{-1-\frac{k}{2}+\frac{\varsigma}{6k+12}}
=\frac{(1+\frac{2}{k})^{(k+1)/2}\sqrt{\pi}}{e^{1-\varsigma/(6k+12)}}  k^{\frac{k+1}{2}} (2e)^{-\frac{k}{2}}\\
&=c_k k^{(k+1)/2} (2e)^{-k/2} \ge  e^{-1}\pi^{1/2}k^{(k+1)/2} (2e)^{-k/2},
\end{align*}
because $c_k=\frac{(1+2/k)^{(k+1)/2}\sqrt{\pi}}{e^{1-\varsigma/(6k+12)}}> \frac{\sqrt{\pi}}{e}$.
%(Actually, $c_k \to \sqrt{\pi}$ as $k\to \infty$ and more accurate estimates $c_k>c$ are possible.)
Combining the last relation with the well known fact that
$\Lambda(B_k(r)) = r^k \Lambda(B_k(1))=\frac{r^k\pi^{k/2}}{\Gamma(1+k/2)}$ 
completes the proof of the lemma.
\end{proof}

\paragraph{Step 3:  small ball bound for $\mathrm{P}_I(\cdot|X)$}
Recall that, with $L=K/(K+1)$,
\[
\mathrm{P}_I(\theta|X) =\bigotimes_i N(X_i\mathrm{1}\{i \le I\}, 
L\sigma_i^2\mathrm{1}\{i\le I\}),  \quad I\in \mathbb{N}.
\]
We have that $\Sigma(I) = \sum_{i=1}^I \sigma_i^2 \le \varepsilon^2 \frac{(2I)^{2p+1}}{2p+1}$.
By Stirling's bound, $ \prod_{i=1}^I \kappa_i=(I!)^p \ge \big((I/e)^I\sqrt{2\pi I}\big)^p$.
Let $Z_1,\ldots, Z_I$ be independent $N(0,1)$ random variables.
Using these relations, Anderson's inequality and Lemma \ref{lem_bound_volume}, 
we obtain that, $\mathrm{P}_{\theta_0}$-almost surely,
\begin{align}
&\mathrm{P}_I\big(\|\theta -\hat{\theta}\|
\le \delta\Sigma^{1/2}(\bar{I}_o)| X) =
\mathrm{P}_I\big(\|\theta -\hat{\theta}\|^2
\le \delta^2\Sigma(\bar{I}_o)| X) \notag\\
&=\mathrm{P}\Big(\sum_{i\le I} (X_i+ \sigma_i \sqrt{L}Z_i  - \hat{\theta}_i)^2 +
\sum_{i>I} \hat{\theta}_i^2 \le \delta^2 \Sigma(\bar{I}_o) |X\Big) \notag\\
&\le 
\mathrm{P}\Big(L\sum_{i\le I} \sigma_i^2 Z_i^2  
\le \delta^2  \Sigma(\bar{I}_o)\Big)
\le
\mathrm{P}\Big(\sum_{i\le I} \sigma_i^2 Z_i^2 \le 
\frac{\delta^2 \Sigma(\bar{I}_o)}{L} \Big) \notag\\
&\le 
\frac{\Lambda\big(B_I(\delta
 \sqrt{\Sigma(\bar{I}_o)/L})\big)} {\prod_{i=1}^I (2\pi \sigma_i^2)^{1/2}} \le
\frac{(2\pi)^{-I/2}}{\prod_{i=1}^I \varepsilon \kappa_i} 
 \frac{e}{\sqrt{\pi}}
\Big(\frac{\delta^2 \Sigma(\bar{I}_o)}{L}\Big)^{I/2} I^{-\frac{I+1}{2}} (2\pi e)^{I/2}  \notag\\&
\le
\label{th5_rel1}
\frac{e I^{-(p+1)/2}}{(2\pi)^{p/2} \sqrt{\pi}}
\Big[\Big(\frac{2e \bar{I}_o}{I}\Big)^{p+1/2}
\Big(\frac{\delta}{\sqrt{L(2p+1)}}\Big)\Big]^I.
\end{align}  

\paragraph{Step 4: applying Lemma \ref{lem_oversmooth}}
Denote for brevity $\varrho=a(K)-\alpha$. By (\ref{rho}), $\varrho>0$.
Applying Lemma \ref{lem_oversmooth} with  
$\varkappa=\frac{\varkappa_0}{2}=\frac{a(K)-\alpha}{2a(K)}$  
(so that $a(K)(1-\varkappa)-\alpha = \frac{a(K)-\alpha}{2}=\frac{\varrho}{2}$), we obtain  
\begin{equation}
\label{th5_rel2}
\mathrm{E}_{\theta_0} \mathrm{P}(\mathcal{I}< 
 \varkappa  \bar{I}_o|X)
 \le C_\alpha^{-1} e^{-\varrho \bar{I}_o/2} 
\end{equation}
for every $\theta_0\in \ell_2$.
Consider the two cases: $e^{-\varrho \bar{I}_o/2} \le \delta$ and 
$e^{-\varrho \bar{I}_o/2}>\delta$.

%Consider two cases $e^{-\rho I_o/2} \le \delta$ and 
%$e^{-\rho I_o/2}>\delta$ (which are not void as $0<\delta\le 1$),  
%where the oracle $I_o=I_o(\theta_0)$ is defined by (\ref{oracle_risk}),
%$\rho=b(K)-\alpha$, and $b(u)$ is defined by (\ref{rho}). 

 \paragraph{Step 5: the case $e^{-\varrho \bar{I}_o/2}>\delta$}
If $e^{-\varrho \bar{I}_o/2}>\delta$, then $\bar{I}_o< 2 \varrho^{-1} \log (\delta^{-1})$.
By using this, (\ref{ddm1}), (\ref{th5_rel1}) and the notation $p_I= \mathrm{P}(\mathcal{I}=I | X)$, 
we derive that, for $e^{-\varrho\bar{I}_o/2}>\delta$,
\begin{align}
\mathrm{E}_{\theta_0}\mathrm{P} &
\big(\|\theta -\hat{\theta}\|\le \delta\Sigma^{1/2}(\bar{I}_o)\big| X\big)=
\mathrm{E}_{\theta_0}\sum_I \mathrm{P}_I \big(\|\theta-\hat{\theta}\|^2
\le \delta^2 \Sigma(\bar{I}_o) | X\big) p_I \notag\\
&\le 
\sum_I \frac{e I^{-(p+1)/2}}{(2\pi)^{p/2} \sqrt{\pi}}
\Big[\Big(\frac{2e \bar{I}_o}{I}\Big)^{p+1/2}
\Big(\frac{\delta}{\sqrt{L(2p+1)}}\Big)\Big]^I \mathrm{E}_{\theta_0} p_I \notag\\
&\le C_2\delta\big[\log (\delta^{-1})\big]^{p+1/2} \sum_I 
\frac{\big(C_1\delta \big[\log(\delta^{-1})\big]^{p+1/2}\big)^{I-1}}{I^{I(p+1/2) +(p+1)/2}}
\mathrm{E}_{\theta_0} p_I \notag\\
\label{th5_rel3}
&\le C_3\delta \big[\log(\delta^{-1})\big]^{p+1/2},
\end{align}
with $C_1=\frac{(4e/\varrho)^{p+1/2}}{(L(2p+1))^{1/2}}$, 
$C_2=\frac{C_1e}{\pi^{1/2} (2\pi)^{p/2}}$, $\varrho=a(K)-\alpha$, 
$L=\frac{K}{K+1}$,  and some $C_3=C_3(\varrho,L,p)$.
%Let us evaluate the constant $C_3$.
%For $a>0$,
%\begin{equation}
%\label{th5_rel4}
%h(a)=\max_{k\in\mathbb{N}} \frac{a^{(k-1)/2}}{k^{(k+1)/2}} \le
%\mathbb{I}\{ a\le 8 \} + 
%\frac{e^{a/(2e)}}{(2a)^{1/2}}\mathbb{I}\{ a>8\},
%\end{equation}
%because $h(a)$ is increasing, $h(8)=1$, and 
%$\max_{u>0} (a/u)^u = e^{a/e}$.
%Since $\max_{u\ge 0} u(\log(u^{-1}))^d = e^{-d} d^d$, 
%by putting $a=C_1e^{-1}$ in (\ref{th5_rel4}), we can take 
%\[
%C_3 =C_2\Big(\mathbb{I}\{ C_1 \le 8e\} + 
%\frac{e^{a/(2e)}}{(2a)^{1/2}}\mathbb{I}\{ C_1> 8e\}\Big) 
%\le C_2\max\Big\{1, \frac{e^{C_1/(2e^2)+1/2}} 
%{(2C_1)^{1/2}} \Big\}.
%\]

\paragraph{Step 6: the case $e^{-\varrho \bar{I}_o/2} \le \delta$}
Now consider the case $e^{-\varrho\bar{I}_o/2} \le \delta$.
%Take $\varkappa=\frac{\varrho}{2(\alpha+\varrho)}$ (which is actually $\varkappa_{1/2}$
%from Lemma \ref{lem_oversmooth}).
%Applying Lemma \ref{lem_oversmooth} with $h=1/2$,  we obtain  
%\begin{equation}
%\label{th5_rel2}
%\mathrm{E}_{\theta_0} \mathrm{P}(\mathcal{I}< 
% \varkappa  I_o|X)
% \le C_\alpha^{-1} e^{-\rho I_o/2} 
%\end{equation}
%for every $\theta_0\in \ell_2$.
Clearly, $\sum_{I<\varkappa \bar{I}_o} p_I = 
\mathrm{P}\big(\mathcal{I}<\varkappa \bar{I}_o\big|X\big)$.
In view of this, (\ref{th5_rel1}) and (\ref{th5_rel2}), 
\begin{align*}
&\mathrm{E}_{\theta_0}\mathrm{P}
\big(\|\theta-\hat{\theta}\| \le \delta   \Sigma^{1/2} (\bar{I}_o) \big| X\big)=
\mathrm{E}_{\theta_0}\sum_I \mathrm{P}_I \big(\|\theta-\hat{\theta}\|^2
\le \delta^2  \Sigma(\bar{I}_o) | X\big) p_I %\mathrm{P}(\mathcal{I}=I | X) 
\notag\\
&\le 
\sum_{ I\ge \varkappa \bar{I}_o} 
\frac{e I^{-\frac{p+1}{2}}}{(2\pi)^{\frac{p}{2}} \sqrt{\pi}}
\Big[\Big(\frac{2e\bar{I}_o}{I}\Big)^{p+\frac{1}{2}}
\Big(\frac{\delta}{\sqrt{L(2p+1)}}\Big)\Big]^I
\mathrm{E}_{\theta_0} p_I %\mathrm{P}(\mathcal{I}=I | X) 
+\mathrm{E}_{\theta_0} \mathrm{P}\big(\mathcal{I}< 
\varkappa \bar{I}_o\big|X\big)\notag\\
%\label{th3_step5} 
& \le 
C_4\delta \sum_I 
\frac{\Big[\big(\frac{2e}{\varkappa}\big)^{p+\frac{1}{2}} 
\frac{\delta}{\sqrt{L(2p+1)}}\Big]^{I-1}}{I^{(p+1)/2}}
\mathrm{E}_{\theta_0} p_I %\mathrm{P}(\mathcal{I}=I | X) 
+\frac{e^{-\varrho \bar{I}_o/2}}{C_\alpha}\le 
(C_4+C_\alpha^{-1})\delta
\end{align*}
if $e^{-\varrho \bar{I}_o/2} \le \delta$ and 
$\big(\frac{2e}{\varkappa}\big)^{p+\frac{1}{2}} \frac{\delta}{\sqrt{L(2p+1)}}\le 1$. 
Here $C_4 = \frac{e}{(2\pi)^{\frac{p}{2}}\sqrt{\pi L(2p+1)}} 
\big(\frac{2e}{\varkappa}\big)^{p+\frac{1}{2}}$.

\paragraph{Step 7: finalizing the proof of Theorem \ref{th2a_posterior}}
The last relation %(\ref{th3_step5}) 
holds if $e^{-\varrho\bar{I}_o/2} \le \delta 
\le \sqrt{L(2p+1)}\big(\frac{\varkappa}{2e}\big)^{p+1/2}= 
\sqrt{\frac{K(2p+1)}{K+1}}\big(\frac{\varkappa}{2e}\big)^{p+1/2}
=\bar{\delta}_{sb}$ and 
the relation (\ref{th5_rel3}) holds 
if $e^{-\varrho\bar{I}_o/2} > \delta$.
Combining these two relations concludes the proof of the theorem:
for  $0<\delta\le (1 \wedge \bar{\delta}_{sb})=\delta_{sb}$, we have that
\[
\mathrm{E}_{\theta_0}\mathrm{P}
\big(\|\theta-\hat{\theta}\|\le \delta \Sigma^{1/2}(\bar{I}_o)\big| X\big)
\le \max\{C_3,C_4+C_\alpha^{-1}\} \delta
\big[\log(\delta^{-1})\big]^{p+\frac{1}{2}}.
\]

%\begin{supplement}
%\sname{Supplement to}
%\label{suppA}
%\stitle{``On coverage and local radial rates of DDM-credible sets''}
%\slink[doi]{???}
%\sdatatype{.pdf}
%\sdescription{The elaboration on some points and some background 
%information related to the paper are provided in the supplement \cite{Belitser:2015}.}
%\end{supplement}

%\end{document}

\newpage
\appendix
\setcounter{equation}{0}
\def\theequation{S\arabic{equation}}

\section{Supplement to ``On coverage and local radial rates of DDM-credible sets''}

In this supplement, we provide the elaboration on some points  
and some background  information related to the paper 
``On coverage and local radial rates of DDM-credible sets''.

In what follows we use the notations and cross-references to numbered elements 
(like equations, sections) from the paper. We again often drop the dependence 
on $\varepsilon$ to avoid overloaded notations.
For two sequences $\alpha_\varepsilon, \beta_\varepsilon>0$, 
$\alpha_\varepsilon \asymp\beta_\varepsilon$ means that 
$\alpha_\varepsilon/\beta_\varepsilon$ is bounded away from zero 
and infinity as $\varepsilon\to 0$.

\subsection{Minimax confidence ball: degenerate optimal solution}

Optimality is a well developed notion in the framework
of minimax estimation theory and therefore the first approach 
to optimality of confidence sets would be based on the minimax 
convergence rates. 
Suppose our prior knowledge about the model 
$X \sim \mathrm{P}_\theta =\mathrm{P}^{(\varepsilon)}_\theta$
is formalized as follows: %the unknown parameter %(sometimes we call it signal or curve) 
$\theta \in \Theta_\beta \subseteq \Theta$.
Here we consider \emph{non-adaptive} situation, that is, the parameter $\beta \in \mathcal{B}$ is known
and we can use this knowledge in the construction of the confidence ball.
Parameter $\beta$ typically has a meaning of smoothness of $\theta$.
By using lower bounds from the minimax estimation theory, we show below  
that the minimax rate $R_\varepsilon(\Theta_\beta)$ 
is in some sense the \emph{best global radial rate}, i.e., the smallest possible among 
all radial rates that are constant on $\Theta_\beta$.

Let $w:\, \mathbb{R}_+ \to \mathbb{R}_+$, be a loss function, i.e., nonnegative and 
nondecreasing on $\mathbb{R}_+$, $w(0)=0$ and $w\not\equiv 0$.
The \emph{maximal risk} of an estimator $\hat{\theta}$ over $\Theta_\beta$ is 
$r_\varepsilon(\Theta_\beta, \hat{\theta}) 
=r_\varepsilon(\Theta_\beta, \hat{\theta},R_\varepsilon)=\sup_{\theta\in\Theta_\beta} 
\mathrm{E}_\theta [w(R^{-1}_\varepsilon d(\hat{\theta},\theta))]$ 
(calibrated by a sequence $R_\varepsilon>0$), and
the  \emph{minimax risk} over $\Theta_\beta$ is 
$r_\varepsilon(\Theta_\beta) = r_\varepsilon(\Theta_\beta,R_\varepsilon) = 
\inf_{\hat{\theta}} r_\varepsilon(\Theta_\beta, \hat{\theta},R_\varepsilon)$,
where the infimum is taken over all possible estimators 
$\hat{\theta}=\hat{\theta}(X) \in \mathcal{L}$, 
measurable functions of the data $X$.
We consider here the asymptotic regime $\varepsilon \to 0$ as 
in the most literature on minimax estimation theory.
A positive sequence $R_\varepsilon=R_\varepsilon(\Theta_\beta)$ 
and an estimator $\hat{\theta}$ are called \emph{minimax rate} and
\emph{minimax estimator} respectively if, for $0<b\le B<\infty$,
\begin{equation}
\label{minimax}
 b\le \liminf_{\varepsilon \to 0} r_\varepsilon(\Theta_\beta,R_\varepsilon) \le 
\limsup_{\varepsilon \to 0} r_\varepsilon(\Theta_\beta, \hat{\theta},R_\varepsilon)\le B.
\end{equation}
The first inequality is called lower bound and the last one upper bound.
Note that the minimax rate is not unique. If $w(u) = u^p$, 
$p>0$ (the most popular choice: quadratic loss function $p=2$),
then often the quantity $\mathnormal{r}_\varepsilon(\Theta_\beta) = 
\inf_{\hat{\theta}} \sup_{\theta\in\Theta_\beta} 
\big(\mathrm{E}_\theta \big[d(\hat{\theta},\theta)\big]^p\big)^{1/p}$, is called the minimax risk. 
In this case, the minimax risk is itself the minimax rate, but so is any sequence
$R_\varepsilon(\Theta_\beta) \asymp \mathnormal{r}_\varepsilon(\Theta_\beta)$. 
If the set $\Theta_\beta$ is known, $\mathnormal{r}_\varepsilon(\Theta_\beta)$ is in 
principle known as well, one would like to derive an explicit expression 
$R_\varepsilon(\Theta_\beta)$ for the minimax rate. There is vast literature on 
this topic, minimax rates and estimators are obtained in a variety of models, 
settings and smoothness classes $\Theta_\beta$.
For example, in classical nonparametric regression model
and density estimation problem with  Sobolev, H\"older or Besov  
classes $\Theta_\beta$ of $d$-variate functions of smoothness 
$\beta$  and the sample size $n$, the minimax rate is
$R_\varepsilon(\Theta_\beta) =(\varepsilon^2)^{\frac{\beta}{2\beta+d}}$ with 
$\varepsilon= n^{-1/2}$.

Suppose that a lower bound in (\ref{minimax}) is established for
zero-one loss $w(u) = \mathrm{1}\{u \ge c\}$ %(this implies the lower bound for other loss functions)
and a (minimax) rate $R_\varepsilon(\Theta_\beta)$: for any $\hat{\theta}$ and some $b>0$,
\begin{equation}
\label{lower_bound1}
\liminf_{\varepsilon \to 0} 
\sup_{\theta\in\Theta_\beta} 
\mathrm{P}_\theta \big(\theta\not\in B(\hat{\theta}, cR_\varepsilon(\Theta_\beta)) \big)=
\liminf_{\varepsilon \to 0} r_\varepsilon(\Theta_\beta, \hat{\theta}, 
R_\varepsilon(\Theta_\beta)) \ge b.
\end{equation}
We claim that it is impossible for a confidence ball $B(\hat{\theta},\hat{r})$
to have simultaneously a global radial rate of a smaller order than 
$R_\varepsilon(\Theta_\beta)$ and its coverage probability being arbitrarily close to 1
uniformly in $\theta \in \Theta_\beta$. 

There are two ways to establish lower bounds for the optimality of 
confidence sets: either assume the coverage relation in (\ref{conf_ball_problem}) 
and show that the size relation must fail or the other way around. In the literature, 
the former approach is commonly used for global minimax radial rates, 
cf.\  \cite{Robins&vanderVaart:2006sup}. However, when we construct confidence 
sets as credible balls with respect to some DDM $\mathrm{P}(\cdot|X)$, 
it is more natural to use the latter approach since the DD-radius gets determined by the DDM 
and typically the size requirement in (\ref{conf_ball_problem}) holds true for the whole set $\Theta$,
whereas the coverage requirement fails to hold for  some ``deceptive'' $\theta\in\Theta$.
 
More precisely, if we assume 
\begin{equation}
\label{lower_bound2}
\liminf_{\varepsilon\to 0} 
\inf_{\theta\in\Theta} \mathrm{P}_\theta \big( \hat{r} 
\le cR_\varepsilon(\Theta_\beta) \big) \ge 1-b/2,
\end{equation}
then 
\begin{align*}
\mathrm{P}_\theta \big(\theta\not\in B(\hat{\theta},\hat{r})\big)
&=
\mathrm{P}_\theta \big(\theta\not\in B(\hat{\theta},\hat{r}), 
\hat{r} \le R_\varepsilon(\Theta_\beta)\big)  +
\mathrm{P}_\theta \big(\theta\not\in B(\hat{\theta},\hat{r}), 
\hat{r} > R_\varepsilon(\Theta_\beta)\big)\\
&\ge 
\mathrm{P}_\theta \big(\theta\not\in B(\hat{\theta},R_\varepsilon(\Theta_\beta)), 
\hat{r} \le R_\varepsilon(\Theta_\beta)\big)  \\
&\ge 
\mathrm{P}_\theta \big(\theta\not\in B(\hat{\theta},R_\varepsilon(\Theta_\beta))\big)
+\mathrm{P}_\theta \big(\hat{r} \le R_\varepsilon(\Theta_\beta)\big) -1.
\end{align*}
Combining this with (\ref{lower_bound1}) and (\ref{lower_bound2}),  we obtain
\[
\liminf_{\varepsilon \to 0} \sup_{\theta\in\Theta_\beta} 
\mathrm{P}_\theta \big(\theta\not\in B(\hat{\theta},\hat{r})\big)
\ge b+1-b/2-1 \ge b/2,
\]
which gives a bound on the coverage probability of 
$B(\hat{\theta},\hat{r})$, at least for some (worst representatives) 
$\theta\in\Theta_\beta$.
We thus established that it is impossible for a confidence ball $B(\hat{\theta},\hat{r})$
to have simultaneously a global radial rate of a smaller order than 
$R_\varepsilon(\Theta_\beta)$ and its coverage probability being arbitrarily close to 1
uniformly in $\theta \in \Theta_\beta$.

On the other hand, suppose now that there is a minimax estimator $\hat{\theta}$ satisfying 
(\ref{minimax}), with, say, $w(u)=u$, and the corresponding  minimax rate 
$R_\varepsilon(\Theta_\beta)$.
If we use the minimax risk $R_\varepsilon(\Theta_\beta)$ as the benchmark 
for the effective radius of confidence balls, then the problem of constructing 
an optimal confidence ball satisfying (\ref{conf_ball_problem}) with the radial rate
$\mathnormal{r}_\varepsilon(\theta) =R_\varepsilon(\Theta_\beta)$ is readily solved.
Indeed, since in this non-adaptive setting the quantity $R_\varepsilon(\Theta_\beta)$ 
is in principle known (could be difficult to evaluate in models), we can simply
take the following confidence ball $B(\hat{\theta},CR_\varepsilon(\Theta_\beta))$, 
i.e., $\hat{r} =R_\varepsilon(\Theta_\beta)$. Then, by (\ref{minimax}),
\[
\limsup_{\varepsilon \to 0}
\sup_{\theta\in\Theta_\beta} \mathrm{P}_\theta \big(\theta \not\in 
B(\hat{\theta},CR_\varepsilon(\Theta_\beta)) \big) 
\le \limsup_{\varepsilon \to 0}\frac{\sup_{\theta\in\Theta_\beta} 
\mathrm{E}_\theta d(\theta,\hat{\theta})}{CR_\varepsilon(\Theta_\beta)} \le 
\frac{B}{C},
\]
so that the coverage relation in (\ref{conf_ball_problem}) will hold for 
sufficiently large $C$. The size relation in (\ref{conf_ball_problem}) 
is trivially satisfied for any $c>1$ any $\alpha_2\in[0,1]$
since $\hat{r} =\mathnormal{r}_\varepsilon(\theta)=R_\varepsilon(\Theta_\beta)$.
This means that the ball $B(\hat{\theta},CR_\varepsilon(\Theta_\beta))$ 
satisfies (\ref{conf_ball_problem}) with 
$\mathnormal{r}_\varepsilon(\theta)=R_\varepsilon(\Theta_\beta)$ 
and $\Theta_{cov}=\Theta_{size}=\Theta_\beta$, for appropriate choices 
of involved constants. Thus the ball 
$B(\hat{\theta},CR_\varepsilon(\Theta_\beta))$, with the deterministic radius 
$R_\varepsilon(\Theta_\beta)$, is optimal in the minimax sense (in the non-adaptive formulation). 
Knowledge $\theta\in\Theta_\beta$  and the fact that radial rates 
are restricted to be global lead to such a simplistic optimal solution. 
But this solution is of course not satisfactory, because even if we know a priori that $\theta\in\Theta_\beta$, 
it is possible that $\theta\in \Theta_{\beta_1} \subset \Theta_\beta$, 
with $\beta_1\not =\beta$. Then the obtained radial rate 
$R_\varepsilon(\Theta_\beta) > R_\varepsilon(\Theta_{\beta_1})$ is bigger than it could have been 
if one had used a ball with a DD-radius that can adapt to the rate $R_\varepsilon(\Theta_{\beta_1})$.

This consideration illustrates that minimax non-adaptive framework 
for the confidence inference lead to degenerate and uninteresting ``optimal'' solution.

\subsection{Bayes approach yields DDMs} %Section \ref{sec_general_appr}}

Suppose we are given a general statistical model $X\sim \mathrm{P}_\theta$, 
$\theta \in \Theta$, and we want to construct a DDM on %(in general, infinite-dimensional) 
parameter $\theta$. Typically, one obtains a DDM on $\theta$ by  applying a
Bayesian approach: put a prior $\pi$ on $\theta$ 
and regard $\mathrm{P}_\theta$ as conditional distribution of 
$X$ given $\theta$, i.e., $X | \theta \sim \mathrm{P}_{\theta}$,  $\theta \sim
\pi$. This leads to the posterior distribution $\Pi(\theta|X)$ which is a DDM 
on $\theta$. A DD-center $\hat{\theta}= \hat{\theta}(X)$  can in turn be 
constructed by using $\Pi(\theta|X)$, e.g., as the mean with respect to 
$\Pi(\theta|X)$ or the MAP-estimator.
Other examples of DDMs include empirical Bayes, (generalized) fiducial 
distributions and bootstrap. In fact, any combination of these can  be used as DDM.
 
In an adaptive inference context, one typically has a family of
priors $\{\pi_\beta, \, \beta \in\mathcal{B}\}$, where parameter $\beta$
models some additional structure  on $\theta$; 
sometimes $\beta$ has a meaning of ``smoothness''. 
There are two basic approaches to 
derive a resulting adaptive posterior $\Pi(\theta|X)$: pure Bayes or empirical Bayes. 
In the first case, we construct a hierarchical prior on  $(\theta,\beta)$:
regard $\pi_\beta$ as a conditional prior
on $\theta$ given $\beta$, and next we put a prior, say $\lambda$, on 
$\beta \in \mathcal{B}$.  This leads to the posteriors $\Pi(\theta|X)$ 
and $\lambda(\beta|X)$ (that may be also useful in the inference).
%thus obtaining the DDMs on $\theta$ and $\beta$. 
In the empirical Bayes approach, each prior $\pi_\beta$ leads to the
posterior $\Pi_\beta(\theta|X)$. We then compute the marginal distribution
$\Pi_\beta$ of $X$ and construct an estimator $\hat{\beta}$ by using 
this marginal distribution (for example, marginal maximum likelihood). 
Next we plug in the obtained $\hat{\beta}$ in the posterior $\Pi_\beta(\theta|X)$, 
so that we get the so called empirical Bayes posterior
$\hat{\Pi}(\theta|X) = \Pi_{\hat{\beta}}(\theta|X)$.
Both resulting DDMs $\Pi(\theta|X)$ and $\hat{\Pi}(\theta|X)$  can be used 
in the construction of confidence sets as DDM-credible sets. 
Also any combination of full Bayes and empirical Bayes approaches
(with respect to different parameters) that leads to some resulting
DDM $\mathrm{P}(\theta|X)$ can in principle be used. 

To some extent, we can manipulate with DDMs as with usual conditional measures. For example,  
if we have a family of DDMs on $\Theta$, say, $\{\mathrm{P}_I(\cdot|X), \, I\in \mathbb{N}\}$ 
and a DDM $\mathrm{P}(\mathcal{I}=I|X)$ on $\mathbb{N}$, we can construct a mixture DDM
$\mathrm{P}(\cdot|X)=\sum_I\mathrm{P}_I(\cdot|X)\mathrm{P}(\mathcal{I}=I|X)$.

\subsection{Remarks about Conditions (A1)--A(3), (\~A1)--(\~A2)} 
\label{subsec_remarks}
Here we collect some remarks about Conditions (A1)--A(3), (\~A1)--(\~A2).

\paragraph{Asymptotic versions of conditions (A1)--(A3) and (\~A1)--(\~A2)}
%Introduce asymptotic versions of conditions (A1)--(A3), (\~A1)--(\~A2), as 
%$\varepsilon \to 0$:  (AA1)--(AA3) and (A\~A1)--(A\~A2) respectively.
Suppose a point $\theta_0\in\Theta$, some radial rate 
$\mathnormal{r}(\theta)$, a DDM 
$\mathrm{P}(\cdot|X)$ and a DD-center $\hat{\theta}=\hat{\theta}(X)$ are given, 
 $M_\varepsilon, M'_\varepsilon, \delta_\varepsilon>0$ and  $\varepsilon \to 0$.
The asymptotic versions of conditions (A1)--(A3), (\~A1)--(\~A2) are as follows.
\begin{itemize}
\item[(AA1)] For some %positive  
$M_\varepsilon \to \infty $, 
$\mathrm{E}_{\theta_0} \big[\mathrm{P}(d(\theta,\hat{\theta}) \ge M_\varepsilon
\mathnormal{r}(\theta_0)|X) \big] \to 0$.
 
\item[(AA2)] For some  
$\delta_\varepsilon \to 0$,
$\mathrm{E}_{\theta_0} \big[\mathrm{P}(d(\theta,\hat{\theta}) \le \delta_\varepsilon
\mathnormal{r}(\theta_0)|X) \big] \to 0$.

\item[(AA3)]  For  some  
$M'_\varepsilon \to\infty $, 
$\mathrm{P}_{\theta_0} \big(d(\theta_0,\hat{\theta}) \ge M'_\varepsilon
\mathnormal{r}(\theta_0) \big)\to 0$.
\end{itemize}

\begin{itemize}
\item[(A\~A1)]  
For  some  
$M_\varepsilon \to \infty $,
$\mathrm{E}_{\theta_0} \big[\mathrm{P}(d(\theta_0,\theta) \ge M_\varepsilon
\mathnormal{r}(\theta_0)|X) \big] \to 0$.
\item[(A\~A2)]  
For some $\delta_\varepsilon \to 0$
and any measurable $\tilde{\theta}=\tilde{\theta}(X)$,
$\mathrm{E}_{\theta_0} \big[\mathrm{P}(d(\theta,\tilde{\theta}) \le \delta_\varepsilon
\mathnormal{r}(\theta_0)|X) \big] \to 0$.
\end{itemize}

\paragraph{Connection to Bayesian nonparametrics}
In the Bayesian framework, when the DDM
$\mathrm{P}(\cdot|X)$ is the posterior (or empirical Bayes posterior) distribution 
on $\theta$ with respect to some prior, condition (\~A1) (and its asymptotic 
version (A\~A1) below) describes the so called posterior contraction rate
$\mathnormal{r}(\theta_0)$. 
To establish such assertions is an interesting and challenging problem nowadays,
especially in nonparametric models when one wants to characterize 
the (frequentist) quality of Bayesian procedures. Much recent research has been devoted to this topic.
We just mention that predominantly global posterior convergence rates
are studied, i.e., $\mathnormal{r}(\theta_0) =R(\Theta)$ 
for all $\theta_0\in \Theta$.
To the best of our knowledge a local posterior convergence rate is considered 
only in \cite{Babenko&Belitser:2010sup}.

%\paragraph{(\~A1)--(\~A2) imply (A1)--(A3) for the default DD-center}
%Notice that a DD-center $\hat{\theta}$
%satisfying conditions  (A1)--(A3)
%is needed, while there is no $\hat{\theta}$ involved in conditions (\~A1)
%and (\~A2). Actually, under condition (\~A1) there is a generic  
%choice $\hat{\theta}$, the default DD-center, 
%that would automatically satisfy conditions (A1) and (A3) by Proposition \ref{prop1}. 
%Besides,  (\~A2) certainly implies (A2) 
%for any DD-center $\hat{\theta}$. This means that (\~A1)--(\~A2) 
%imply (\~A1) and (A2) which in  turn 
%imply (A1)--(A3) for that generic DD-center $\hat{\theta}$. 

\paragraph{Pushing the conditions to the utmost}
The smaller the radial rate $\mathnormal{r}(\theta_0)$, 
the easier (A2) to satisfy, but the harder (A1), (A3) and (\~A1).  
We are interested in the smallest possible radial rate 
since this quantity will govern the size of the resulting confidence ball. 
Thus, the right strategy would be first to determine the smallest radial rate 
$\mathnormal{r}(\theta_0)$ for which (\~A1) (or (A1) and (A3)) holds,
preferably uniformly over $\theta_0\in\Theta$. This would be the so called upper
bound  for the contraction rate of the DDM $\mathrm{P}(\cdot|X)$ 
around $\theta_0 \in \Theta$. Next, one needs to study whether (A2) holds as well
with $\mathnormal{r}(\theta_0)$ for $\theta_0\in\Theta$; 
if not possible for all $\theta_0\in\Theta$, then for $\theta_0\in\Theta_0$ 
with the ``largest'' $\Theta_0 \subset \Theta$. This is so called lower bound
for the contraction rate of the DDM $\mathrm{P}(\cdot|X)$ 
around $\hat{\theta}$. 

Typically, the upper bound (\~A1) for the DDM-contraction rate 
holds for all $\theta\in\Theta$ with a ``good'' local radial rate,
whereas the lower bound (A2) only for $\theta \in \Theta_0$, 
with some set of  ``non-deceptive'' parameters $\Theta_0\subset\Theta$.

%Looking ahead, we will see in the normal means model that it is possible to determine the 
%smallest radial rate for which (\~A1) holds  uniformly over $\theta_0\in\Theta$,
%whereas condition (A2) holds only for $\theta \in \Theta_0$, 
%where $\Theta_0\subset\Theta$.
 
%\begin{remark}
%Looking ahead, we will see in the normal means model that 
%the upper bound (\~A1) for the DDM-contraction rate 
%holds for all $\theta\in\Theta$ with a ``good'' local radial rate,
%whereas the lower bound (A2) only for $\theta \in \Theta_0\subset\Theta$, 
%where $\Theta\setminus \Theta_0$ is the set of 
%``deceptive'' parameters that should be removed from $\Theta$ for (A2) to hold 
%with the same radial rate. 
%\end{remark}

\subsection{Examples of applying Propositions \ref{prop2a} and \ref{prop2b}}

 \paragraph{Normal case}
Suppose we observe a sample $X=X^{(\varepsilon)}=(X_1,\ldots, X_n)$ 
from $N(\theta_0, \sigma^2)$, $\theta_0\in\mathbb{R}$, 
where $\varepsilon = \sigma n^{-1/2}$. 
%Then $\bar{X}_n  \sim N(\theta_0, \sigma^2_n)$ 
%with $\sigma^2_n=\sigma^2/n$, which is, by sufficiency, equivalent to the original model.
Take the estimator $\hat{\theta}= \bar{X}=\frac{1}{n}\sum_{i=1}^n X_i\sim N(\theta_0, \varepsilon^2)$
and the radial rate $\mathnormal{r}(\theta_0) =
\mathnormal{r}_\varepsilon(\theta_0) =\varepsilon$.
The normal prior $\pi = N(\mu,\tau^2)$ on $\theta$, 
leads to the normal posterior $\pi(\theta|X)=N\big(\frac{\varepsilon^2 \mu+\tau^2 \bar{X}}
{\varepsilon^2+\tau^2}, \frac{\varepsilon^2\tau^2}{\varepsilon^2+\tau^2}\big)$.
%then the posterior is
%\[
%\pi(\theta|X) = N\Big(\frac{\sigma^2_n \mu+\tau^2\bar{X}_n}
%{\sigma^2_n+\tau^2}, \frac{\sigma^2_n\tau^2}{\sigma^2_n+\tau^2}\Big).
%\] 
Then, as DDM on $\theta$ we take 
\[
\mathrm{P}(\theta|X)=\pi(\theta|X)\big|_{\mu=\hat{\mu}}=
N \Big(\bar{X}, \frac{\varepsilon^2\tau^2}{\varepsilon^2+\tau^2}\Big),
\]
the empirical Bayes posterior with $\hat{\mu}=\bar{X}$, 
and construct the  DDM-credible ball (in this case: interval) $B(\hat{\theta},M\hat{r}_\kappa)$ 
for $\theta_0$, according to the general procedure from the paper. %(\ref{conf_ball}).
%It is readily seen that (A1) and (A3) are satisfied 
%with $\phi_1(M) = \phi_2(M) =(2/\sqrt{2 \pi}) e^{-M^2/2}/M$ and
%condition (A2) is satisfied (if $\varepsilon \le \tau$) with 
%$\psi(\delta)=C_3\delta$, where $C_1= 2/\sqrt{2 \pi}$, $C_2 =1/2$
%and $C_3 =1/\sqrt{\pi}$.
Then (A1) and (A3) are satisfied with $\phi_1(M) = \phi_2(M) = 
C e^{-cM^2}/M$. Indeed, for a $\xi\sim N(0,1)$, 
\[
\mathrm{P}\big(|\hat{\theta}-\theta| \ge M\mathnormal{r}(\theta_0)|X\big)=
\mathrm{P}\Big(\frac{\varepsilon\tau|\xi | }{\sqrt{\varepsilon^2+\tau^2}}
\ge M\varepsilon\Big) \le
\mathbb{}P(|\xi| \ge M) \le\frac{2e^{-M^2/2}}{\sqrt{2\pi}M},
\]
\[
\mathrm{P}\big(|\hat{\theta}-\theta_0| \ge M\mathnormal{r}(\theta_0)\big)=
\mathrm{P}\big(|\varepsilon \xi | \ge M \varepsilon\big) = 
\mathbb{}P(|\xi| \ge M) \le \frac{2e^{-M^2/2}}{\sqrt{2\pi}M}.
\]
Assume $\varepsilon \le \tau$,  then condition (A2) is also satisfied 
with $\psi(\delta)=\delta/\sqrt{\pi}$:   
\[
\mathrm{P} \big(|\hat{\theta}-\theta| \le \delta\mathnormal{r}(\theta_0)|X\big)=
\mathrm{P}\big(|\xi| \le \delta \sqrt{1+\varepsilon^2/\tau^2}\big) 
\le \mathrm{P}(|\xi| \le \delta\sqrt{2})\le \delta/\sqrt{\pi}.
\]
%Since conditions (A1)--(A3) are fulfilled, Propositions \ref{prop2a} and  \ref{prop2b} hold true.
One can think of the above two properties of the normal distribution as  ``ring tightness''.
The functions $\phi_1, \phi_2$ and $\psi$ do not depend on $\varepsilon$ and $\theta_0$,
so that, by using Propositions \ref{prop2a} and \ref{prop2b} as described above, we can derive non-asymptotic 
coverage and size relations in (\ref{conf_ball_problem}) for the DDM-credible interval 
$B(\hat{\theta},M\hat{r}_\kappa)$.
%In this case we obtained non-asymptotic conditions, which  immediately imply 
%asymptotic versions for the classical asymptotic regime $\varepsilon \to 0$, 
%or equivalently  $n \to \infty$. 

Of course, %in this case there is 
the classical  confidence interval $\bar{X}_n\pm z_{1-\alpha/2}\sigma/\sqrt{n}$
has the same radial rate whose coverage may even be (non-asymptotically) 
better. In that respect, the above example is somewhat uninteresting and is 
provided only for  the illustrative  purposes.
%No wonder this resulting credible interval has a good coverage probability as the posterior, being 
%normal, concentrates in a ``ring-manner'' around the truth.  
%One can further consider multivariate normal case and construct confidence set as credible set 
%on basis of an empirical Bayes posterior. Non-asymptotic improvement of the coverage probability
%is possible in this case (see Tseng and Brown (1997) and further references therein). 
%We will not  however pursue this problem here.

\paragraph{Bernstein-von Mises case}
%\begin{remark}[\bf Bernstein-von Mises case]
For the finite dimensional parameter, consider a general situation 
when some mild regularity conditions on the model and the prior 
lead to the resulting asymptotically normal posterior. This is the so called 
\emph{Bernstein-von Mises} property as often termed in the literature. 
Suppose $X=X^{(n)}\sim \mathrm{P}_{\theta_0}$, $\theta\in\Theta$,
information parameter $\varepsilon = n^{-1/2}$,
with a prior $\theta\sim \pi$ on some $\sigma$-algebra 
$\mathcal{B}_{\Theta}$ on $\Theta$ and a $\sqrt{n}$-consistent 
estimator $\hat{\theta}$ such that  the asymptotic version of 
(A3), namely (AA3)  (given in Subsection \ref{subsec_remarks}), is satisfied with 
the radial rate $\mathcal{R}_n(\theta_0)=n^{-1/2}$ and 
in $\mathrm{P}_{\theta_0}$-probability
\[
\sup_{B\in\mathcal{B}_{\Theta}} \big| \pi(B|X) 
-N(\hat{\theta}, I(\theta_0))(B)\big| \to 0, 
\quad \text{as }\;n \to \infty.
\]
where $N(\mu, \Sigma)(B)=\mathrm{P}(Y \in B)$ with 
$B\sim N(\mu, \Sigma)$ for some multivariate normal distribution with 
mean $\mu$ and covariance matrix $\Sigma$.
Besides, (A1) and (A2) hold for the DDM $N(\hat{\theta}, I(\theta_0))$. All these facts imply that
the asymptotic versions (AA1)--(AA3) introduced in Supplement are satisfied
with $\varepsilon=n^{-1/2}$. Asymptotic versions of Propositions \ref{prop2a} and \ref{prop2b} 
follow immediately, which yields (asymptotically) a full  coverage probability
and the optimal global radial rate $n^{-1/2}$, which is of course well known. 

Interestingly, there is nothing special about normal distribution in the above 
arguments, any resulting limiting distribution with a ``ring structure'' will do the same job.
Ring structure means negligible probability mass outside a ring, whose inner radius is a sufficiently small 
multiples of the radial rate and the outer radius is a sufficiently big multiples of the radial rate.
In fact, the existence of an exact limiting distribution 
is also not decisive, ``ring tightness'' (which is nothing else but (AA1)--(AA2)) 
would be enough. For example, the Bernstein-von Mises
property is more than needed if we only want to make sure that 
a credible set serves as a proper confidence set.

\subsection{Corollary from Propositions \ref{prop1}--\ref{prop2b}} 

Propositions \ref{prop1}--\ref{prop2b} and Remark \ref{rem_def_ball} 
entail the following corollary for the default confidence ball 
$\tilde{B}_{M,\kappa}$ defined by (\ref{def_ball}). 
\begin{corollary} 
\label{cor_cor1}
Let a DDM  $\mathrm{P}(\cdot|X)$ satisfy conditions (\~A1) and (\~A2) 
with some radial rate $\mathnormal{r}(\theta_0)$, $\theta_0\in\Theta$, 
and some functions $\varphi$ and $\psi$, respectively.
Let $\kappa\in (0,1)$, the default ball $\tilde{B}_{M,\kappa}$ be
defined by (\ref{def_ball}) and 
$\hat{r}_\kappa$  be its DD-radius defined by (\ref{radius}).
Then  for  any $M, \delta>0$,
\[
\mathrm{P}_{\theta_0} \big(\theta_0 \not \in \tilde{B}_{M,\kappa}\big)
\le \frac{3\varphi(\frac{2M \delta}{5})}{2}+ \frac{\psi(\delta)}{1-\kappa}, 
\;\; \mathrm{P}_{\theta_0}\big(\hat{r}_\kappa \ge M
\mathnormal{r}(\theta_0)\big)
\le \frac{3\varphi(\frac{M}{5})}{2\kappa} +  \frac{\varphi(\frac{M}{2})}{\kappa}.
\]
\end{corollary}
This corollary can be used for establishing the optimality 
framework (\ref{conf_ball_problem}) in the same way as 
Propositions \ref{prop2a} and \ref{prop2b} as we outlined in 
Subsection \ref{subsec_cov_size}, provided the functions 
$\varphi$ and $\psi$ from conditions (\~A1) and (\~A2) are bounded uniformly 
over appropriate sets $\Theta_{cov}$ and $\Theta_{size}$.

\subsection{Minimality of condition (A2)} 

Let us demonstrate that condition (A2) is in some sense the minimal condition for providing a
sufficient $\mathrm{P}_{\theta_0}$-coverage of the $\mathrm{P}(\cdot|X)$-credible ball
with the sharpest rate.
\begin{proposition} 
\label{prop4}
For a DDM  $\mathrm{P}(\cdot|X)$ on $\Theta$ and a 
DD-center $\hat{\theta}$, let the ball $B(\hat{\theta},M\hat{r}_\kappa)$ 
be constructed according to (\ref{conf_ball}) with any $\kappa\in (0,1)$ and $M>0$.
Further, for a $\theta_0\in\Theta$ and a radial rate 
$\mathnormal{r}(\theta_0)$, denote  
%\begin{align*}
$\psi_2(\delta)=\psi_2(\delta,\varepsilon,\theta_0)=
\mathrm{P}_{\theta_0} \big(d(\theta_0,\hat{\theta}) \le \delta
\mathnormal{r}(\theta_0) \big)$, 
$\alpha(\delta)=\alpha(\delta,\varepsilon,\theta_0)=
\mathrm{E}_{\theta_0} \big[\mathrm{P}(d(\theta,\hat{\theta}) > \delta 
\mathnormal{r}(\theta_0)|X) \big]$.
%\end{align*}
Then
\[
\mathrm{P}_{\theta_0}\big(\theta_0 \in B(\hat{\theta},M\hat{r}_\kappa)\big)
\le \psi_2(\delta M)+\alpha(\delta)\kappa^{-1} \quad \text{for any}\;\; \delta>0.
\]
\end{proposition}

\begin{proof}  In view of  
the definition (\ref{conf_ball}), we derive
\begin{align*}
&\mathrm{P}_{\theta_0}\big(\theta_0 \in B(\hat{\theta},M\hat{r}_\kappa)\big) \\
&=
\mathrm{P}_{\theta_0}\big(\theta_0 \in B(\hat{\theta},M\hat{r}_\kappa), 
\hat{r}_\kappa \le \delta\mathnormal{r}(\theta_0)\big) 
+ \mathrm{P}_{\theta_0}\big(\theta_0 \in B(\hat{\theta},M\hat{r}_\kappa), 
\hat{r}_\kappa > \delta\mathnormal{r}(\theta_0)\big) \\
&\le
\mathrm{P}_{\theta_0}\big(d(\hat{\theta}, \theta_0) 
\le\delta M \mathnormal{r}(\theta_0)\big) 
+\mathrm{P}_{\theta_0}\big(\hat{r}_\kappa > \delta\mathnormal{r}(\theta_0)\big)\\
&\le \mathrm{P}_{\theta_0}\big(d(\hat{\theta}, \theta_0) 
\le\delta M \mathnormal{r}(\theta_0)\big) +
\mathrm{P}_{\theta_0} \big(\mathrm{P}(d(\hat{\theta},\theta)\le  
\delta\mathnormal{r}(\theta_0)|X) \le 1-\kappa\big) \\
&\le 
\mathrm{P}_{\theta_0}\big(d(\hat{\theta}, \theta_0) 
\le\delta M \mathnormal{r}(\theta_0)\big) +
\frac{\mathrm{E}_{\theta_0} \big(\mathrm{P}(d(\hat{\theta},\theta)>  
\delta\mathnormal{r}(\theta_0)|X\big)}{\kappa}\\
&\le \psi_2(\delta M)+\frac{\alpha(\delta)}{\kappa}. \qedhere
\end{align*}
\end{proof}
One should interpret this proposition as follows. First, given a DD-center
$\hat{\theta}$, we determine a local radial rate $\mathnormal{r}(\theta_0)$
such that $\psi_2(\delta)\le \bar{\alpha}(\delta)$ for all $0<\delta\le \delta_0$, for some 
``small''  $\bar{\alpha}(\delta)$. 
This describes the sharpest rate for estimating $\theta_0$ by $\hat{\theta}$.
Next,
$\alpha(\delta)$ being small for small $\delta$  means 
that the DDM $\mathrm{P}(\cdot|X)$ concentrates around $\hat{\theta}$
with  a faster rate than %the radial rate 
$\mathnormal{r}(\theta_0)$,
which  can be regarded as negation of condition (A2).
The above proposition says basically that, under negation of (A2) with the sharpest rate, 
the coverage probability of the credible ball $B(\hat{\theta},M\hat{r}_\kappa)$
is bounded from above. Thus, (A2) is the minimal condition if we want to have 
the sharpest  rate and a good coverage.
This quantifies the following simple intuitive idea:
if  the DDM $\mathrm{P}(\cdot|X)$
contracts  in the DD-center $\hat{\theta}$  faster than 
$\mathnormal{r}(\theta_0)$, then the resulting radius  of the credible 
ball $B(\hat{\theta},M\hat{r}_\kappa)$ 
is going to be of a smaller order than $\mathnormal{r}(\theta_0)$.
But this is going to be (over-optimistically) too small if 
the convergence rate of the center $\hat{\theta}$ to the truth $\theta_0$  is not faster than 
$\mathnormal{r}(\theta_0)$. Then the credible ball $B(\hat{\theta},M\hat{r}_\kappa)$
will clearly miss the truth with some probability bounded away from zero.

\subsection{Inverse and direct Gaussian sequence models}

%Model (\ref{model}) is known to be the \emph{normal sequence model}. 
Model (\ref{model}) is known to be the sequence version of 
the \emph{inverse signal-in-white-noise model}.
This model captures many of the conceptual issues 
associated with nonparametric estimation, with a minimum of technical complication.
Gaussian white noise models are of a canonical type of model 
which serves as a purified approximation to some other statistical models such as 
nonparametric regression model, density estimation, spectral function estimation, 
by virtue of the so called \emph{equivalence principle}. %; see Efromovich (1999).  
The statistical inference results for the generic model (\ref{model}) can %in principle 
be conveyed to other models, according to this equivalence principle.
However, in general the problem of establishing the equivalence 
in a precise sense is a delicate task.  Below we outline the relations with some other models.
 
Let $\mathbb{H}$, $\mathbb{G}$ be two separable Hilbert spaces and $A$ be a continuous 
operator $A: \, \mathbb{H} \to \mathbb{G}$. Suppose we observe 
\[
Y=Af +\varepsilon \xi,
\]
where $\varepsilon>0$ is the noise level, $\xi$ is Gaussian white noise on $\mathbb{G}$,
i.e., $\langle \xi, g \rangle \sim N(0,\|g\|^2)$ and $\text{Cov} (\langle \xi, g \rangle, 
\langle \xi, g'\rangle)= \langle g, g' \rangle$ for any $g,g' \in \mathbb{G}$, $\|\cdot\|$ and
$\langle \cdot, \cdot \rangle$ denote the norm and scalar product in $G$.
The goal is to recover $f\in \mathbb{H}$. 
Suppose that $A^*A$ ($A^*$ stands for the adjoint of $A$) is a compact operator so that 
it has a complete orthonornal system of eigenvectors $\{\phi_i, \, i \in \mathbb{N}\}$  in $\mathbb{H}$ 
with corresponding eigenvalues $\lambda_i>0$, i.e.,  $A^*A\phi_i = \lambda_i \phi_i$.
Then $\{\psi_i, \, i \in \mathbb{N}\}$, with $\psi_i =\lambda_i^{-1/2} A\phi_i$, is an orthonormal 
basis  in $\mathbb{G}$, and $A^*\psi_i = \lambda^{-1/2}_i A^*A\phi_i =\lambda^{1/2} \phi_i$.
Now, with $\theta_i=\langle f, \phi_i\rangle$, we have 
$A^*Af = A^*A \sum_i \theta_i \phi_i  = 
\sum_i \lambda_i \theta_i \phi_i$, so that
$\langle Af, \psi_i\rangle = \lambda^{-1/2}_i \langle Af, A\phi_i \rangle =
 \lambda^{-1/2}_i \langle A^*Af, \phi_i\rangle =\lambda^{1/2}_i \theta_i$.
Then the Fourier coefficient of $Y$ with respect to $\{\psi_i, \, i \in \mathbb{N}\}$ are 
$Y_i = \langle Y, \psi_i\rangle = \langle Af, \psi_i\rangle+ \varepsilon \langle \xi, \psi_i \rangle 
= \lambda^{1/2}_i \theta_i +\varepsilon \xi_i$, or 
\[
X_i = \theta_i+\sigma_i \xi_i, \quad i \in \mathbb{N},
\]
where $X_i = \lambda^{-1/2}_i Y_i$, $\sigma_i =  \lambda^{-1/2}_i \varepsilon$ 
and $\xi_i$'s are independent $N(0,1)$ random variables.
We thus obtained the inverse signal-in-white-noise 
model (\ref{model}), more details can be found in \cite{Cavalier:2008sup}.

For the remainder of this section, we consider the direct case $\kappa^2_i=1$ 
of model (\ref{model}). This model can also be derived from 
the generalized linear Gaussian model as introduced 
in \cite{Birge&Massart:2001sup}: for some separable Hilbert space $\mathbb{H}$ 
with scalar product $\left\langle \cdot, \cdot \right\rangle$,
\[
Y^{(\varepsilon)}(x)=\left\langle y,x \right\rangle +\varepsilon W(x), \quad x \in \mathbb{H},
\]
where $W$ is a so called isonormal process; 
see the exact definition in \cite{Birge&Massart:2001sup}.
Take any orthonormal basis $\{b_i, \, i \in \mathbb{N}\}$ in $\mathbb{H}$ 
and consider $X_i=Y^{(\varepsilon)}(b_i)$, $i \in \mathbb{N}$, 
to reduce the above model to (\ref{model}).

The following model is known as the \emph{white noise model}. 
We observe a stochastic process 
$Y^{(\varepsilon)}(t)$, $t\in[0,1]$, satisfying the stochastic differential equation
\[
dY^{(\varepsilon)}(t) =f(t)dt +\varepsilon dW(t), \quad t\in [0,1], 
\]where $f \in \mathbb{L}_2([0,1])$ 
is an unknown signal and $W$ is a standard Brownian
motion which represent the noise of intensity $\varepsilon$. 
If $\{b_i(t), \, i \in \mathbb{N}\}$ is an orthonormal basis in $\mathbb{L}_2([0,1])$, 
then the white noise model  can be translated into direct version of model
(\ref{model}) with observations $X_i=\int_0^1b_i(t)dY^{(\varepsilon)}(t)$ and 
parameter $\theta_i=\int_0^1b_i(t) f(t)dt$, $i\in\mathbb{N}$. 
%While interesting in communication theory in its own right, 
%the white noise model also provides a good approximation to a variety of
%curve estimation problems.
 
As the last related example,  we mention the discrete regression model:
\begin{equation}
\label{regr_model}
Y_i= f(x_i)+ e_i, \quad i\in\mathbb{N}_n, 
\end{equation}
where  $e_i$'s are independent $N(0,\sigma^2)$, $x_i\in[0,1]$ are 
deterministic distinct points and $f(t)$ 
is an unknown function.  Let $Y=(Y_1, \ldots, Y_n)^T$, 
$f=(f(x_1),\ldots, f(x_n))^T$, 
$\{b_1,\ldots, b_n\}$ be an orthonormal (column) basis of $\mathbb{R}^n$,
$W=(b_1, \ldots,b_n)^T$.
Denote 
\begin{equation}
\label{W_transform}
X=n^{-1/2}WY, \quad\theta=n^{-1/2}Wf, \quad \varepsilon=n^{-1/2}
\end{equation}
to reduce (\ref{regr_model}) again to the direct version of model (\ref{model}),
with the convention that $\theta=(\theta_1,\ldots,\theta_n,0,0\ldots)$ in (\ref{model}) 
has now zero coordinates starting from $(n+1)$-th position.
Clearly,  $\|\tilde{\theta}-\theta\|^2 =n^{-1} \|\tilde{f}-f\|^2$
for $\tilde{\theta}=n^{-1/2}W\tilde{f}$.

%\begin{remark}
%\label{rem_besov_ball}
If $x_i=i/n$, $n=2^{J+1}$ and $f(t)\in\mathbb{L}_2([0,1])$ in (\ref{regr_model}), 
we can choose a convenient wavelet basis (of regularity $r>0$) 
in $\mathbb{L}_2([0,1])$ and apply 
the corresponding discrete wavelet transform $W$ in (\ref{W_transform})
to the original data $Y$. Assume that the original curve 
$f$ belongs to a certain scale of Besov  balls 
(from Besov space $B^s_{p,q}$, with $\max\{0,1/p-1/2\} <s<r$, $p,q\ge 1$) from 
$\mathbb{L}_2([0,1])$, that include among others H\"older ($B^s_{\infty,\infty}$)
and Sobolev ($B^s_{2,2}$) classes of smooth functions.
Then the corresponding noiseless discrete wavelet transform 
$n^{1/2} \theta= Wf$ belongs to the corresponding scale of 
Besov balls in $\ell_2$. There is a dyadic indexing of vector $n^{1/2} \theta$, but 
it can be reduced to the (direct) setting of (\ref{model}) by an appropriate ordering; 
the details are nicely explained in \cite{Birge&Massart:2001sup}.
%\end{remark}

%\begin{remark}
To give an idea how, according to the equivalence principle,  
the results for the model (\ref{model}) can be conveyed to other
(equivalent) models, let us outline a possible approach to the discrete regression model 
(\ref{regr_model}): %under conditions of Remark \ref{rem_besov_ball}:
\begin{itemize}
\item[1)] consider the discrete regression model 
(\ref{regr_model}) and assume that the unknown signal $f$ belongs
to a Besov ball $B_{p,q}^s(Q)$ with an unknown smoothness $s$; 
\item[2)] apply a discrete wavelet transform, as in (\ref{W_transform}), to the data $Y=(Y_i,\, i \in\mathbb{N}_n)$
from (\ref{regr_model}) to obtain the data $X$ of form (\ref{model}); 
\item[3)] construct the DDM $\mathrm{P}(\theta|X)$ (\ref{ddm1}), obtain all the results for it
in terms of the data $X$;
\item[4)] by (\ref{W_transform}), transform the DDM $\mathrm{P}(\theta|X)$ to the DDM $\mathrm{P}(f|Y)$ 
for the signal $f$, now in terms of the data $Y$ from (\ref{regr_model});
\item[5)] by equivalence of the norms for $\theta$ and $f$, obtain the results for the 
DDM $\mathrm{P}(f|Y)$ from the results for the DDM $\mathrm{P}(\theta|X)$.
\end{itemize}
For example the resulting DDM $\mathrm{P}(f|Y)$ will concentrate around the true $f_0$ 
from the $\mathrm{P}_{f_0}$-perspective at least with the optimal minimax rate 
corresponding to the smoothness $s$. It will take a fair piece of effort to implement 
this outlined approach in details, but conceptually it is a straightforward matter.
%\end{remark} 

\subsection{Checking conditions (\ref{condition_sigma}) for the mildly ill-posed case}

Consider conditions (\ref{condition_sigma}) for the mildly ill-posed case 
$\kappa^2_i=i^{2p}$. As $\sigma^2_i =\varepsilon^2 \kappa^2_i$, these are equivalent 
to the same conditions for the sequence $\kappa^2_i=i^{2p}$. 
In these notations, conditions (\ref{condition_sigma}) can be rewritten as follows:
for any $\rho, \gamma>0$, $\tau_0 >1$, there exist some positive $K_1$, $K_2=K_2(\rho)$, 
$K_3=K_3(\gamma)$, $K_4 \in (0,1)$, $\tau>2$  and $K_5=K_5(\tau_0)$ such that  
%\begin{equation}
%\label{condition_sigma1}
\begin{align*}
&(i)\; n^{2p+1} \le  K_1 \sum_{i=1}^n i^{2p}, 
\quad (ii)\; \sum_{i \le \rho n} i^{2p}\le  K_2
\sum_{i=1}^n i^{2p}, \\
& (iii)\;
\sum_{n=1}^\infty e^{-\gamma n} \Big(\sum_{i=1}^{n} i^{2p} \Big) \le K_3, \quad
(iv)\;
\sum_{i=1}^{\lfloor m/\tau \rfloor} i^{2p} \le (1-K_4) 
\sum_{i=1}^m i^{2p},\\
%\sum_{i=\lfloor m/\tau \rfloor+1}^m i^{2p} \ge K_4 \sum_{i=1}^m i^{2p},\\
&(v)\;
l \lfloor l/\tau_0\rfloor^{2p} \ge K_5\sum_{i=\lfloor l/\tau_0 \rfloor+1}^l i^{2p}, 
%\quad  \text{for all} \quad l,n\in\mathbb{N}, 
\end{align*}
%\end{equation}
hold for all $n\in\mathbb{N}$, all $m \ge \tau$ and all $l \ge \tau_0$.

Let us derive the constants $K_1,K_2,K_3, K_4,\tau, K_5$  
for the  mildly ill-posed case $\kappa_i^2 = i^{2p}$, 
$p\ge 0$.
First, we recall elementary relations:
\begin{align}
\label{sum_i^2p}
\frac{n^{2p+1}}{2p+1} =\int_0^n x^{2p} dx \le \sum_{i=1}^n i^{2p}
\le \int_0^{n+1} x^{2p}dx = \frac{n^{2p+1}(1+\frac{1}{n})^{2p+1}}{2p+1}.
\end{align}

(i) From (\ref{sum_i^2p}) it follows 
$\frac{n^{2p+1}}{2p+1} \le \sum_{i=1}^n i^{2p}$,
%\le \int_0^{n+1} x^{2p}dx = \frac{(n+1)^{2p+1}}{2p+1} \le \frac{2^{2p} I^{2p+1}}{2p+1}$, 
so that $K_1 =2p+1$.

(ii) In view of (\ref{sum_i^2p}), we have 
\[
\sum_{i \le \rho n} i^{2p} \le \frac{(\rho n +1)^{2p+1}}{2p+1} 
\le \frac{n^{2p+1}(\rho+n^{-1})^{2p+1}}{2p+1}\le 
(\rho+1)^{2p+1} \sum_{i=1}^n i^{2p},
\]
so that 
$K_2 = (\rho+1)^{2p+1}$.

(iii) Using (\ref{sum_i^2p}) and and the fact that $\max_{u\ge 0} (e^{-\gamma u} u^p)
=e^{-p} (p/\gamma)^p$, we evaluate 
\begin{align*}
\sum_n e^{-\gamma n} \Big(\sum_{i=1}^{n} i^{2p} \Big)  
&\le  \frac{2^{2p+1}}{2p+1} \sum_n e^{-\gamma n} n^{2p+1}\\
&\le \frac{2^{2p+1}\max_{u\ge 0} (e^{-\gamma u/2} u^{2p+1})}{2p+1}
\sum_n e^{-\gamma n/2} \\
&\le 
\frac{2^{2p+1} e^{-(2p+1)}(2p+1)^{2p+1}}{(2p+1)(\gamma/2)^{2p+1}}
\sum_n e^{-\gamma n/2} \\
&= \frac{4^{2p+1}(2p+1)^{2p}}{(e\gamma)^{2p+1}(e^{\gamma/2}-1)}, 
\end{align*}
that is, $K_3 =\frac{4(8p+4)^{2p}}{(e\gamma)^{2p+1}(e^{\gamma/2}-1)}$.

(iv) 
Denote for brevity $m_\tau=  \lfloor m/\tau\rfloor$. 
%Using (\ref{sum_i^2p}) and the elementary inequality $(1+x)^{-\alpha} \ge 1-\alpha x$ for $x,\alpha\ge 0$, 
%we evaluate
%\begin{align*}
%\sum_{i=m_\tau+1}^m i^{2p} &\ge \frac{m^{2p+1}}{2p+1} - \frac{(m_\tau+1)^{2p+1}}{2p+1}
%\ge \frac{m^{2p+1}}{2p+1} - \frac{m_\tau^{2p+1}(1+\tfrac{1}{m_\tau})^{2p+1}}{2p+1}\\
%&\ge \frac{m^{2p+1}}{2p+1} - \frac{m^{2p+1} 2^{2p+1}}{(2p+1)\tau^{2p+1}} %\big(1+\tfrac{\tau}{I}\big)^{2p+1}  
%=\frac{m^{2p+1}\big(1-(\tfrac{2}{\tau})^{2p+1}\big)}{2p+1}  \\
%&\ge 
%\big(1-(\tfrac{2}{\tau})^{2p+1}\big) \big(1+\tfrac{1}{m}\big)^{-(2p+1)}\sum_{i=1}^m i^{2p}  \\
%&\ge
%\big(1-(\tfrac{2}{\tau})^{2p+1}\big)
%\big(1-\tfrac{2p+1}{m}\big)\sum_{i=1}^m i^{2p}  \\ 
%&\ge \big(1-(\tfrac{2}{\tau})^{2p+1}\big) \big(1-\tfrac{2p+1}{\tau}\big)\sum_{i=1}^m i^{2p} 
%\ge \frac{1}{2} \sum_{i=1}^m i^{2p},
%\end{align*}
%if $1-(\tfrac{2}{\tau})^{2p+1}\ge \frac{3}{4}$ and 
%$1-\tfrac{2p+1}{\tau} \ge \frac{2}{3}$, or 
%$\tau \ge 2^{\frac{2p+3}{2p+1}}$ and $\tau \ge 6p+3$.
%Thus, we obtained $K_4=\frac{1}{2}$ and $\tau$ can be any number 
%satisfying 
%\[
%\tau \ge \max\big\{2^{\frac{2p+3}{2p+1}}, 6p+3\big\}.
%\]
Using (\ref{sum_i^2p}), 
\begin{align*}
\sum_{i=1}^{m_\tau} i^{2p}  &\le
\frac{m_\tau^{2p+1}(1+\tfrac{1}{m_\tau})^{2p+1}}{2p+1}
\le 
\frac{m^{2p+1}(\tfrac{2}{\tau})^{2p+1}}{2p+1} \\
&\le (\tfrac{2}{\tau})^{2p+1} \sum_{i=1}^m  i^{2p}
\le \tfrac{1}{2} \sum_{i=1}^m i^{2p}
\end{align*}
if $(\tfrac{2}{\tau})^{2p+1} \le \frac{1}{2}$, or $\tau \ge 2^{1+1/(2p+1)}$.
Thus, we obtained $K_4=\frac{1}{2}$ and $\tau$ can be any number 
satisfying $\tau \ge 2^{1+1/(2p+1)}$.

(v) Evaluate  
\begin{align*}
\sum_{i=\lfloor l/\tau_0 \rfloor+1}^l i^{2p} 
&\le  l^{2p+1} \le l \big(\tau _0 \lfloor l/\tau_0\rfloor+\tau_0\big)^{2p} \\
&\le  l \big(\tau_0 \lfloor l/\tau_0\rfloor \big)^{2p} 
\Big(1+ \frac{1}{\lfloor l/\tau_0\rfloor} \Big)^{2p} \le l \lfloor l/\tau_0\rfloor^{2p} (2\tau_0)^{2p},
\end{align*} 
so that $K_5 = (2\tau_0)^{-2p}$.

\subsection{Proof of Theorem \ref{th2}} 
%\begin{theorem}[\bf oracle inequality]
%%\label{th2}
%Let the conditions of Theorem \ref{th1} be fulfilled,
%$\theta_0\in\ell_2$, and $\tilde{\theta}$  be defined by (\ref{tilde_estimator}).
%Then there exist  a constant $C_{est}=C_{est}(K,\alpha)\ge 1$
%such that 
%\[
%\mathrm{E}_{\theta_0} \|\tilde{\theta}-\theta_0\|^2
%\le C_{est}\mathnormal{r}^2(\theta_0).
%\]
%\end{theorem}
\begin{proof}[Proof of Theorem \ref{th2}]
The proof of this theorem is essentially contained in the proof of Theorem \ref{th1}.
First recall that, according to (\ref{tilde_estimator}),
$\tilde{\theta}= \mathrm{E}(\theta|X) = \sum_I X(I) \mathrm{P}(\mathcal{I}=I|X)$,
with $X(I) = \{X_i(I), \, i \in \mathbb{N}\} = \{ X_i \mathrm{1}\{i\le I\}, \, i\in\mathbb{N}\}$.
Now, by the Fubini theorem and the Cauchy-Schwarz inequality,
\begin{align*}
\mathrm{E}_{\theta_0} \|\tilde{\theta} -\theta_0\|^2
&=\mathrm{E}_{\theta_0}\sum_i 
\Big( \sum_I X_i(I) \mathrm{P}(\mathcal{I}=I|X) -\theta_{0,i} \Big)^2 
\\&
\le
\mathrm{E}_{\theta_0}\sum_i \sum_I
\big(X_i(I)  -\theta_{0,i} \big)^2 \mathrm{P}(\mathcal{I}=I|X) 
\\&=
\mathrm{E}_{\theta_0}\sum_I \|X(I)-\theta_{0}\|^2 \mathrm{P}(\mathcal{I}=I|X)\\
&=
\mathrm{E}_{\theta_0}\sum_I \Big(\sum_{i\le I}\sigma^2_i\xi_i^2 +
\sum_{i>I}\theta_{0,i}^2\Big) \mathrm{P}(\mathcal{I}=I|X)\\
&\le M^2 \mathnormal{r}^2(\theta_0) \mathrm{E}_{\theta_0}(T_1+T_2 +T_3),
\end{align*}
where $T_1,T_2,T_3$ are defined in (\ref{sum_of_Ts}). In the last step of 
the proof of Theorem \ref{th1}, it is established that  
$\mathrm{E}_{\theta_0}(T_1+T_2 +T_3) \le \frac{C_{or}}{M^2}$.  
The theorem follows with the constant $C_{est}=C_{or}$. 
\end{proof}

The local rate $\mathnormal{r}(I,\theta_0)$ defined by (\ref{local_risks})
is also the $\ell_2$-risk of the projection 
estimator $\hat{\theta}(I)= X(I)$:
$\mathrm{E}_{\theta_0}\|\hat{\theta}(I) -\theta_0\|^2 = 
\mathnormal{r}^2(I,\theta_0)$.
One can regard the oracle rate (\ref{oracle_risk}) 
as the smallest possible risk over the  family of (projection) estimators 
$\hat{\Theta}(\mathbb{N})=\{\hat{\theta}(I), \, I\in\mathbb{N}\}$, namely
\[
\mathnormal{r}^2(\theta_0)=\mathnormal{r}^2(I_o,\theta_0) = \inf_{ I\in\mathbb{N}}
\mathrm{E}_{\theta_0}\|\hat{\theta}(I) -\theta_0\|^2=
\mathrm{E}_{\theta_0}\|\hat{\theta}(I_o ) -\theta_0\|^2.
\]
Theorem \ref{th2} claims basically that the estimator $\tilde{\theta}$  given by (\ref{tilde_estimator})
\emph{mimics the projection oracle estimator $\hat{\theta}(I_o )$}, which is, strictly speaking, not 
an estimator as it depends on the true $\theta_0$ through $I_o=I_o(\theta_0)$.

%%including scales of Sobolev ellipsoids, certain scales of Besov 
%%ellipsoids, some scales of hyperrectangles, tail classes, etc.; 
%examples of scales can be found in  Birg\' e and Massart (2001), 
%Cavalier and Tsybakov (2001) and Babenko and Belitser (2010).
%%See Supplement for more details on this.
%For example, from Theorem \ref{th2} it follows that, with some $c=c(K,\alpha,\beta, Q)$,
%$\sup_{\theta_0 \in \Theta_\beta(Q)} \mathrm{E}_{\theta_0} \|\tilde{\theta}-\theta_0\|^2
%\le c R^2(\Theta_\beta)$, i.e.,
%the estimator $\tilde{\theta}$ is adaptive minimax over the scale of Sobolev ellipsoids.

%The corresponding estimator
%$\hat{\theta}(I_o)$ is called the \emph{oracle estimator}  
%over the family  $\hat{\Theta}(\mathbb{N})$. Strictly speaking, this is not 
%an estimator as it depends on the true $\theta_0$ through 
%$I_o=I_o(\theta_0)$.
%The interpretation of the oracle $I_o(\theta_0)$ is that it selects the first 
%$I_o$ most significant variables. Birg\' e and Massart (2001) call this 
%\emph{ordered variable selection}, this is a suitable strategy when the unknown 
%$\theta_0$ possesses some relevant smoothness structure, e.g.,
%belongs to some (unknown) ellipsoid. 
%The estimation result of Theorem \ref{th2} was obtained earlier for  
%a blockwise Stein's estimator by Cavalier and Tsybakov (2001) and for 
%a penalized estimator by Birg\' e and Massart (2001), 
%%both papers treat this problem 
%within a more general framework. 

\subsection{Notion of covering by a local rate} 

%is smaller than the minimax rate $R(\Theta_\beta)$ uniformly over 
%$\theta_0 \in\Theta_\beta$, for all $\beta \in \mathcal{B}$. 
%This is certainly true if the local rate  $\mathnormal{r}(\theta_0)=
%\inf\{\mathcal{R}(\mathcal{A})\}$,
%where the family of local rates $\mathcal{R}(\mathcal{A})=
%\{\mathnormal{r}(\alpha,\theta_0),\, \alpha \in \mathcal{A}\}$
%is such that for any $\beta\in\mathcal{B}$ there 
%exists an $\alpha=\alpha(\beta)\in \mathcal{A}$ such that 
%$\mathnormal{r}(\alpha(\beta),\theta_0)  \le c R(\Theta_\beta)$ for all
%$\theta_0\in\Theta_\beta$ and some uniform $c$.
%
%This is certainly true if the local rate  $\mathnormal{r}(\theta_0)$
%is the smallest over a such family of local rates
%that for any $\beta\in\mathcal{B}$ there exists a local rate from this family that 
%is smaller than $R(\Theta_\beta)$ uniformly in $\theta_0 \in \Theta_\beta$.
%
%Now observe the following simple fact: if for each $\beta \in \mathcal{B}$ and 
%each $\theta_0 \in \Theta_\beta$, some local rate $\mathnormal{r}(\theta_0)$ is uniformly smaller 
%than the minimax rate $R(\Theta_\beta)$. 

Recall that all the quantities involved depend 
on the information parameter $\varepsilon$, but we skip this dependence here.
Suppose we have a family of local rates $\mathcal{R}(\mathcal{A})=
\{\mathnormal{r}(\alpha,\theta),\, \alpha \in \mathcal{A}\}$, e.g.,
in our case the family defined by (\ref{local_risks}) with $\mathcal{A} = \mathbb{N}$.
Let $\mathnormal{r}(\theta)=\inf_{\alpha \in \mathcal{A}} 
\mathnormal{r}(\alpha,\theta)$ be the smallest local rate  over 
$\mathcal{R}(\mathcal{A})$,  called the \emph{oracle rate}. 
If  $\mathnormal{r}(\theta)= \mathnormal{r}(\alpha_o,\theta)$ for some 
$\alpha_o = \alpha_o(\theta) \in \mathcal{A}$, we call this value \emph{oracle}. 

We say that the family %of local rates 
$\mathcal{R}(\mathcal{A})$ \emph{covers} a scale 
$\Theta(\mathcal{B})=\{\Theta_\beta, \, \beta\in\mathcal{B}\}$
with the corresponding family of minimax rates $\{R(\Theta_\beta),\,\beta\in\mathcal{B}\}$
if for any $\beta\in\mathcal{B}$ there 
exists an $\alpha=\alpha(\beta)\in \mathcal{A}$ such that 
$\mathnormal{r}(\alpha(\beta),\theta)  \le c R(\Theta_\beta)$ for all
$\theta\in\Theta_\beta$ and some uniform $c$. 
Basically, this means that the family $\mathcal{R}(\mathcal{A})$ is rich 
enough to contain the minimax rates over the whole scale
$\Theta(\mathcal{B})$. Then, %of course, 
for all $\beta \in \mathcal{B}$,
\[
\mathnormal{r}(\theta) 
\le c R(\Theta_\beta) \;\;
\text{for all} \;\; \theta \in \Theta_\beta, \quad
 \text{so that} \quad
\sup_{\theta\in\Theta_\beta}\mathnormal{r}(\theta) 
\le c R(\Theta_\beta),
\]
which is the property (\ref{local_impl_global}).
If the above property holds for some local rate $\mathnormal{r}(\theta)$ 
(not necessarily associated with some family of rates), 
we say that the \emph{local rate $\mathnormal{r}(\theta)$ covers 
$\Theta(\mathcal{B})$}. As we already discussed in the paper,
the local results with a local radial rate $\mathnormal{r}(\theta)$  
imply the global  minimax results for all scales which are covered 
by the radial rate $\mathnormal{r}(\theta)$. 
Therefore, in order to motivate the obtained local results, one needs to ensure
this property at least for some interesting scales. 

We can extend the idea of \emph{covering} %(introduced in subsection \ref{ddm_contraction}) 
to two different families of local  rates.   
We say that a family of local rates $\mathcal{R}_1(\mathcal{A})=
\{\mathnormal{r}_1(\alpha,\theta),\, \alpha \in \mathcal{A}\}$
\emph{covers} another family of local rates $\mathcal{R}_2(\mathcal{B})
=\{\mathnormal{r}_2(\beta,\theta),\, \beta \in \mathcal{B}\}$
over some $\Theta_0$
if for each $\theta\in\Theta_0$ and $\beta\in \mathcal{B}$ there exists 
an $\alpha=\alpha(\theta,\beta)$ such that 
for some uniform constant $c=c(\Theta_0,\mathcal{A},\mathcal{B})$
\[
\mathnormal{r}_1(\alpha,\theta)
\le c  \mathnormal{r}_2(\beta,\theta).
\]
This leads of course to the relation between the oracle rates: $\mathnormal{r}_1(\theta) \le 
\mathnormal{r}_2(\theta)$ for all $\theta\in \Theta_0$.
If $\Theta_0$ contains the set of interest (e.g., $\Theta_0=\Theta$ is the whole space), 
then clearly a DDM-contraction result with the oracle rate over the family 
$\mathcal{R}_1(\mathcal{A})$ will immediately imply the 
DDM-contraction result with the oracle rate over the family 
$\mathcal{R}_2(\mathcal{B})$. 
%We say the family of local rates $\mathcal{R}'_(\mathcal{A})$ 
%is \emph{stronger} than $\mathcal{R}''(\mathcal{B})$. 

For example, it can be easily shown that our family of local rates  
$\mathcal{R}(\mathbb{N})=
\{\mathnormal{r}(I,\theta),\, I \in \mathbb{N}\}$
defined by (\ref{local_risks}) covers the family of local radial rates
$\mathcal{R}_1(\mathbb{R}_+)=
\{\mathnormal{R}_{lin}(\lambda,\theta),\, \lambda\in\Lambda_1(\mathcal{R}_+)\}$, where
$\mathnormal{R}^2_{lin}(\lambda,\theta)=\sum_i \big[\sigma^2_i\lambda_i^2
+(1-\lambda_i)^2 \theta_i^2 \big]$ is the risk of the linear estimator 
$\hat{\theta}(\lambda) =(\lambda_i X_i, \, i \in \mathbb{N})$
with the weights $\lambda=( \lambda_i,\, i \in \mathbb{N})$,
and
\[
\Lambda_1(\mathcal{R}_+)= \Big\{\lambda(\beta)=( \lambda_i(\beta),\, i \in \mathbb{N}): \,
\lambda_i(\beta)= \frac{i^{-(2\beta+1)}}{\sigma^2_i+i^{-(2\beta+1)}},\, 
\beta \in \mathcal{R}_+\Big\}.
\]
This is the family of the risks of the minimax estimators over the Sobolev smoothness scale
$\{ \mathcal{E}_S(\beta,Q), \, \beta>0\}$, where  $\mathcal{E}_S(\beta,Q)$ is defined by 
(\ref{sob_ell}).
This is also the family of posterior convergence rates 
for the prior $\theta \sim \pi_\beta= \bigotimes _i N(0,i^{-(2\beta+1)})$;
cf.\  \cite{Szabo&etal:2015sup} and \cite{Belitser&Ghosal:2003sup} 
(for the direct case $\kappa^2_i=1$).

In fact, $\mathcal{R}(\mathbb{N})$ covers even the richer
family of local rates $\mathcal{R}_2(\Lambda_{mon}) =
\{\mathnormal{R}_{lin}(\lambda,\theta),\, 
\lambda\in \Lambda_{mon}\}$, where 
\begin{align}
\label{Lambda_mon}
\Lambda_{mon}=
\big\{\lambda=( \lambda_i,\, i \in \mathbb{N}): \,
\lambda_i\in[0,1], \, \lambda_i \ge \lambda_{i+1}, \, i \in \mathbb{N})\big\}.
\end{align}
This is the family of risks of the linear estimators $\hat{\theta}(\lambda)$,
with monotone weights $\lambda\in \Lambda$.
Indeed, for any $\lambda \in \Lambda_{mon}$ take 
$N_\lambda =\max\{i:\, \lambda_i \ge 1/2\}$ to derive
\begin{align*}
\mathnormal{R}^2_{lin}(\lambda,\theta) &= 
 \sum_i \big[ \sigma^2_i \lambda_i^2+(1-\lambda_i)^2 \theta_i^2\big] 
\ge\sum_{i \le N_\lambda}\frac{\sigma^2_i}{4} +
\sum_{i >N_\lambda+1} \frac{\theta_i^2}{4}\\
&=\frac{\mathnormal{r}^2(N_\lambda,\theta)}{4}
\ge  \frac{\mathnormal{r}^2(I_o,\theta)}{4}.
\end{align*}
Clearly, $\mathcal{R}_1(\mathbb{R}_+) \subset \mathcal{R}_2(\Lambda)$.
Besides, $ \mathcal{R}_2(\Lambda_{mon})$ contains also 
the family of risks of the minimax Pinskers estimators (which are 
asymptotically minimax  over Sobolev ellipsoids up to the constant)
and the family of risks of  the (minimax) Tikhonov regularization estimators, 
which correspond to spline estimators  in the problem of curve estimation. 
%\end{remark} 
 
\subsection{Proof of (\ref{loc<global})} 

Recall the definitions (\ref{ellipsoids}) of ellipsoid $\mathcal{E}(a)$ 
and hyperrectangle $\mathcal{H}(a)$. 
First consider the hyperrectangles $\mathcal{H}(a)$. It follows from \cite{Donoho&etal:1990sup} that  
\begin{align*}
R^2(\mathcal{H}(a)) &= \inf_{\hat{\theta}}\sup_{\theta \in \mathcal{H}(a)} 
\mathrm{E}_\theta \|\hat{\theta}-\theta\|^2  \ge 
\frac{4}{5}\inf_b \sup_{\theta \in \mathcal{H}(a)} 
\mathrm{E}_\theta \|\tilde{\theta}(b)-\theta\|^2 \\
&=
\frac{4}{5} \sum\nolimits_i \frac{a_i^2 \sigma^2_i}{a_i^2+\sigma^2_i},
%\ge  \frac{4}{5}\sum_i \frac{(a_i^2 \wedge \varepsilon^2)}{2}
\end{align*}
where $\tilde{\theta}(b)=(\tilde{\theta}_i(b), \, \in \mathbb{N})$, 
$b=(b_i\in \mathbb{R},\, i \in \mathbb{N})$, is the class of linear estimators 
$\tilde{\theta}_i(b)=b_i X_i$. Take $N_a=\max\{i: \, \sigma^2_i \le a_i^2\}$, 
then for any $\theta_0 \in \mathcal{H}(a)$ (for some \emph{unknown} $a$) we have
\begin{align*}
\sum_i \frac{a_i^2 \sigma^2_i}{a_i^2+\sigma^2_i} 
&\ge  
\sum_{i\le N_a}\frac{\sigma^2_i}{2} +\sum_{i>N_a} \frac{a_i^2}{2} \ge 
\frac{1}{2} \inf_I 
\Big\{\sum_{i\le N_a}\sigma^2_i+\sum_{i>I} a_i^2\Big\} \\
&\ge
\frac{1}{2} \inf_I \Big\{\sum_{i\le I}\sigma^2_i+\sum_{i>I} \theta_{0,i}^2\Big\}
=\frac{\mathnormal{r}^2(\theta_0)}{2}.
\end{align*}
Combining the last two relation yields the second bound in (\ref{loc<global}).

Now  suppose $\theta_0\in \mathcal{E}(a)$ for some \emph{unknown} $a$. 
From \cite{Belitser&Levit:1995sup} it follows that  
for any $\theta_0\in \mathcal{E}(a)$ and some $\mathring{\lambda}
=(\mathring{\lambda}_i, \, i \in \mathbb{N})\in \Lambda_{mon}$  
($\Lambda_{mon}$ is defined by \eqref{Lambda_mon})):
%i.e., $\mathring{\lambda}_i \in [0,1] $ and $\mathring{\lambda}_i 
%\ge \mathring{\lambda}_{i+1}$, $i\in\mathbb{N}$):
\begin{align*}
R^2(\mathcal{E}(a)) &\ge 
%\sup_{\theta \in \mathcal{E}(a)} \sum_i \frac{\sigma^2_i \theta^2_i/\pi^2}{\sigma^2_i+\theta^2_i/\pi^2}\ge  
\inf_\lambda \sup_{\theta \in \mathcal{E}(\pi a)} \mathnormal{R}_{lin}(\lambda,\theta) \\
&\ge \pi^{-2} \inf_\lambda \sup_{\theta \in \mathcal{E}(a)} \mathnormal{R}_{lin}(\lambda,\theta) 
=\pi^{-2}  \sup_{\theta \in \mathcal{E}(a)} \mathnormal{R}_{lin}(\mathring{\lambda},\theta) \\
&=
\pi^{-2}  
\sup_{\theta \in \mathcal{E}(a)}  \sum_i 
\big[\mathring{\lambda}^2_i \sigma^2_i+ (1-\mathring{\lambda}_i)^2 \theta^2_i\big]\\
&\ge
\pi^{-2}\Big[\sum_{i:\, \mathring{\lambda}_i \ge 1/2} \frac{\sigma^2_i}{4} + 
\sup_{\theta \in \mathcal{E}(a)} \sum_{i: \mathring{\lambda}_i <1/2} \frac{\theta^2_i}{4} \Big]\\
&\ge 
(2\pi)^{-2} \Big(\sum_{i\le N_{\mathring{\lambda}}} \sigma^2_i + 
 a^2_{N_{\mathring{\lambda}}+1}  \Big)
\ge (2\pi)^{-2}
\inf_I \Big\{\sum_{i\le I} \sigma^2_i + a^2_{I+1}\Big\}\\
&\ge 
(2\pi)^{-2}\ \inf_I\Big\{\sum_{i\le I}\sigma^2_i+ \sum_{i>I} \theta_{0,i}^2\Big\} 
=(2\pi)^{-2}\mathnormal{r}^2(\theta_0),
\end{align*}
which leads to the first bound in (\ref{loc<global}).

The exact form of weights $\mathring{\lambda}\in \Lambda_{mon}$ 
is not important, but we just remind here that these are the so called 
Pinsker optimal weights (cf.\ \cite{Pinsker:1980sup}): 
$\mathring{\lambda}_i = (1-\mathring{\mu}/a_i)_+$, where $x_+ = x\vee 0$
and $\mathring{\mu}=\mathring{\mu}(\sigma,a)$ is the unique solution of
the equation  
\[
\sum_i \sigma^2_i (1-\mathring{\mu}/a_i)_+ /(a_i\mathring{\mu}) =1.
\] 
The constant $(2\pi)^{-2}$ is actually too conservative. For example, for the direct case 
$\kappa^2_i =1$, it follows from \cite{Donoho&etal:1990sup} 
(see also Proposition 3 in \cite{Birge&Massart:2001sup}) that  
\begin{align*}
R^2(\mathcal{E}(a)) &= \inf_{\hat{\theta}}\sup_{\theta \in \mathcal{E}(a)} 
\mathrm{E}_\theta \|\hat{\theta}-\theta\|^2 \ge (4.44)^{-1} 
\inf_I\{I\varepsilon^2+a^2_{I+1}\} \\
&\ge (4.44)^{-1} \inf_I\Big\{I \varepsilon^2+ \sum_{i>I} \theta_{0,i}^2\Big\}
\ge  (4.44)^{-1} \mathnormal{r}^2(\theta_0). 
\end{align*} 

%\begin{remark}
%\label{rem_minimax_rates}
Since $R^2(\mathcal{E}(a)) \le \inf_I \sup_{\theta \in \mathcal{E}(a)} 
\mathrm{E}_\theta \|X(I)-\theta\|^2  = \inf_I \big\{\sum_{i\le I}\sigma^2_i+a^2_{I+1}\big\}$
and $R^2(\mathcal{H}(a)) \le \inf_I \sup_{\theta \in \mathcal{H}(a)} 
\mathrm{E}_\theta \|X(I)-\theta\|^2  =\inf_I \big\{\sum_{i\le I}\sigma^2_i+\sum_{i>I} a_i^2\big\}$, 
we conclude that 
\begin{align}
\label{minimax_rates}
R^2(\mathcal{E}(a)) \asymp  \inf_I \Big\{\sum_{i\le I}\sigma^2_i+a^2_{I+1}\Big\}, \;
R^2(\mathcal{H}(a))  \asymp \inf_I \Big\{\sum_{i\le I}\sigma^2_i+\sum_{i>I} a_i^2\Big\}.
\end{align}
%\end{remark}

%???
%for the scale of Sobolev ellipsoids $\Theta_\beta =\Theta_\beta(Q)= 
%\big\{ \theta \in \ell_2: \, \sum_{i>1} \theta_i^2 i^{2\beta} \le Q \big\}$, $\beta>0$, 
%it is known that the minimax 
%(quadratic) rate $R^2(\Theta_\beta) = \inf_{\hat{\theta}}\sup_{\theta \in \Theta_\beta} 
%\mathrm{E}_\theta \|\hat{\theta}-\theta\|^2 \ge C(\beta,Q) \varepsilon^{4\beta/(2\beta+1)}$.
%For $I_\beta= \lfloor c \varepsilon^{-2/(2\beta+1)}\rfloor$ 
%($\lfloor a\rfloor = \max\{k\in\mathbb{Z}:\, k \le a\}$), 
%we obtain 
%$
%\mathnormal{r}^2(\theta_0) \le \mathnormal{r}^2(I_\beta,\theta_0)  = 
%\varepsilon^2 I_\beta + \sum_{i> I_\beta} \theta_{0,i}^2 \le 
%C\varepsilon^{4\beta/(2\beta+1)} \le c R^2(\Theta_\beta)$
%for any $\theta_0\in\Theta_\beta$, so that Theorem \ref{th1} implies  that,
%with some $C=C(K,\alpha,\beta,Q)$,
%\[
%\sup_{\theta_0\in \Theta_\beta(Q)} 
%\mathrm{E}_{\theta_0} \mathrm{P}\big(\|\theta-\theta_0\|
%\ge M R(\Theta_\beta)\, \big| X\big) \le \frac{C}{M^2}.
%\]  

%We provide some more details and examples %(of covering and covered families of rates and scales) 
%in Supplement.
%This is discussed in some detail in Supplement,
%Details can be found in  Birg\' e and Massart (2001), 
%Cavalier and Tsybakov (2001) and Babenko and Belitser (2010).

%\subsection{Implications: oracle estimation inequality}
%\label{subsec_est_oracle}

\subsection{Other choices for DDM, over-shrinkage effect} 
\label{rem_choice}

Notice that we do observe the Bayesian tradition as our DDM $\mathrm{P}(\cdot|X)$ 
defined by (\ref{ddm1}) results from certain empirical Bayes posterior. 
However, in principle we can manipulate with different ingredients in constructing 
DDMs: different choices for $\mathrm{P}_I(\cdot|X)$ and $\mathrm{P}(\mathcal{I}=I|X)$ 
in (\ref{ddm1}) are possible,  not necessarily coming from the (same) Bayesian approach.

In (\ref{tau_i}), we could take
$\tau_i^2(I) =%\tau_i^2(I,\varepsilon,K_1,K_2)=    
K_1 \sigma^2_i \mathrm{1}\{i\le I\}+ 
K_2 \sigma^2_i \mathrm{1}\{i> I\}$ for 
some $0\le K_2< K_1$ (the choice in (\ref{tau_i}) is a particular case with $0=K_2<K_1=K$), 
to possibly improve constants in the main results by choosing appropriate $K_2$,
but this would further complicate the expressions 
without gaining anything conceptually.  

\paragraph{Empirical Bayes posterior with respect to $I$}
One more choice for DDM within Bayesian tradition is the
empirical Bayes posterior $\hat{\mathrm{P}}(\cdot|X)$ with respect to $I$
introduced by (\ref{ddm2}) in Remark \ref{rem_ddm2}. We remind its definition:
\begin{align*}
%\label{ddm2}
\hat{\mathrm{P}}(\cdot|X)=\mathrm{P}_{\hat{I}}(\cdot|X),\quad \text{with} \quad
\hat{I}=\min\Big\{\arg\!\max_{I\in\mathbb{N}} \mathrm{P}(\mathcal{I}=I|X) \Big\},
%\min\big\{ I \in \mathbb{N}:\,\mathrm{P}(\mathcal{I}=I|X)
%=\max_J \mathrm{P}(\mathcal{I}=J|X)\big\},
\end{align*}
where $\mathrm{P}_I(\cdot|X)$ and $\mathrm{P}(\mathcal{I}=I|X)$ 
are defined by respectively (\ref{measure_P_I}) and (\ref{p(I|X)}). The
$\arg\!\max$ gives a subset of $\mathbb{N}$ in general, $\hat{I}$ is 
the smallest element in this set.

Let us demonstrate that the DDM $\hat{\mathrm{P}}(\cdot|X)$
has exactly the same properties as the DDM $\mathrm{P}(\cdot|X)$ 
defined by (\ref{ddm1}). 
%When beginning with the proof of Theorem \ref{the} for the 
%DDM $\hat{\mathrm{P}}(\cdot|X)$, 
By the definition of $\hat{I}$, we derive that, for any $I,I_0 \in \mathbb{N}$ 
and any $ h\in[0,1]$,
\[
\mathrm{P}_{\theta_0}(\hat{I} =I) \le 
\mathrm{P}_{\theta_0} \Big(\frac{\mathrm{P}(\mathcal{I}=I|X)}
{\mathrm{P}(\mathcal{I}=I_0|X)} \ge 1\Big) \le 
\mathrm{E}_{\theta_0} \Big[ \frac{\mathrm{P}(\mathcal{I}=I|X)}
{\mathrm{P}(\mathcal{I}=I_0|X)} \Big]^h, 
%=\mathrm{E}_{\theta_0}\Big[\frac{\lambda_I 
%\bigotimes_i  \varphi(X_i,X_i(I),\tau_i^2(I)+\varepsilon^2)}
%{ \lambda_{I_0} \bigotimes_i \varphi(X_i,X_i(I_0),\tau_i^2(I_0)
%+\varepsilon^2)} \Big]^h
\]
which yields the analogue of (\ref{th1_lem3}). From this point on,
the proof of the properties of the DDM $\hat{\mathrm{P}}(\cdot|X)$
proceeds exactly in the same way as the proof for
the DDM $\mathrm{P}(\cdot|X)$ defined by (\ref{ddm1}),
with the only difference that everywhere (in the claims and in the proofs), 
$\mathrm{1}\{\hat{I}=I\}$ is substituted instead of $\mathrm{P}(\mathcal{I}=I|X)$
and $\mathrm{P}_{\theta_0} (\hat{I} = I)$ 
is substituted instead of $\mathrm{E}_{\theta_0} \mathrm{P}(\mathcal{I} =I|X)$.

An interesting connection of this DDM to penalized estimators is discussed 
in Subsection \ref{subsec_connect}. 

\paragraph{Other choices for $\mathrm{P}_I(\cdot|X)$}
For example, if we only were interested in the upper bound result (Theorem \ref{th1}) 
for the resulting $\mathrm{P}(\cdot|X)$ (\ref{ddm1}),
instead of the DDM (\ref{measure_P_I}) we could use
$\mathrm{P}_I(\cdot|X) =\bigotimes_i N(X_i\mathrm{1}\{i\le I\}, 
\sigma_i^2(I) \})$ with any variances $\sigma_i^2(I) $ such that 
$\sum_i \sigma_i^2(I)  \le  C\sum_{i\le I} \sigma^2_i$ 
for some $C>0$. Even the degenerate DDM  $\mathrm{P}_I(\cdot|X)$ with 
$\sigma^2_i(I)=0$ (or $L=0$ in (\ref{measure_P_I})) 
would lead to the oracle DDM-contraction rate.  
On the other hand, this choice would however make the lower 
bound result (Theorem \ref{th2a_posterior}) impossible to hold.  
In fact, non-normal distributions in the construction of the DDMs
$\mathrm{P}_I(\cdot|X)$ are also possible as we only use 
the Markov inequality when dealing with $\mathrm{P}_I(\cdot|X)$, 
just the right choice of the first two moments would be sufficient
for the upper bound result.
However, when proving the lower bound result, Theorem \ref{th2a_posterior},
we need to deal with a small ball probability, which is a relatively well studied problem for 
the Gaussian distribution. For a non-normal case, one would first have to derive 
small ball probability results.

\paragraph{Other choices for $\mathrm{P}(\mathcal{I}=I|X)$}
Instead of the mixing DDM $\mathrm{P}(\mathcal{I}=I|X)$ (\ref{p(I|X)}) in (\ref{ddm1}), 
the main results would also hold for the following DDM: 
%that can also do the same job of mimicking the oracle. 
\begin{equation}
\label{P(I|x)}
\Pi'(\mathcal{I}=I|X) = 
\frac{\lambda_I \bigotimes_i  \varphi(X_i,0, \sigma^2_i+\tau_i^2(I))}
{\sum_J  \lambda_J \bigotimes_i  \varphi(X_i,0, \sigma^2_i+\tau_i^2(J))},
\quad I\in\mathbb{N},
\end{equation}
with $\tau^2_i(I)$ and $\lambda_I$ defined by (\ref{tau_i}).
%There is only a minor difference in the proof of Theorem \ref{th1} (in step 1) 
%for this DDM, as compared to DDM (\ref{p(I|X)}).
The DDM $\Pi'(\mathcal{I}=I|X)$ defined by (\ref{P(I|x)}) is nothing else but
the posterior probability of $\mathcal{I}$  with respect to the prior  
(\ref{prior_pi}) with $\mu_i(I)=0$ for all $i,I\in\mathbb{N}$; we denote this prior by $\Pi'$.
%It was introduced by Babenko and Belitser (2010) with $K=1$ and 
%$\alpha \in \big[\frac{1}{6} -\log\big(\frac{2}{\sqrt{3}}\big), \frac{1}{2} \big]$.
The right hand side of (\ref{P(I|x)}) means the $\mathrm{P}_{\theta_0}$ almost sure limit  
\begin{align*}
\Pi'(\mathcal{I}=I|X) &=\lim_{m\to\infty}\Pi'(\mathcal{I}=I|X_1,\ldots X_m)\\
&=
\lim_{m\to \infty} \frac{ \lambda_I\bigotimes_{i=1}^m \varphi(X_i,0, \sigma^2_i+\tau_i^2(I))}
{\sum_J  \lambda_J\bigotimes_{i=1}^m  \varphi(X_i,0, \sigma^2_i+\tau_i^2(J))},
\end{align*}
which exists by the martingale convergence theorem.

\paragraph{``Over-shrinkage'' effect of (mixtures of) normal priors}
Although the prior  $\Pi'$ (the prior  defined by
(\ref{prior_pi}) with $\mu_i(I)=0$ for all $i,I\in\mathbb{N}$) leads to the 
``correct'' posterior (\ref{P(I|x)}) on $I$ (in the sense that it can be used 
instead of the DDM (\ref{p(I|X)}) in (\ref{ddm1})),
it yields the ``over-shrunk'' resulting posterior on $\theta$. Indeed,  
\begin{align}  
\label{measure_P_I2}      
\Pi'_I(\cdot |X) =\Pi'(\cdot| X, \mathcal{I} =I)=
\bigotimes\nolimits_i N\big(LX_i(I), L\sigma^2_i\mathrm{1}\{i\le I\}\big),
\end{align} 
%\begin{remark}
%For the direct case $\kappa^2_i=1$, the DDM (\ref{ddm1}) 
%with mixing DDM $\mathrm{P}(\mathcal{I}=I|X)$ defined by (\ref{P(I|x)})
%was introduced and studied 
%in Babenko and Belitser (2010) for  values $L=1/2$ and $K_1=1$
%and $\alpha \in \big[\frac{1}{6} -\log\big(\frac{2}{\sqrt{3}}\big), \frac{1}{2} \big]$.
%For arbitrary $K,\alpha >0$ and $L=\frac{K}{K+1}<1$, (\ref{ddm1})  
%can be interpreted as the posterior distribution of $L^{-1}\theta$
%with respect to the prior $\pi_{K,\alpha}$ defined by (\ref{prior_pi}) with $\mu_i(I) =0$, $i \in \mathbb{N}$. 
%Indeed, even if $\mathrm{P}(\mathcal{I}=I|X)$ perfectly mimics the oracle $I_o$, i.e.,
%concentrates on the oracle, the resulting actual posterior 
%$\mathrm{P}(\cdot|X)=\sum\nolimits_I\mathrm{P}_I(\cdot|X)\mathrm{P}(\mathcal{I}=I|X)$
%will contract at best  to $L\theta_0$ and not to $\theta_0$, that is it is not even consistent.
with $L=\frac{K}{K+1}<1$, so that
the actual resulting posterior of $\theta$
\begin{align}  
\label{posterior_Pi'}     
\Pi'(\cdot|X)=
\sum\nolimits_I\Pi'_I(\cdot|X)\Pi'(\mathcal{I}=I|X)
\end{align}
contracts, from the $\mathrm{P}_{\theta_0}$-perspective, 
to $L\theta_0$ and not to $\theta_0$.
This has to do with the \emph{shrinkage effect} of (mixtures of) normal priors 
towards the prior mean, which is inherent to the normal-normal model. 
This has already been observed in \cite{Johnstone&Silverman:2004sup},
and discussed at length by \cite{Babenko&Belitser:2010sup} and 
\cite{Castillo&vanderVaart:2012sup}. The approaches in the first and 
third papers are based on (mixtures of) heavy-tailed priors instead of normal. 
%Yet another 
A related approach is to add one more level of hierarchy in (\ref{prior_pi}) 
by putting a heavy-tailed prior on variances $\tau^2_i(I)$.
This will of course again destroy the normal conjugate structure of the prior,
whereas normal/mixture-of-normals model has an advantageous feature that all the quantities 
involved can be explicitly computed and controlled. 

Basically, a ``correct'' DDM $\Pi'_I(\cdot|X)$ in  the expression (\ref{posterior_Pi'}) 
should be of the form  $\Pi'_I(\cdot |X)=\bigotimes\nolimits_i 
N\big(X_i(I), L\sigma^2_i\mathrm{1}\{i\le I\}\big)$ for any $L>0$.
Within the DDM methodology, one can, in principle, adjust
the posterior (\ref{measure_P_I2}) by blowing it up (by the factor $L^{-1}$) 
or by shifting it (by the factor $(1-L)X(I)$), or one can simply use the DDM 
(\ref{measure_P_I}) instead of (\ref{measure_P_I2}). However, such %ad-hoc 
manipulations with posteriors are not done by the %hard core 
committed Bayesians. If one insists on normal mixture prior  and wants to get 
a correct posterior (\ref{measure_P_I2}), the only way to achieve this is to 
take the prior variances $\tau^2_i \gg \sigma^2_i$, 
in the asymptotic sense as $\varepsilon \to 0$, so that $L \approx 1$. 
However, this makes the whole consideration necessarily asymptotic. 
A more important issue with this approach
is that we were unable to derive good concentration properties for the 
posterior $\Pi'(\mathcal{I}=I|X)$ in this case. 

Thus, for a Bayesian who would like to use normal/mixture-of-normals model, 
there is a following dilemma: if the prior variances $\tau^2_i$ are of order $\sigma^2_i$, 
we obtain a ``correct'' $\Pi'(\mathcal{I}=I|X)$, but over-shrunk (towards prior mean) 
$\Pi'_I(\cdot|X)$'s; on the other hand, if $\tau^2_i \gg \sigma^2_i$, 
$\Pi'_I(\cdot|X)$'s are then ``correct'', but $\Pi'(\mathcal{I}=I|X)$ does not posses 
good concentration properties (at least we were unable to establish this).

The empirical Bayes approach resolves this issue, also within the Bayesian 
paradigm, as we demonstrated in the paper. The idea is  to treat 
the prior means as parameters  chosen by the empirical 
Bayes procedure, which removes the over-shrinkage effect.

\subsection{Connection to penalized estimators}
\label{subsec_connect}
%\begin{remark}
%\label{alt_conf_ball}
In view of Remark \ref{rem_ddm2}, %all the above results hold also 
Theorem \ref{th6} also holds for the DDM  $\mathrm{P}_{\hat{I}}(\cdot|X)$ 
defined by (\ref{ddm2}), instead of the DDM  $\mathrm{P}(\cdot|X)$ given by (\ref{ddm1}).
If we take the DDM-expectation with respect to the 
DDM  $\mathrm{P}_{\hat{I}}(\cdot|X)$ (like we did in (\ref{tilde_estimator}) 
for the DDM $\mathrm{P}(\cdot|X)$), we obtain the estimator 
\[
\hat{\theta} = X(\hat{I})=(X_i\mathrm{1}\{i\le \hat{I}\}, i\in\mathbb{N}).
\] 
In the direct case $\kappa_i^2=1$, some basic computations reveal that 
$\hat{I}$ is the minimizer of  
\[
\mathit{crit}(I)=-\|X(I)\|^2 %\sum_{i\le I} X_i^2
+(\log(K+1) +2 \alpha) \varepsilon^2 I,
\]
so that $\hat{\theta} = X(\hat{I})$ turns out to be 
the so called \emph{penalized projection estimator} 
with the penalty constant $P(K,\alpha)=\log(K+1) +2 \alpha$,
%for the \emph{ordered variable selection} case 
studied by \cite{Birge&Massart:2001}.

Interestingly, the conditions $K\ge 1.87$ and $a(K)>\alpha>0$,
coming from Theorems \ref{th1} and \ref{th2a_posterior}, lead to the following range for the
penalty constant: $P(K,\alpha)\in [1.05, 1.2]$. 
%This range results from imprecise bounds in the proof (the results probably hold true 
%for a bigger range). 
Although this is probably not the most precise range, 
the fact itself (that $P(K,\alpha)\in [1.05, 1.2]$) 
reconfirms, from a different perspective, 
the conclusion of \cite{Birge&Massart:2001sup} that the penalty constant should certainly 
be bigger than 1, but not too large.

\subsection{Relation to the results of \cite{Szabo&etal:2015sup}}

Here we demonstrate that our local results 
for the DDM $\mathrm{P}(\cdot|X)$ defined by (\ref{ddm1}) imply, 
among others, the non-asymptotic versions of the global minimax results 
obtained in the intriguing paper by 
Szab\' o, van der Vaart and van Zanten \cite{Szabo&etal:2015sup}. 
In our notations, the observations in \cite{Szabo&etal:2015sup} are
$X' = %X^{(n)}=
(X'_i, \, i \in \mathbb{N})\sim \mathrm{P}_{\theta_0}=\mathrm{P}_{\theta_0}^{(n)}=
\bigotimes_i N(\theta_{0,i} \kappa^{-1}_i,n^{-1})$, which is effectively the same model as (\ref{model})
with $X'_i=\kappa^{-1}_i X_i$ and $n^{-1/2} = \varepsilon$.
A family of priors on $\theta$ is considered in \cite{Szabo&etal:2015sup}: 
$\Pi_\alpha=\bigotimes_i N\big(0,i^{-(2\alpha+1)}\big)$, 
$\alpha \in [0,A]$, leading to the posteriors $\Pi_\alpha(\cdot |X')$ 
and the marginal distributions $X'\sim \Pi_{\alpha, X'}$ with the (marginal) 
likelihood $\ell_n(\alpha) =\ell_n(\alpha,X')$. The proposed DDM  
is $\Pi_{\hat{\alpha}_n} (\cdot |X')$ with $\hat{\alpha}_n = \arg\!\max_{\alpha\in[0,A]} \ell_n(\alpha)$,
which is the empirical Bayes posterior with respect to the parameter $\alpha$.
The DDM $\Pi_{\hat{\alpha}_n} (\cdot |X')$ is then used to construct a DDM-credible ball 
whose coverage and size properties were studied. 

The main results in \cite{Szabo&etal:2015sup} are the asymptotic 
(as $n\to \infty$ or, in our notations, as $\varepsilon \to 0$) 
versions of the minimax framework (\ref{adapt_conf_ball_problem}), 
for $\Theta'_{cov} =\Theta_{pt}$ and the following four scales:  
Sobolev hyperrectangles $\mathcal{H}_S$, Sobolev ellipsoids $\mathcal{E}_S$  
and the two \emph{supersmooth} scales, analytic ellipsoids $\mathcal{E}_A$ 
and parametric hyperrectangles $\mathcal{H}_P$
(the notations in \cite{Szabo&etal:2015sup} are $\Theta^\beta(Q)$, 
$S^\beta(Q)$, $S^{\infty,c,d}(Q)$ and $C^{00}(N_0,Q)$, respectively).
Precisely, let $Q,\beta,c,d>0$, $N_0\in\mathbb{N}$, and $\mathcal{E}(a), \mathcal{H}(a)$ 
be defined by (\ref{ellipsoids}). Then 
\begin{subequations}
\begin{align}
\label{sob_hyper}
&\mathcal{H}_S=\mathcal{H}_S(\beta,Q)
=\mathcal{H}(a) \;\; \text{with} \;\; a_i^2= Q i^{-(2\beta+1)}, \\
\label{sob_ell}
&\mathcal{E}_S=\mathcal{E}_S(\beta,Q)=
\mathcal{E}(a) \;\; \text{with} \;\;  a^2_i=Q i^{-2\beta}, \\ 
\label{an_ell}
&  
\mathcal{E}_A=\mathcal{E}_A(c,d,Q)=\mathcal{E}(a) \;\; \text{with} \;\; a_i^2=Qe^{-ci^d},\\
\label{par_hyper}
&\mathcal{H}_P=\mathcal{H}_P(N_0, Q)
=\mathcal{H}(a) \;\; \text{with} \;\; a_i^2=Q \mathrm{1}\{i\le N_0\}.
\end{align} 
By using (\ref{minimax_rates}), it is easy to compute the corresponding minimax rates 
over these scales, under the asymptotic regime $\varepsilon \to 0$ (or, $n\to \infty$):
\begin{align*}
R^2(\mathcal{E}_S) &\asymp \varepsilon^{4\beta/(2\beta+2p+1)} 
=n^{-2\beta/(2\beta+2p+1)}, \\
R^2(\mathcal{H}_S) &\asymp \varepsilon^{4\beta/(2\beta+2p+1)}=
n^{-2\beta/(2\beta+2p+1)},\\
R^2(\mathcal{E}_A) &\asymp \varepsilon^2 (\log \varepsilon^{-1})^{(2p+1)/d} 
= \frac{(\log n)^{(2p+1)/d}}{n}, \\
R^2(\mathcal{H}_P) &\asymp  \varepsilon^2 =n^{-1}.
\end{align*}
\end{subequations}
Notice that the parametric class $\mathcal{H}_P$ automatically satisfies EBR.

The DDM $\Pi_{\hat{\alpha}_n} (\cdot |Y)$ is well suited to model Sobolev-type scales: 
the optimal (minimax) radial rates are obtained in the size relation of 
(\ref{adapt_conf_ball_problem}) for Sobolev hyperrectangles $\mathcal{H}_S$ and 
ellipsoids $\mathcal{E}_S$, but only suboptimal rates for 
the two supersmooth scales $\mathcal{E}_A$ and $\mathcal{H}_P$:
\begin{align*}
\frac{(\log n)^{(p+1/2)\sqrt{\log n}}}{n}
&\gg R^2(\mathcal{E}_A)
=\frac{(\log n)^{(2p+1)/d}}{n} , \\
\frac{e^{(3p+3/2)\sqrt{\log N_0}\sqrt{\log n}}}{n} 
&\gg R^2(\mathcal{H}_P)=\frac{1}{n}.
\end{align*}

%{\footnotesize  
%\begin{tabular}{lcc} %\hline
%Scale  & definition of the scale &notation in \cite{Szabo:2015}  \\ \hline 
%Sobolev ellipsoid & $ \mathcal{E}_S=\mathcal{E}_S(\beta,Q)=
%\mathcal{E}(a) \;\; \text{with} \;\;  a^2_i=Q i^{-2\beta}$ & $S^\beta(Q)$ \\  %\hline 
%analytic ellipsoid& $\mathcal{E}_A=\mathcal{E}_A(c,d,Q)=\mathcal{E}(a) \;\; \text{with} 
%\;\; a_i^2=Qe^{-ci^d}$ & $S^{\infty,c,d}(Q)$\\ %\hline 
%Sobolev   hyperrectangle & $\mathcal{H}_S=\mathcal{H}_S(\beta,Q)
%=\mathcal{H}(a) \;\; \text{with} \;\; a_i^2= Q i^{-(2\beta+1)}$ &$\Theta^\beta(Q)$ \\ %\hline
%parameteric hyperrectangle & $\mathcal{H}_P=\mathcal{H}_P(N_0, Q)
%=\mathcal{H}(a) \;\; \text{with} \;\; a_i^2=Q \mathrm{1}\{i\le N_0\}$ &  $C^{00}(N_0,Q)$  %\hline
%\end{tabular} }
 
For the DDM $\mathrm{P}(\cdot|X)$ defined by (\ref{ddm1}), Theorem \ref{th6} 
implies, in view of  $\Theta_{pt}\subseteq \Theta_{eb}$ and (\ref{loc<global}), 
the non-asymptotic minimax results (\ref{adapt_conf_ball_problem}) for all 
ellipsoids $\mathcal{E}(a)$ and hyperrectangles $\mathcal{H}(a)$ defined by (\ref{ellipsoids}), 
for \emph{all unknown} (non-increasing) $a$. 
Now note that, according to (\ref{sob_hyper})--(\ref{par_hyper}), the four above mentioned 
scales from \cite{Szabo&etal:2015sup} are particular examples of  ellipsoids $\mathcal{E}(a)$ 
and hyperrectangles $\mathcal{H}(a)$, with specific choices of sequence $a$. 
Hence, the minimax results (\ref{adapt_conf_ball_problem}) for all the four scales
(including the two supersmooth scales  $\mathcal{E}_A$ and $\mathcal{H}_P$)
follow for the DDM (\ref{ddm1}).  
Asymptotic version can readily be derived from the non-asymptotic ones.
Recall that the scope of the DDM $\mathrm{P}(\cdot|X)$ extends 
further than the above mentioned four scales, even beyond general families 
of ellipsoids and hyperrectangles. The local results of Theorem \ref{th6} 
deliver the minimax results  of type (\ref{adapt_conf_ball_problem}) for all 
scales for which (\ref{local_impl_global}) holds; for example, also for the scales of tail classes 
and $\ell_p$-bodies.

\end{document}